\documentclass{cmslatex}
\usepackage[paperwidth=7in, paperheight=10in,margin=.875in]{geometry}
\usepackage[backref,colorlinks,linkcolor=red,anchorcolor=green,citecolor=blue]{hyperref}

\usepackage[utf8]{inputenc}
\usepackage{amsmath,amssymb,amsmath}
\usepackage{graphicx}
\usepackage{subfigure}
\usepackage{cite}
\numberwithin{equation}{section}
\usepackage[colorinlistoftodos]{todonotes}
\usepackage{enumerate}

\newcommand{\abs}[1]{\left\vert#1\right\vert}

\newcommand{\paren}[1]{\left(#1\right)}
\newcommand{\bracket}[1]{\left[#1\right]}
\newcommand{\set}[1]{\left\{#1\right\}}

\newcommand{\E}{\mathbb{E}}
\newcommand{\R}{\mathbb{R}}

\newcommand{\Z}{\mathbb{Z}}

\newcommand{\eps}{\epsilon}

\newcommand{\NS}{N}

\renewcommand{\P}{\mathbb{P}}

\DeclareMathOperator{\Var}{Var}

\DeclareMathOperator{\bigo}{O}
\DeclareMathOperator{\littleo}{o}

\DeclareMathOperator{\osc}{osc}

\DeclareMathOperator{\Int}{Int}
\DeclareMathOperator{\Gap}{Gap}

\DeclareMathOperator{\MSDO}{MSD}

\DeclareMathOperator*{\argmax}{argmax}

\def\MSD{\MSDO_n}

\begin{document}

\title{Sampling from Rough Energy Landscapes
\thanks{Received date:}
}

%% Headers for CMS
\author{Petr Plech\'{a}\v{c}%
        \thanks{Department of Mathematical Sciences,
         University of Delaware,
         Newark, DE 19716, USA, (plechac@udel.edu)}
    \and
         Gideon Simpson%
         \thanks{Department of Mathematics, Drexel University, Philadelphia, PA (grs53@drexel.edu)}
}

\pagestyle{myheadings}
\markboth{Sampling from rough energy landscapes}{P. Plech\'a\v{c}, G. Simpson}

\maketitle

\begin{abstract}

We examine challenges to sampling from Boltzmann distributions associated with multiscale energy landscapes.  The multiscale features, or ``roughness,''  corresponds to highly oscillatory, but bounded, perturbations of a smooth landscape.  Through a combination of numerical experiments and analysis we demonstrate that the performance of Metropolis Adjusted Langevin Algorithm can be severely attenuated as the roughness increases.   In contrast, we prove that Random Walk Metropolis is insensitive to such roughness.  We also formulate two alternative sampling strategies that incorporate large scale features of the energy landscape, while resisting the impact of fine scale roughness; these also outperform Random Walk Metropolis.  Numerical experiments on these landscapes are presented that confirm our predictions.  Open questions and numerical challenges are also highlighted.

\end{abstract}

\begin{keywords}
Markov Chain Monte Carlo, random walk Metropolis, Metropolis adjusted Langevin, rough energy landscapes, multi-scale energy landscapes, mean squared displacement
\end{keywords}

\begin{AMS}
 65C05, 65C40, 60J22
\end{AMS}

\section{Introduction}
\label{s:intro}

In this work, we consider the task of sampling from a Boltzmann distribution,
\begin{equation} \label{e:boltz1} \mu(dx) = Z^{-1}e^{-\beta V(x)}dx,\quad Z =
\int e^{-\beta V(x)}dx, \quad V:\R^n \to \R, \end{equation} when $V$ is, in some
sense ``rough,'' or ``rugged.'' We are particularly interested in multiscale landscapes  of the form
\begin{equation} \label{e:V1}
V_\eps(x) = V_0(x) + V_1(x,x/\eps)\,,\;\; \quad \eps>0 \,.
\end{equation}
Here, $V_0$ is a smooth, long range, trapping potential ($V_0(x)\to \infty$ as $|x|\to \infty$) that is bounded from  below, and $V_1$
is bounded with local short wavelength features. $V_1(x,y)$ will also be assumed to be smooth. An example of such a rough landscape, and its impact on the associated
distribution, is shown in Figure~\ref{f:rough1}.   Model potentials like  \eqref{e:V1} serve as prototypes for rough and multiscale landscapes found in disordered media and soft matter, \cite{pollak2008,Hu:2018977,duncan2016noiseinduced,arous2003multiscale,owhadi2003anomalous}.  The goal of the present work is to assess how such roughness impacts the performance of well known
Markov Chain Monte Carlo (MCMC) sampling strategies like Random Walk Metropolis (RWM) and Metropolis Adjusted
Langevin (MALA).

\begin{figure}
\subfigure[]{\includegraphics[width=6.25cm]{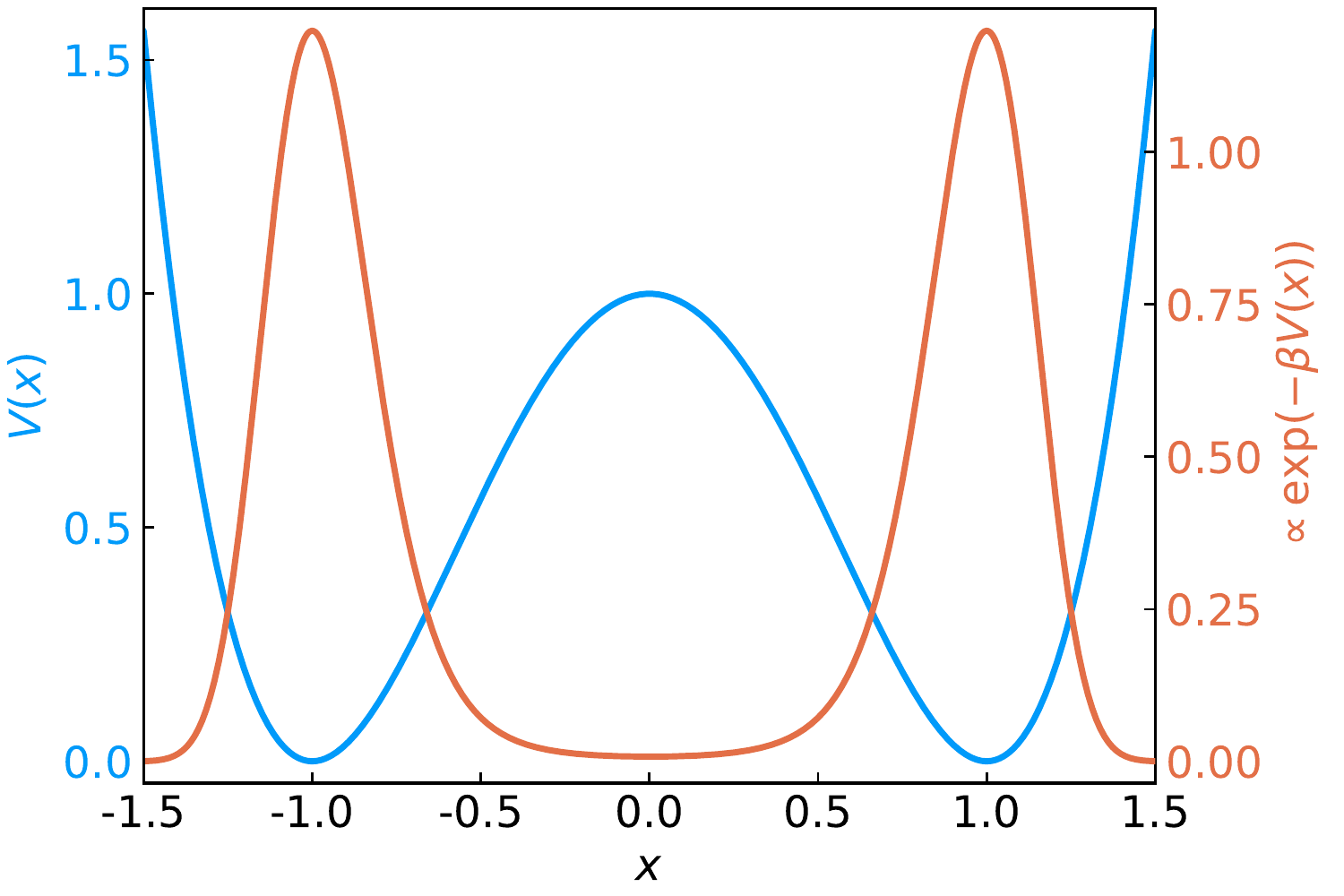}}
\subfigure[]{\includegraphics[width=6.25cm]{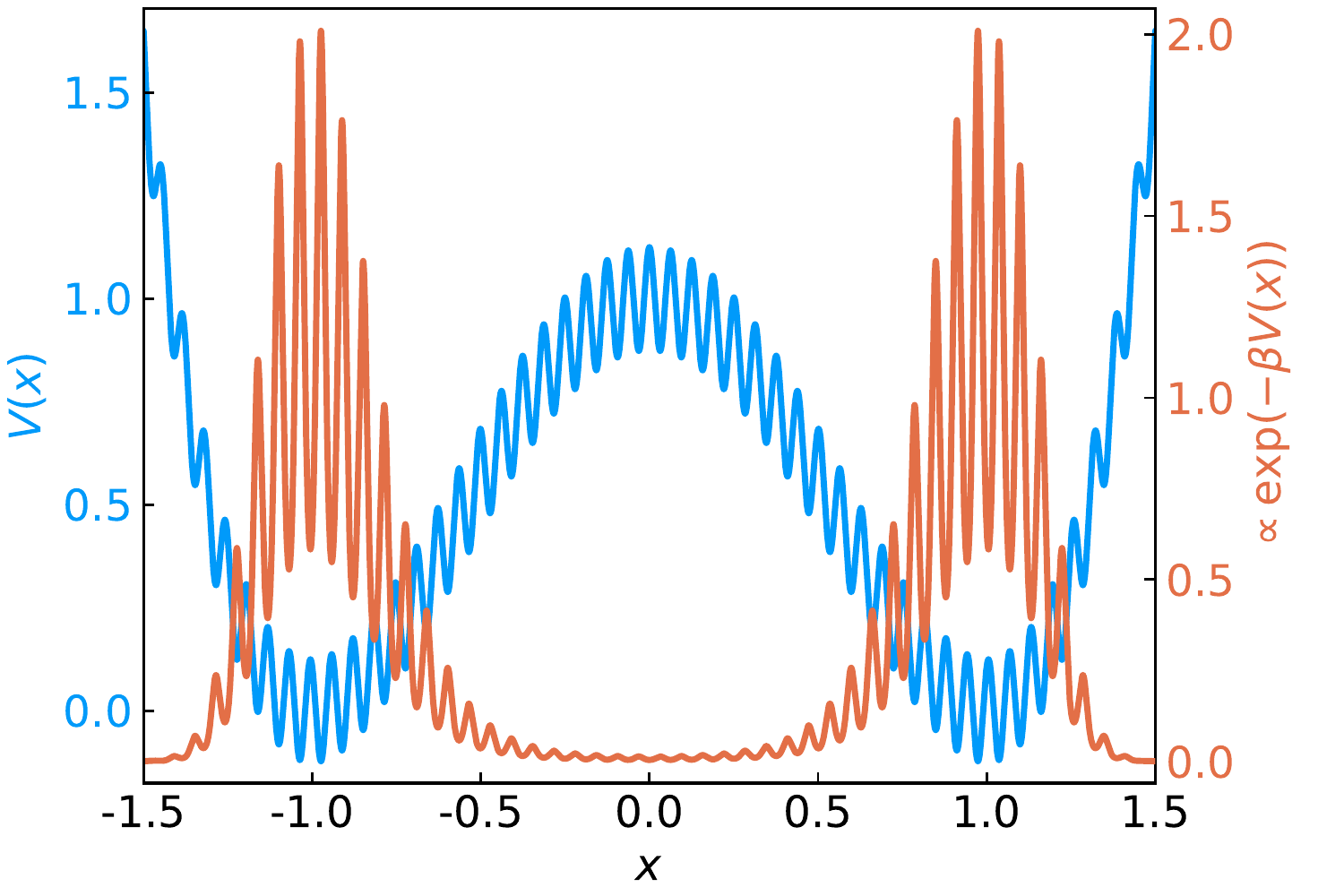}}
\caption{In (a), we see a smooth multimodal energy landscape.  In contrast,
the landscape in (b), is rough, with many internal energy barriers.  The underlying potentials are $V_0(x) = (x^2-1)^2$ and $V_\eps(x) = V_0(x) + \tfrac{1}{8}\cos(x/\eps)$ with $\eps=0.01$ and $\beta=5$ (color online).}
\label{f:rough1}
\end{figure}

We recall that RWM and MALA generate samples for $e^{-\beta V(x)}$ with proposals\footnote{Proposals will be denoted with a superscript $\mathrm{p}$.}
\begin{align}
    \label{e:RWMprop}
    \text{RWM: }X_{k+1}^{\mathrm{p}} &= X_k + \sqrt{\beta^{-1}}\sigma \xi_{k+1}, \quad \xi_{k+1}\sim N(0, I),\\
    \label{e:MALAprop}
    \text{MALA: }X_{k+1}^{\mathrm{p}} &= X_k - \tfrac{\sigma^2}{2}\nabla V(X_k)  + \sqrt{\beta^{-1}}\sigma \xi_{k+1}, \quad \xi_{k+1}\sim N(0, I)\,.
\end{align}
These proposals are then accepted or rejected with the appropriate rule to
ensure detailed balance with respect to $\mu(dx) \propto e^{-\beta V(x)}dx$.

As an example, sample MALA paths with $\sigma=1$ for the landscapes in
Figure~\ref{f:rough1} are shown in Figure~\ref{f:paths1}.  A path for
the smooth landscape exhibits better mixing than the one for the rough
landscape.  On the rough landscape the trajectory stagnates.\footnote{Stagnation corresponds to persistent rejection of proposals.}   This begs the question of whether or not $\sigma=1$ was merely a poorly chosen value -- perhaps with a different value the rough landscape would also be efficiently sampled.  Large values of $\sigma$ result in proposals with greater magnitude, but few will be accepted, and the trajectory will move slowly.  Conversely,
small values of $\sigma$ produce more readily accepted proposals, but their size limits exploration of the state space.  Consequently, an optimal choice of $\sigma$ is anticipated for each distribution.

Assuming we tune our sampler to the optimal $\sigma$ for each $V_\eps$ we seek to assess how $\eps$ impacts sampling performance.  {In this work, optimality, at a fixed value of $\epsilon$ and at a fixed dimension, will refer to maximization of some measurement of mixing, discussed below over the set of numerical parameters, {i.e.} $\sigma$ in the case of MALA.}  Ultimately, our work indicates that even at the optimal value of proposal variance,  MALA will cease to be effective as $\eps\to 0$.   In contrast, even a poorly tuned RWM sampler remains robust in the $\epsilon\to 0$ limit.

{In a ``global'' sense, robustness refers to the stability of the mixing and asymptotic variance properties of the chain as $\epsilon$ vanishes.  This can be  quantified through the spectral gap of the transition operator $\mathcal{T}$ of a given MCMC method\footnote{Some authors refer to  \eqref{e:GapT} as the interval, $\Int(\mathcal{T})$.  In \cite{Rosenthal2003}, $\Gap(\mathcal{T})$ is instead defined as $1- \sup |\lambda|$.  }
\begin{equation}
    \label{e:GapT}
    \Gap(\mathcal{T}) \equiv 1 - \sup_{\lambda\in \sigma(\mathcal{T})\setminus\{1\}} \lambda \equiv 1 - \Lambda.
\end{equation}
The gap controls both the mixing of the process and the time averaged variance constant, \cite{Rosenthal2003}.  In particular, if the gap vanishes ($\Lambda\to 1$), then the mixing collapses and the variance bound explodes.  The relationship between the gap and these quantities is given below. }

Experimentally accessible measurements of mixing can be found by looking at observables.  In particular, we consider the mean square displacement of the chain in stationarity.  If this remains positive in the $\eps\to 0$ limit, it provides a ``local'' (in the sense of a single observable) notion of robustness. In addition to being straightforward to estimate through simulations, the mean squared displacement provides an upper bound on the spectral gap.

To obtain better performance than RWM, we also formulate two related sampling strategies that incorporate information about the large scale ({i.e.} long wavelength) features of the energy landscape through $V_0$ in \eqref{e:V1}.  Indeed, our results, particularly Theorem~\ref{thm:eps} and Corollary~\ref{cor:Teps} show that for potentials that can be decomposed as in \eqref{e:V1} with $V_0$ smooth and trapping and $V_1$ rough but bounded, if the proposal of the sampling strategy is $\eps$-independent, then the performance of the method will also be $\eps$-independent.

\begin{figure}
\subfigure[Paths for Figure \ref{f:rough1}(a)]{\includegraphics[width=6.25cm]{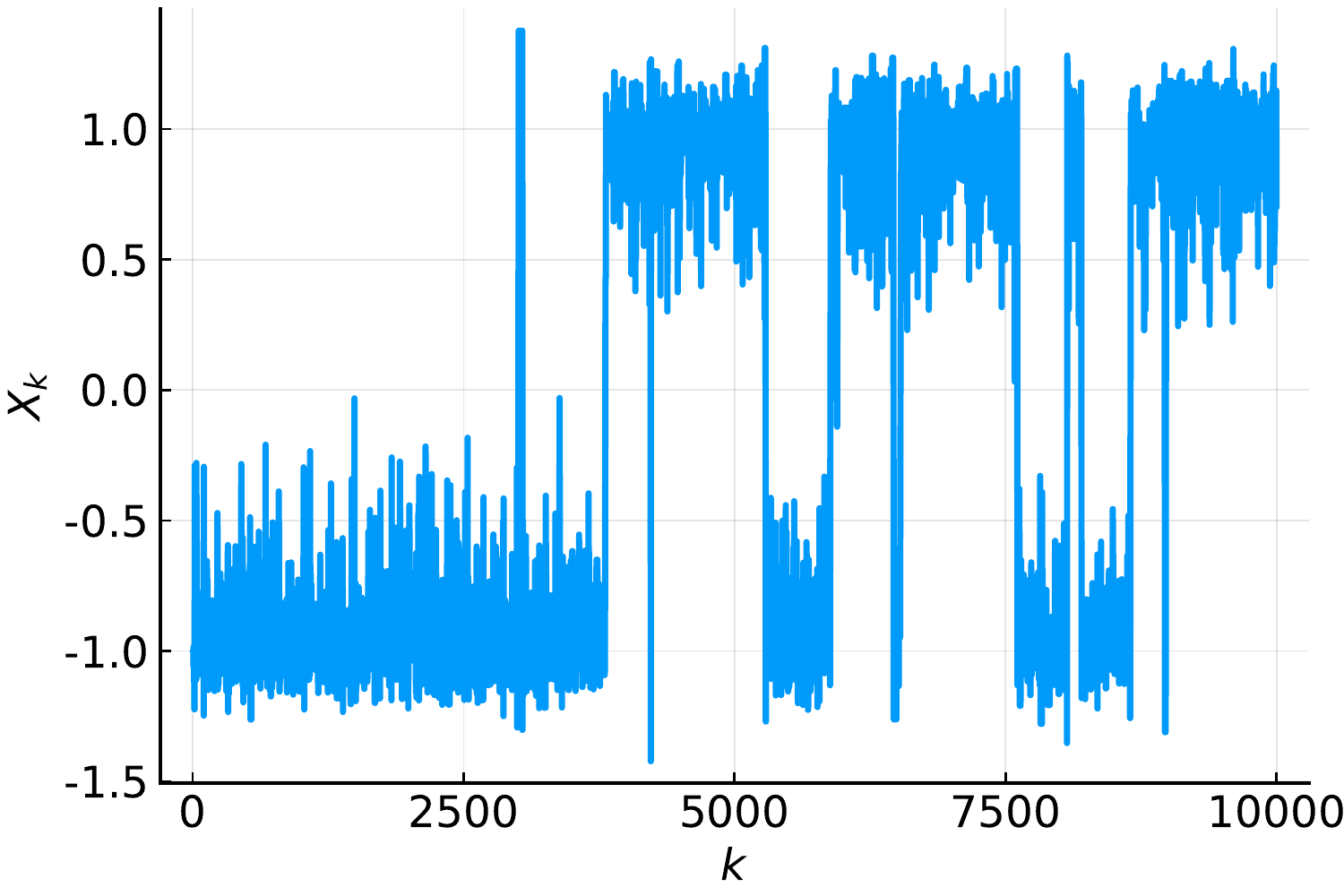}}
\subfigure[Paths for Figure \ref{f:rough1}(b)]{\includegraphics[width=6.25cm]{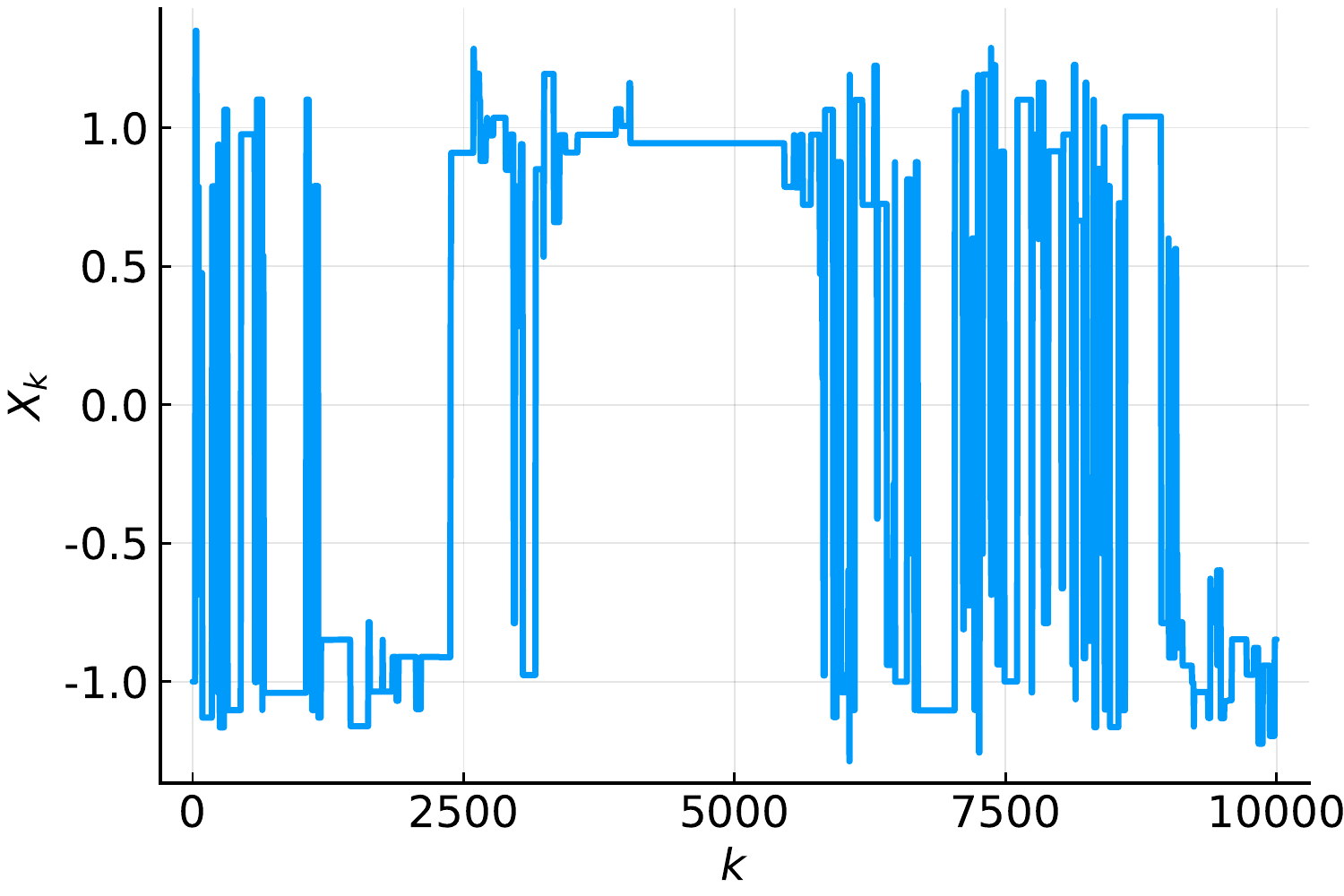}}
\caption{Sample paths corresponding to the energy landscapes in Figure~\ref{f:rough1}.  The samples were generated using MALA with $\sigma=1$ (color online).}
\label{f:paths1}
\end{figure}

\subsection{Review of prior work}

 The question of optimizing $\sigma$ to maximize performance was initially examined in \cite{roberts1997weak,roberts1998optimal} and has been subsequently  studied in other works, including
\cite{beskos2009optimal,beskos2015asymptotic,kuntz2014,kuntz2018non,kuntz2018,ottobre2016function,bourabee2018,jourdain2014,jourdain2015}.
Many of these works consider an energy landscape of the type
\begin{equation}
\label{e:Vsum}
V_n(x) = \sum_{i=1}^n v(x_i), \quad x =(x_1,\ldots,x_n)\in \R^n\,.
\end{equation}
This choice of the potentials induces {Boltzmann distributions that are products} (for brevity, we take $\beta=1$):
\begin{equation}
\label{e:musep}
\mu(dx) \propto \prod_{i=1}^n e^{-v(x_i)}dx_i.
\end{equation}
Thus, coordinates only interact through the accept/reject step of the method.  Some results for potentials other than \eqref{e:Vsum}  have also been obtained.   In \cite{beskos2009optimal}, the authors treat distributions which have a density $e^{-\Phi(x)}$ with respect to a product measure  \eqref{e:musep}.  In earlier works, \cite{roberts1997weak,roberts1998optimal,beskos2009optimal}, it was often assumed that whichever sampler is studied, the process is in stationarity.  This has been relaxed in the more recent results, \cite{kuntz2014,kuntz2018non,kuntz2018,jourdain2014,jourdain2015,beskos2015asymptotic}.

Some of our results will also specialize to the rough analog of \eqref{e:Vsum},
\begin{equation}
  \label{e:Vepssum}
  V_{\eps,n}(x) = \sum_{i=1}^n \underbrace{v_0(x_i) +
v_1(x_i,x_i/\eps)}_{\equiv v_\eps(x_i)}\,.
\end{equation}
As in the case of \eqref{e:V1}, we assume  $v_0$  is trapping while
$v_1$ is uniformly bounded.  Both $v_0(x_i)$ and $v_1(x_i,y_i)$ are assumed to be smooth functions.

In the case of \eqref{e:Vsum}, in stationarity, the performance can be measured by considering the means square displacement (MSD)
\begin{equation}
    \label{e:MSD}
    \MSD = \E_\mu[|X_{k+1}-X_k|^2]\,.
\end{equation}
The strategy of \cite{beskos2009optimal} is to use the MSD as a proxy for mixing,
and $\sigma$ is selected to maximize it in the limit of $n\to \infty$.  Indeed, this leads to the results that, as $n\to \infty$,
\begin{equation}
  \label{e:MSDdof}
\MSD/n = \ell^2 n^{-I} a(\ell;v) + \littleo(n^{-I})\,,\;\; \quad I>0\,,
\end{equation}
where $I$ and $a$ both depend on the method, but {\it not} on the dimension $n$.  The choice of $\sigma$ is then related to $\ell$ by
\begin{equation}
  \label{e:sigmaell}
  \sigma^2 = {\ell^2}{n^{-I}}
\end{equation}
The function $a$ is the mean acceptance rate in the $n\to \infty$ limit, \cite{beskos2013optimal}.

Consequently, we can maximize this measure of performance as $n\to \infty$ by solving
\begin{equation}
\label{e:ellstar}
    \ell_\star = \argmax_\ell\ell^2 a(\ell;v)\,.
\end{equation}
The optimal $\sigma_\star$ is inferred from
\eqref{e:sigmaell}, and there is an associated optimal acceptance rate,
$a(\ell_\star)$.
This optimization provides a strategy for tuning the value of $\sigma$ to achieve the optimal acceptance rate approximately 23\% for RWM.  Analogously, one tunes MALA to have a 57\% acceptance rate, \cite{roberts1997weak,roberts1998optimal},
and Hamiltonian Monte Carlo (HMC) to have a 65\% acceptance rate, \cite{bourabee2018}.

For RWM $I=1$ while it is $I=1/3$ for MALA.  The function $a$ has the explicit form
\begin{equation}
  \label{e:Accept}
  a(\ell;v) = 2 \Phi\paren{-\frac{\ell^{1/I}}{2}\sqrt{\mathcal{K}[v]}}\,,
\end{equation}
with $\mathcal{K}$ a functional involving derivatives of $v$; $\Phi$ is the standard normal ($N(0,1)$) cumulative distribution function.  A similar result holds for HMC, \cite{bourabee2018}.

There are  caveats to applying these results in practical computations, \cite{Potter:201533a}.  In particular, the results are obtained as $n \to \infty$ for distributions of form \eqref{e:musep}, and often assume the process to be in  stationarity.  Many distributions of interest will not be of this form, so target acceptance rates, like 23\% for RWM, may be inappropriate.   {However, in the recent work \cite{yang2020optimal}, it was demonstrated that for a more general class of distributions than \eqref{e:musep}, 23\% remains the optimal acceptance rate for RWM.}

{We mention these results because they motivate certain aspects of this work, such as the examination of product measures and the examination of $\MSD$ as a proxy for mixing. However, we emphasize that this work is focused on problems at fixed $n$, letting $\eps\to 0$.  The preceding results can provide guidance in the $\eps\to 0$ limit, but they would necessitate first taking $n\to \infty$.}

We also highlight the recent work in \cite{beskos2015asymptotic} which studies ``ridged'' densities associated with potentials of the form
\begin{equation}
  \label{e:ridgedV}
V_\eps(x) = V_0(x_1) + V_1(x_1, x_2/\eps)\,,\;\; \quad x = (x_1,x_2)\in \R^{n = n_1+n_2}\,.
\end{equation}
Here, roughness is only present in a subset of the degrees of freedom ($x_2$).  Examining RWM for such a problem, the authors are able to
derive a limiting diffusion from which they can find an optimal step size.  This
limiting diffusion has a state dependent diffusion coefficient. Both the drift and
diffusion coefficients are nontrivial requiring averaging against the rough
degrees of freedom.

{Related to the work on ridged densities, and the present work, is  \cite{Livingstone1908}.  In this work, the authors consider the case that one of the coordinates is scaled differently than the others.  This would correspond to a potential like
\begin{equation}
\label{e:Vmismatch}
    V_\eps(x) = V(x_1/\eps, x_2, \ldots, x_n), \quad x \in \R^n\,.
\end{equation}
In \cite{Livingstone1908} the authors also look for algorithms that are less sensitive to length scale variations in the gradients.  They obtain results on MALA and HMC showing poor behavior in the $\eps\to 0$ limit.  An important tool that they use in their analysis is the Dirichlet form and its relationship to the spectral gap;  we, too, make use of that approach.}

Another relevant work is \cite{durmus2018efficient}.  There, the authors
sought to perform gradient based sampling on non-differentiable energy landscapes and proposed using a
Moreau-Yosida regularization.  This approach is related to one of the mechanisms that we
propose in order to overcome roughness, though our potentials are smooth,
but highly oscillatory.

\subsection{Measures of performance and notions of robustness.}

{
As mentioned, the key metrics that we use to assess performance are the spectral gap, \eqref{e:GapT}, along with the MSD, \eqref{e:MSD}.  We recall the relationships amongst these quantities,  %these are as follows,
\cite{Rosenthal2003, Livingstone1908}.  First, the relaxation to the equilibrium in the total variation (TV)  and the time averaged variance constant (TAVC) are controlled by  $\Lambda$ and $\Gap(\mathcal{T})$:
\begin{align}
    \label{e:TVmixing}
    \lim_{k\to\infty}\frac{1}{k}\log\|\rho_0 T^k - \mu\|_{\rm TV} &\leq \log \Lambda= \log(1-\Gap(\mathcal{T}))\,,\\
    \label{e:TAVC}
    \lim_{k\to \infty}\frac{1}{k}\Var_\mu\paren{\sum_{j=0}^{k-1} f(X_j)}&\leq \frac{1+\Lambda}{1-\Lambda}\E_{\mu}[f(X)^2] =  \frac{2-\Gap(\mathcal{T})}{\Gap(\mathcal{T})}\E_{\mu}[f(X)^2]\,,
\end{align}
for an initial distribution with $d\rho_0/d\mu\in L^2(\mu)$ and any function $f \in L^2(\mu)$.   Consequently, if $\Gap(\mathcal{T})\to 0$, mixing ceases and there is no upper bound on the TAVC.  On the other hand, if the gap remains positive then there is a priori bound on the TAVC.  Due to the inequality in \eqref{e:TVmixing} a positive spectral gap does {\it not} imply  a lower bound on mixing.}

{A key relationship is between spectral gap and the Dirichlet form
\begin{equation}
\label{e:Dirichlet}
    \inf_{f\in L_{0,1}^2(\mu)} \tfrac{1}{2}\E_{\mu}[|f(X_{k+1})-f(X_k)|^2] = 1 - \sup_{\lambda \in \sigma(\mathcal{T})\setminus\{1\}} \lambda = \Gap(\mathcal{T})\,,
\end{equation}
where $L_{0,1}^2(\mu)$ is the subset of mean zero, unit variance functions in $L^2(\mu)$.  An elementary computation reveals that this is equivalent to
\begin{equation}
\label{e:Dirichlet2}
    \Gap(\mathcal{T}) = \inf_{f\in L^2(\mu)} \frac{\E_{\mu}[|f(X_{k+1})-f(X_k)|^2]}{2\Var_\mu(f(X))}
\end{equation}
Consequently, for any square integrable function, if we can estimate the term on the right-hand side of \eqref{e:Dirichlet2} , we have obtained an upper bound on the spectral gap.  This allows us to use \eqref{e:MSD} as a proxy for the gap;   if we find, numerically, that $\MSD\to 0$ as $\eps\to 0$, that is strong empirical evidence that $\Gap(\mathcal{T}_\eps)\to 0$ too.  Again, due to the infimum  in the Dirichlet form, positivity of the one step jumping distance of any observable, including $\MSD$, does not imply positivity of the gap.  }

{We thus formalize  two notions  of robustness.  Our robustness criteria for a method with a transition operator $\mathcal{T}_\eps$ for sampling the Boltzmann distribution of potential $V_\eps$ are
\begin{align}
\label{e:Globalrobustness}
\text{Global Robustness Criterion:} & \quad  \liminf_{\eps\to 0} \Gap(\mathcal{T}_\eps)>0\,,\\
\label{e:MSDrobustness}
    \text{Local Robustness Criterion:}&\quad \liminf_{\eps\to 0} \frac{\MSD}{2\Var_\mu(X)} >0\,.
\end{align}
Analogous robustness conditions can be constructed for other observables.  Failure to be locally robust immediately implies the method cannot be globally robust.  Likewise, a globally robust method automatically implies local robustness. However, there may be distributions and methods for which, on a particular observable, the method is locally robust, but for which it fails to be globally robust. }

{In \eqref{e:Globalrobustness} and \eqref{e:MSDrobustness}, no mention is made of the choice of the numerical parameter $\sigma$.  Some methods may fail to be robust if $\sigma$ is not chosen carefully.  Thus, we introduce two conditional notions of robustness that depend on $\sigma$
\begin{align}
\label{e:Globalrobustness2}
\text{Global Robustness Criterion with Optimal $\sigma$:} & \quad  \liminf_{\eps\to 0} \sup_{\sigma }\Gap(\mathcal{T}_\eps(\sigma))>0\,,\\
\label{e:MSDrobustness2}
    \text{Local Robustness Criterion with Optimal $\sigma$:}&\quad \liminf_{\eps\to 0}  \sup_{\sigma}\frac{\MSD(\sigma)}{2\Var_\mu(X)} >0\,.
\end{align}
These notions generalize to methods with additional numerical parameters.  Obviously, if a method fails to satisfy \eqref{e:Globalrobustness2}, it will fail to satisfy \eqref{e:Globalrobustness}.  Conversely, if the method satisfies \eqref{e:Globalrobustness}, then it will also satisfy \eqref{e:Globalrobustness2}.  Analogous relationships can be formulated for local robustness.}

\subsection{Results on performance in the presence of roughness.}

Our key observations and results are:
\begin{enumerate}[(i)]
    \item {At a fixed dimension $n$, the performance of RWM for potentials of the form \eqref{e:V1}, is globally robust in the sense of \eqref{e:Globalrobustness}.    This is a consequence of Corollary~\ref{cor:Teps}.  Indeed, any method that uses $\eps$-independent proposals will  similarly satisfy \eqref{e:Globalrobustness}. Proposals with a sufficiently mild $\eps$-dependence will also be robust; see Corollary~\ref{cor:Teps2}.  The methods need not be optimally tuned for \eqref{e:Globalrobustness} to hold.}

    \item {A rigorous result is established that, subject to certain assumptions,  MALA fails to be globally robust when $\sigma$ is inadequately scaled for potentials like \eqref{e:Vepssum}.  This is the content of
    Theorem~\ref{t:malascaling2}.  Specifically, when $\sigma$ is too large relative to $\eps$, the spectral gap will close.  }

    \item Numerical experiments and an explicit example indicate that for $n=1$, the optimal scaling of MALA is  $\sigma\propto \sqrt{\eps}$, so that $\MSDO_1 \propto \eps$.  In contrast, for $n$ sufficiently large, the empirical optimal scaling is $\sigma\propto \eps$, so that $\MSD/n \propto \eps^2$.

    The experiments also indicate that MALA is not locally robust even at an empirically determined optimal $\sigma$, \eqref{e:MSDrobustness2}.  Thus, the method suffers generically  in the $\eps\to 0$.

    \item We formulate two alternative methods for potentials of type \eqref{e:V1} that use large scale information contained in $V_0$.  The first method, which we call {\it Modified MALA}, uses the proposals  (at $\beta=1$)
    \begin{equation}
    \label{e:ModMALAV0prop}
        X^{\mathrm{p}}_{k+1} = X_k  - \tfrac{\sigma^2}{2}\nabla V_0(X_k) + \sigma \xi_{k+1}\,.
    \end{equation}
    This method fits in the class of sampling methods studied in \cite{bourabee2018} where the authors rigorously established that for the path-wise accuracy of Metropolized integrators it is sufficient to accurately simulate the diffusion term, which \eqref{e:ModMALAV0prop} does.

    The second method, we call an {\it Independence Sampler}, uses proposals (at $\beta =1$)
    \begin{equation}
    \label{e:MCMCV0prop}
         X^{\mathrm{p}}_{k+1} = Y_{k+1}\sim e^{-V_0(x)}dx.
    \end{equation}
    The $Y_{k}$ samples are assumed to be independent, generated by an auxiliary process. Both of these methods are also insensitive to the roughness and
    outperform RWM.

    As both of these methods use $\eps$-independent
    proposals Corollary~\ref{cor:Teps} allows us to conclude they will
    also be globally robust in the $\eps\to 0$  limit.

    \item For potentials that do not admit an obvious decomposition like \eqref{e:V1} we propose using local entropy approximation, \cite{chaudhari2018deep,chaudhari2016entropysgd} to extract the large scale information needed for either the Modified MALA method  or the Independence Sampler.

\end{enumerate}

\medskip
In Section~\ref{s:asympt} we identify a bound on the performance of RWM and other algorithms, and we consider the asymptotic behavior of MALA as $\eps\to 0$.  In Section~\ref{s:algorithms} we present alternative methods that are also robust to $\eps \to 0$.  Numerical experiments are presented in Section~\ref{s:numerics}, and we conclude with a discussion in Section~\ref{s:disc}.

\section*{Acknowledgements}
 The work of PP was partially supported by  DARPA project W911NF-15-2-0122 and GS was supported by US National Science Foundation Grant DMS-1818716. The authors thank UCLA IPAM for hosting them during the beginning of this project.  The authors also thank M. Luskin, S. Osher, J. Mattingly, and N. Bou-Rabee for helpful discussions.

\section{Bounds on performance.}
\label{s:asympt}

In this section\ we present bounds on performance with respect to the robustness criteria.  For any MCMC method, let $q(x\to y)$ denote the associated proposal kernel, and define
\begin{equation}
\label{e:R}
    R(x,y) = V(x) - V(y) + \log \frac{q(y\to x)}{q(x\to y)}\,.
\end{equation}
Consequently, the proposal $X_n^{\mathrm{p}}$ is accepted with probability $F(R(X_n, X_n^{\mathrm{p}}))$.   The two forms of $F$ that we consider here are
\begin{subequations}
\label{e:Ffunc}
\begin{align}
\label{e:metrop}
  \text{Metropolis:} &\quad  F(r) = 1 \wedge e^{r}\,,\\
\label{e:barker}
  \text{Barker:} &\quad  F(r) = (1+e^{-r})^{-1}\,.
\end{align}
\end{subequations}

\subsection{Roughness independent bounds.}
\label{s:lower}

For potentials of type \eqref{e:V1} one can obtain $\eps$-independent upper and lower bounds on a variety of quantities.  Indeed, by the boundedness assumption in the introduction, we are assured that
\begin{equation}
\label{e:oscV}
    \osc V_1  = \sup_x V_1(x,x/\eps) - \inf_x V_1(x,x/\eps)< \infty\,.
\end{equation}
Our main results in this subsection are the following:{
\begin{thm}
\label{thm:eps}
Let $V_\eps$ be a potential of type \eqref{e:V1} and assume $V_1$ has uniform in $\eps$  bounded oscillation in the sense of \eqref{e:oscV}.  If the sampling strategy uses proposals $q(x\to y)$ that are $\eps$-independent, then for any $f$, the one step jumping distance is bounded by $\eps$ independent constants:
\begin{equation}
\label{e:processbound}
\begin{split}
    e^{-2\osc V_1}\E_{\mu_0}[|f(X_{k+1})-f(X_k)|^2]&\leq \E_{\mu}[|f(X_{k+1})-f(X_k)|^2]\\
    &\quad \leq e^{2\osc V_1}\E_{\mu_0}[|f(X_{k+1})-f(X_k)|^2]\,.
    \end{split}
\end{equation}
In \eqref{e:processbound}, $\mu_0(dx)\propto e^{-V_0(x)}dx$.
\end{thm}}

{A corollary to this result provides $\eps$ independent bounds on the spectral gap:
\begin{cor}
\label{cor:Teps}
Under the same assumptions as Theorem~\ref{thm:eps}
\begin{equation}
    e^{-3 \osc V_1} \Gap(\mathcal{T}_0) \leq \Gap(\mathcal{T}_\eps)\leq e^{3 \osc V_1} \Gap(\mathcal{T}_0)
\end{equation}
where $\mathcal{T}_0$ is the transition operator of the method with proposals generated by $q$ sampling $\mu_0$.
\end{cor}}

{\begin{rem}
We emphasize that these results, and the results in the rest of Section~\ref{s:lower}, are all at a fixed dimension $n$.  Constants, like $e^{\pm 2\osc V_1}$ in \eqref{e:processbound}, may depend in an unfavorable way on $n$.  Additionally,  Theorem~\ref{thm:eps} and Corollary~\ref{cor:Teps} hold independently of the choice of any numerical parameters, like $\sigma$ in RWM.
\end{rem}}

To prove Theorem~\ref{thm:eps} and its corollary, we first prove the following bounds on the distribution.
\begin{lem}
\label{lem:eps1}
Let $V_\eps$ be a potential of type \eqref{e:V1}, and assume $V_1$ has uniform in $\eps$  bounded oscillation in the sense of \eqref{e:oscV}.  Then
\begin{equation}
\label{e:mubounds}
e^{-\osc V_1}\mu_0(dx)\leq    \mu(dx)\leq e^{\osc V_1}\mu_0(dx)
\end{equation}
\end{lem}
\begin{proof}
The proof of this follows from direct estimates on the densities.  First,
\begin{equation*}
  Z^{-1}e^{- V(x)}\geq Z^{-1} e^{-V_0(x)} e^{-\sup V_1(x,x/\eps)}
\end{equation*}
while
\begin{equation*}
  Z = \int e^{- V(x)} \leq \underbrace{\int e^{- V_0(x)}dx}_{\equiv Z_0}e^{-\inf V_1(x,x/\eps)}
\end{equation*}
\end{proof}

As an immediate consequence we have bounds for the mean.
\begin{cor}
\label{cor:eps}
Let $V_\eps$ satisfy the same assumptions as in Lemma \ref{lem:eps1}.  Then for any non-negative observable, $f$, which may depend on a small parameter $\eps$,
\begin{equation}
\label{e:obseps}
    e^{-\osc V_1}\E_{\mu_0}[f(X)]\leq \E_{\mu}[f(X)]\leq e^{\osc V_1}\E_{\mu_0}[f(X)].
\end{equation}
\end{cor}
Similarly we obtain bounds on the variance.
{\begin{lem}
\label{lem:vareps}
Let $V_\eps$ satisfy the same assumptions as in Lemma \ref{lem:eps1}.  Then $L^2(\mu)$ and $L^2(\mu_0)$ are equivalent as sets, and
\begin{equation}
    \label{e:vareps}
    e^{-\osc V_1}\Var_{\mu_0}(f(X))\leq \Var_{\mu}(f(X)) \leq e^{\osc V_1}\Var_{\mu_0}(f(X))
\end{equation}
\end{lem}
\begin{proof}
For $f \in L^2(\mu)$, Lemma \ref{lem:eps1} ensures
\begin{equation*}
    e^{-\osc V_1} \int f(x)^2 \mu_0(dx) \leq \int f(x)^2 \mu(dx)<\infty
\end{equation*}
Hence, $f \in L^2(\mu_0)$ too and $L^2(\mu)\subset L^2(\mu_0)$.  A similar computation shows the reverse inclusion.
Next, for $f\in L^2(\mu)$, let $f_c(x) = f(x) - \E_{\mu_0}[f(X)]$.  Then, using Lemma \ref{lem:eps1} again,
\begin{equation*}
\begin{split}
    \Var_{\mu}(f(X)) &= \Var_{\mu}(f_c(X))\\
    &= \E_{\mu}[f_c(X)^2] - \E_{\mu}[f_c(X)]^2 \leq \E_{\mu}[f_c(X)^2] \\
    & \leq e^{\osc V_1} \E_{\mu_0}[f_c(X)^2] = e^{\osc V_1} \Var_{\mu_0}[f(X)]
\end{split}
\end{equation*}
Reversing the roles of $\mu$ and $\mu_0$ establishes the analogous lower bound in \eqref{e:vareps}.
\end{proof}}

{\begin{rem}
\label{rem:variance}
For potentials satisfying the assumptions of the preceding results, it will often be sufficient to examine $\E_\mu[|f(X_{k+1})-f(X_k)|^2]$ since a prior bound on the variance is provided by Lemma \ref{lem:vareps}.
\end{rem}}

\begin{proof}{{\bf Proof of Theorem~\ref{thm:eps}}.}

Given potential $V = V_0 + V_1$ of the form \eqref{e:V1} satisfying the assumptions:
\begin{equation}
\label{e:Rsplit}
\begin{split}
    R(x,y) & =\underbrace{(V(x)-V_0(x)}_{V_1(x,x/\eps)} - \underbrace{(V(y)-V_0(y) )}_{V_1(y,y/\eps)}  + \underbrace{\log \paren{\frac{e^{-V_0(y)}q(y\to x)}{e^{-V_0(x)}q(x\to y)} }}_{R_0(x,y)}
    \end{split}
\end{equation}
Since $\osc V_1$ is bounded,
\begin{equation}
R_0(x,y) +\osc V_1 \geq R(x,y) \geq  R_0(x,y)- \osc V_1.
\end{equation}
Consequently, for either choice of \eqref{e:Ffunc},
\begin{equation}
  \label{e:Rbounds}
  e^{\osc V_1} F(R_0(x,y))\geq F(R(x,y))\geq e^{-\osc V_1} F(R_0(x,y)).
\end{equation}

Since the proposal, $q(x\to y)$, is also assumed $\eps$-independent, then, for any $X_0$,
\begin{subequations}
\label{e:MSDbounds}
\begin{align}
  \begin{split}
\E_{\mu}[|f(X_{k+1)}-f(X_k)|^2]&=\E_{\mu}[|f(X_{k+1}^{\mathrm{p}})-f(X_k)|^2F(R(X_{k}, X_{k+1}^{\mathrm{p}}))]\\
& \geq e^{-\osc V_1}\E_{\mu}[|f(X_{k+1}^{\mathrm{p}})-f(X_k)|^2F({R_0}(X_{k}, X_{k+1}^{\mathrm{p}}))]
\end{split}\\
\begin{split}
\E_{\mu}[|f(X_{k+1)}-f(X_k)|^2]& \leq e^{\osc V_1}\E_{\mu}[|f(X_{k+1}^{\mathrm{p}})-f(X_k)|^2F({R_0}(X_{k}, X_{k+1}^{\mathrm{p}}))]
\end{split}
\end{align}
\end{subequations}
In the preceding upper and lower bounds on $\E[|X_{1}-X_0|^2]$, no $\eps$ is present. If we now apply Lemma \ref{lem:eps1} to the upper and lower bounds in \eqref{e:MSDbounds}, we obtain \eqref{e:processbound}.
\end{proof}
Corollary~\ref{cor:Teps} follows from the theorem and Lemma \ref{lem:vareps}.
\begin{proof}{{\bf Proof of Corollary~\ref{cor:Teps}}.}

{For any non constant $f$,
\begin{equation*}
    \frac{\E_{\mu}[|f(X_{k+1})-f(X_k)|^2]}{2 \Var_{\mu}(f(X))}\leq e^{3\osc V_1}\frac{\E_{\mu_0}[|f(X_{k+1})-f(X_k)|^2]}{2 \Var_{\mu_0}(f(X))}.
\end{equation*}
Taking the infinum over $L^2(\mu)$,
\begin{equation*}
    \Gap(\mathcal{T}_\eps) \leq e^{3 \osc V_1} \inf_{f \in L^2(\mu)}\frac{\E_{\mu_0}[|f(X_{k+1})-f(X_k)|^2]}{2 \Var_{\mu_0}(f(X))} = e^{3 \osc V_1}\Gap(\mathcal{T}_0).
\end{equation*}
This last equality is due to $L^2(\mu)$ and $L^2(\mu_0)$ being equivalent as sets.  An analogous computation establishes the lower bound.}
\end{proof}

Theorem~\ref{thm:eps} and Corollary~\ref{cor:Teps} immediately apply to RWM, as it has a proposal independent of $\eps$.  Consequently, for RWM,
\begin{equation*}
    \MSD=\E_\mu[|X_{k+1}-X_k|^2]\sim \sigma^2.
\end{equation*}
In contrast, as MALA proposals include gradients of the potential, for potentials of the form \eqref{e:V1}, these results will not apply.

The results can be modified to allow for proposals that have some $\eps$ dependence:
\begin{thm}
\label{thm:eps2}
Let $V_\eps$ be a potential of the type \eqref{e:V1}, and assume $V_1$ has   bounded oscillation, uniformly in $\eps$, in the sense of \eqref{e:oscV}.  Assume the sampling strategy's proposal kernel, $q(x\to y)$, satisfies the inequality
\begin{equation}
    C q_0(x\to y) \leq q(x\to y) \leq D q_0(x\to y)\,,
\end{equation}
where $q_0(x\to y)$ is an $\eps$-independent proposal kernel.  Then the performance, as measured by $\MSD$, is $\eps$-independent:
\begin{equation}
\label{e:processbound2}
\begin{split}
    e^{-2\osc V_1}\frac{C}{D}\E_{\mu_0}[|f(X_{k+1})-f(X_k)|^2]&\leq\E_{\mu}[|f(X_{k+1})-f(X_k)|^2]\\
    &\quad \leq e^{2\osc V_1}\frac{D}{C}\E_{\mu_0}[|f(X_{k+1})-f(X_k)|^2]\,.
    \end{split}
\end{equation}
\end{thm}
\begin{proof}
As in the proof of Theorem \ref{thm:eps}, we begin by writing
\begin{equation*}
\begin{split}
    R(x,y)& = (V(x) - V_0(x)) - (V(y) - V_0(y)) +\log\paren{\frac{e^{-V_0(y)} q(y\to x)}{e^{-V_0(x)}q(x\to y}}\\
     & \leq  V_1(x) - V_1(y) +\log\paren{\frac{D}{C}\frac{e^{-V_0(y)} q_0(y\to x)}{e^{-V_0(x)}q_0(x\to y}}\\
     &\leq \osc V_1 + \log\frac{D}{{C}} + R_0(x,y)\,.
    \end{split}
\end{equation*}
Analogously,
\begin{equation*}
    R(x,y)\geq - \osc V_1+\log\frac{C}{D} + R_0(x,y)\,.
\end{equation*}
Consequently, for a function $f$,
\begin{align*}
\E_\mu[|f(X_{k+1})-f(X_k)|^2]  &\leq e^{\osc V_1} \frac{D}{C} \E_{\mu}[|f(X_{k+1}^{\rm p}) - f(X_k)|^2 F(R_0(X_k, X_{k+1}^{\rm p}))]\\
\E_\mu[|f(X_{k+1}-f(X_k)|^2]  &\geq e^{-\osc V_1} \frac{C}{D} \E_\mu[|f(X_{k+1}^{\rm p} - f(X_k)|^2 F(R_0(X_k, X_{k+1}^{\rm p}))]\,.
\end{align*}
Using Lemma \ref{lem:eps1} yields the result.
\end{proof}

{An analog of Corollary \ref{cor:Teps} can then be established, which we present without proof:
\begin{cor}
\label{cor:Teps2}
Under the same assumptions as Theorem \ref{thm:eps2},
\begin{equation}
e^{-3 \osc V_1}\frac{C}{D} \Gap(\mathcal{T}_0) \leq \Gap(\mathcal{T}_\eps)\leq e^{3 \osc V_1}\frac{D}{C} \Gap(\mathcal{T}_0)\,.
\end{equation}
\end{cor}}

A method where Theorem~\ref{thm:eps2} would apply is the tamed (or truncated) MALA, \cite{bourabee2010pathwise,roberts1996geometric,hutzenthaler2012strong}.  Given a parameter $\delta>0$, one form of the tamed MALA (at $\beta=1$) is
\begin{equation}
\label{e:tamedMALA}
    X_{k+1}^{\rm p} = X_k - \frac{\sigma^2}{2}\frac{\nabla V(X_k)}{1 \vee (\delta |\nabla V(X_k)|)} + \sigma \xi_{k+1}, \quad \xi_{k+1} \sim N(0,I)\,.
\end{equation}
This has the effect of ensuring that the gradient term is never larger
than $1/\delta$, mitigating stiff regimes of the state space.

\begin{cor}
The performance, as measured by the spectral gap, of the tamed MALA on potentials of the type \eqref{e:V1}, is insensitive to $\eps$.
\end{cor}
\begin{proof}
For the tamed MALA as implemented in \eqref{e:tamedMALA}, let
\begin{equation*}
     f_\delta(X_k) = - \frac{\nabla V(X_k)}{1 \vee (\delta |\nabla V(X_k)|)}\,.
\end{equation*}
Then its proposal density is
\begin{equation}
\label{e:tamedq}
    q(x\to y) \propto \exp\set{-\tfrac{1}{2\sigma^2}|y-x-\tfrac{\sigma^2}{2}f_\delta(x)|^2}\,.
\end{equation}
Since $|f_\delta(x)|\leq \delta^{-1}$,
\begin{equation}
    2|y-x|^2 - \tfrac{\sigma^4}{2\delta^2}\leq |y-x-\tfrac{\sigma^2}{2}f_\delta(x)|^2\leq 2|y-x|^2 + \tfrac{\sigma^4}{2\delta^2}\,.
\end{equation}
Setting
\begin{equation}
q_0(x\to y) \propto \exp\set{-\tfrac{1}{\sigma^2}|y-x|^2}\,,
\end{equation}
we have that the tamed MALA proposal kernel satisfies
\begin{equation}
    e^{-\frac{\sigma^4}{2\delta^2}}q_0(x\to y)\leq q(x\to y) \leq e^{\frac{\sigma^4}{2\delta^2}}q_0(x\to y)\,.
\end{equation}
Thus, Corollary~\ref{cor:Teps2} applies.
\end{proof}

\subsection{Optimal Proposal Variance for MALA.}
\label{s:mala}

{Next, we perform a formal calculation on the performance of MALA with respect to proposal variance $\sigma$ and roughness $\epsilon$.  Though this does not lead to any immediate conclusions, it may guide future analysis.}

We define $M(\sigma)$ to be the $\MSD$ obtained with $\sigma$, assuming
stationarity, $x \sim \mu$, with the Barker rule, \eqref{e:barker}.  Thus, the
optimal value of $\sigma$ can be obtained by solving $M'(\sigma)=0$. Proceeding with this strategy we have
\begin{equation}
\label{e:Msig}
    M(\sigma) = \int |x-y|^2 F(R(x,y; \sigma)) g(y;x,\sigma)dy \mu(dx)\,,
\end{equation}
where $g$ is the Gaussian density of the normal distribution $N( x - \frac{\sigma^2}{2}\nabla V(x), \sigma^2 I)$ and $\mu(dx)$ is the Boltzmann distribution.

\begin{proposition}
\label{p:Msig}
At a critical point of \eqref{e:Msig}, where $M'(\sigma)=0$,
\begin{equation}
    \label{e:Msigopt2}
    \begin{split}
    0 & = \frac{\sigma^4}{2}\underbrace{\E\bracket{|x-y|^2 F^2|\nabla V(x)|^2}}_{\equiv A} + \sigma^2 n \underbrace{\E[|x-y|^2 F]}_{= M}-\underbrace{\E[|x-y|^4 F]}_{\equiv C}\,,
    \end{split}
\end{equation}
or
\begin{equation}
\label{e:Mquad1}
    \sigma^2 =2 \set{\frac{n M}{C} + \sqrt{\paren{\frac{n M}{C}}^2 + \frac{2A}{C}}}^{-1}\,.
\end{equation}

\end{proposition}
While this is a formal calculation, it does not require any particular assumptions on the potential beyond the associated distribution being normalizable and having the necessary moments.
\begin{proof}
See Appendix~\ref{s:derviative} for the derivation of \eqref{e:Msigopt2}.  The expression \eqref{e:Mquad1} follows by solving for $\sigma^2$.
\end{proof}

\subsection{Insights from the high dimensional limit.}
\label{s:highd}

We recall from the introduction that for product distributions such as \eqref{e:musep}, as $n\to \infty$, the $\MSD$ per degree of freedom is given by \eqref{e:MSDdof}.  {Here, we consider the $n\to\infty$ limit at fixed $\eps$, and then examine the impact of the roughness through $\eps$.  This corresponds to a consideration of the order $\lim_{\eps\to 0}\lim_{n\to \infty}$.  Elsewhere in this manuscript, we focus on fixed $n$ while letting $\eps \to 0$.}

The optimal $\ell$ for \eqref{e:MSDdof} can be found using \eqref{e:ellstar} and  \eqref{e:Accept}.  From \eqref{e:sigmaell} we can then find the optimal choice of $\sigma$ and obtain the scaling of the $\MSD/n$ in \eqref{e:MSDdof}.  After changing coordinates to $\lambda= - \ell^{1/I}\sqrt{\mathcal{K}[v]}/2$, and optimizing in this transformed coordinate, one can deduce that the optimal $\ell^2 \sim \mathcal{K}^{-I}$.  Consequently, for potentials like \eqref{e:Vepssum}, the optimal scaling, as $n \to \infty$, is $\sigma^2\sim \paren{{n\mathcal{K}}}^{-I}$ and
\begin{itemize}
    \item[(i)] For RWM, $I=1$ and $\mathcal{K}[v_\eps]= \E_\mu[(v_\eps')^2]$, so that $\sigma^2 \sim \eps^2/n$;
    \item[(ii)] For MALA, $I=1/3$ and $\mathcal{K}[v_\eps] = \E_\mu[5 (v_\eps''')^2 + 3 (v_\eps'')^3]/48\sim \eps^{-6}$ so that $\sigma^2\sim \eps^{2}/n^{1/3}$.
\end{itemize}
In this limit, both RWM and MALA will fail to be locally robust at optimal $\sigma$ (in the sense of \eqref{e:MSDrobustness2}) as $\eps\to 0$. Thus, they will also suffer when a suboptimal value is selected for any observable.  This does not contradict the analysis of Section~\ref{s:lower}.  There, in the product case, our lower bound would have the prefactor $e^{-n \osc v_1}$, which vanishes more rapidly than the $\eps^2/n$ we have here.

\subsection{Scaling in dimension one.} \label{s:oned}

We consider the question of sampling from the distribution associated with $V(x) = \tfrac{1}{2}\eps^{-1}x^2$ in $n=1$, using MALA.  While this may seem odd, as this corresponds to a Gaussian distribution, it reveals a distinct scaling in low dimensions.  It can be interpreted as the necessary scaling to efficiently sample individual modes of the highly multi-modal landscape in Figure~\ref{f:rough1}(b), or alternatively, as the analog of the stiff differential equation problem $\dot{x} = -\eps^{-1} x$,  $0<\eps\ll 1$.

In this case,
\begin{equation}
    X^{\mathrm{p}}_{k+1} = X_{k} -\tfrac{\sigma^2}{2} \eps^{-1} X_k + \sigma\xi_{k+1}\,,
\end{equation}
and the term in the accept/reject rule is
\begin{equation}
\label{e:scalarR}
    R(x,y) = \frac{\sigma^2 }{8\eps^2}(x^2 - y^2)\,.
\end{equation}

With the aid of a computer algebra system (see, also, Appendix~\ref{s:dimensionone}) we conclude
\begin{equation}\label{e:MSDone}
\begin{split}
    \mathrm{MSD}_1 &= \eps\frac{2\delta}{\pi(4 + \delta(-2+\delta))}\set{(8+\delta^3)\arctan\paren{\sqrt{\frac{8}{\delta^3}}} - 2 \sqrt{2}\delta^{3/2}}\\
    &=  \eps m(\delta)\,.
\end{split}
\end{equation}
A consequence of this computation is that it provides an explicit example for which MALA will not satisfy the local robustness criterion, even with optimal $\sigma$, \eqref{e:MSDrobustness2}.

The function $m(\delta)$ has a single maximum
(see Figure~\ref{fig:mfunction}), and the maximum value is achieved at $\delta_\star = 1.27797$, and $m(\delta_\star) = 1.8494$.  Thus, the optimal scaling for the time step is
\begin{equation*}
\sigma = \sqrt{{2\delta_\star}{\eps}}\approx{1.59873}{\sqrt{\eps}}\,.
\end{equation*}
This scaling is a different from the scaling found in Section~\ref{s:mala}.  We note that this optimal choice is smaller than the Euler-Maruyama stability threshold  which requires $\sigma\leq 2 \sqrt{\eps}$.  Additionally, the acceptance rate at this optimal value is $0.70$, which is quite different than the $n\to \infty$ acceptance rate of $0.57$.

\subsection{Poorly scaled proposals.}
\label{s:badscaling}

{We show rigorously that, under particular assumptions, if $\sigma$ is improperly scaled, MALA will fail to be robust in the sense of \eqref{e:Globalrobustness}.   Our approach is based on the method of \cite{Livingstone1908} (see, also, \cite{Rosenthal2003,Zanella:201733a}) using the relationship between the Dirichlet form and the spectral gap, \eqref{e:GapT}.}

{As in \cite{Livingstone1908}, the interval $\Lambda$
is bounded in terms of the {\it conductance} of a set, $K$, with $\mu(K)\in (0,1/2)$.  Using the observable
\begin{equation}
\label{e:conductobs}
f(x) = \frac{1}{\sqrt{(1-\mu(K))\mu(K)}}(1_K(x) - \mu(K))\,,
\end{equation}
we immediately have for a method with the transition operator $\mathcal{T}$
\begin{equation}
    \Gap(\mathcal{T})\leq 2 \P_\mu(X_{k+1}\in K^c \mid X_k \in K)\,.
\end{equation}}

{\begin{proposition}
\label{p:malascaling}
Assume that
\begin{enumerate}

    \item  The potential is $V(x)= V_0(x) + V_1(x)$, where $V_0$ is trapping and  $V_1(x) = V_1(x, x/\eps)$ is bounded.

    \item A concentration inequality holds for $\mu$
    \begin{equation}
    \label{e:concentration}
        \P_\mu(|X|\geq t) \leq C_0 e^{- \kappa t^{p}}
    \end{equation}
    where $C_0$, $\kappa$, and $p$ are independent of $\eps$.

    \item There exists  $m>0$ such that for all $\eps$ sufficiently small, the set $K = \{x\,|\,|x|>m\}$ satisfies
    \begin{equation}
    \label{e:muKbound}
        0< \underline{\mu}_K<\mu(K)<\overline{\mu}_K<\tfrac{1}{2}\,,
    \end{equation}
    where $\underline{\mu}_K$ and $\overline{\mu}_K$ are independent of $\eps$.

    \item The gradient $\nabla V_0$ satisfies a  growth bound with an exponent $\delta > 0$, such that for all $x$
    \begin{equation}
    \label{e:growth}
        |\nabla V_0(x)|\leq D_0 + D_1 |x|^\delta\,.
    \end{equation}

    \item Let $\alpha\geq 0$, $\theta\in \R$, and $\gamma>0$ satisfy
    \begin{equation}
       \label{e:powers}
        2\alpha + \delta\gamma + \theta <1\,.
    \end{equation}
    \end{enumerate}

    Let $\sigma = \eps^{\alpha}$, and define the sets
    \begin{align}
    \label{e:Kepsgamma}
    K_{\eps,\gamma} & = \{m<|x|\leq \eps^{-\gamma}\}\subset K\,,\\
    \label{e:largegradeints}
        L_{\eps,\theta} &= \{x\mid |\nabla V_1(x)|> \eps^{\theta-1}\}\,.
    \end{align}
    Then it holds
    \begin{equation}
    \label{e:Interval1}
    \begin{split}
        \Gap(\mathcal{T}_{\rm MALA})
        &\lesssim e^{{- \frac{1}{32}  \eps^{-2(1-\alpha-\theta)}}}+ e^{-\kappa\eps^{-p\gamma }}\\
        &\quad + \P_\mu(X_{k+1}\in K^c\cap L_{\eps, \theta}^c, X_k \in K\cap L_{\eps,\theta}^c)\,.
        \end{split}
    \end{equation}
\end{proposition}
In \eqref{e:Interval1}, the relation $\lesssim$ is introduced.  Generically, if $a\lesssim b$, then there is a constant, $C$, independent of $b$, such that $a\leq C b$.

Before proving the result, a few remarks are in order.
\begin{enumerate}[(i)]
    \item The concentration inequality for $\mu$ will hold provided $\mu_0$, the distribution associated with $V_0$, has a concentration inequality.  This is a consequence of Corollary~\ref{cor:eps}.

    \item The bound \eqref{e:muKbound} will similarly hold provided such a condition holds for $V_0$.  Again, this is a consequence of Corollary~\ref{cor:eps} by selecting a set $K$ for which $0<\mu_0(K)< \tfrac{1}{2}e^{-\osc V_1}$.

    \item The set $L_{\eps, \theta}$ captures the points in  space where $\nabla V_1$  will be large.  To obtain a rate, it is necessary to know the measure of $L_{\eps,\theta}^c$, requiring additional details on the structure of the potential.  In particular, it is essential to estimate
    \begin{equation}
    \label{e:smallset}
        \P_\mu(X_{k+1}\in K^c\cap L_{\eps, \theta}^c, X_k \in K\cap L_{\eps,\theta}^c)\,,
    \end{equation}
    which is clearly related to the conductance of set $K$ with the additional constraints that $X_k$ and $X_{k+1}$ both reside in $ L_{\eps, \theta}^c$, the set where the gradient is  ``small''.

    \item To see how the set \eqref{e:smallset} can become small,
    consider the case that $V_0= x^2/2$ and $V_1 = \cos(x/\eps)$.   Then $\nabla V_1(x) = -\sin(x/\eps)/\eps$.  For any $\theta>0$ and all $\eps>0$ sufficiently small
    \begin{equation*}
    \label{e:exampleL}
        L_{\eps,\theta}^c = \set{x \mid |\sin(x/\eps)|\leq \eps^\theta}\subset \bigcup_{k \in \Z} [-\eps^{1+\theta}, \eps^{1+\theta}] + \eps \pi k\,.
    \end{equation*}
    Then, recalling that $g(y;x,\sigma)$ is the Gaussian density of the proposal
    \begin{equation*}
        \P_\mu(X_{k+1}\in K^c\cap L_{\eps, \theta}^c, X_k \in K\cap L_{\eps,\theta}^c) \leq \int_{K\cap L_{\eps,\theta}^{c}} \mu(dx)\int_{K^c\cap L_{\eps,\theta}^{c}} g(y;x,\sigma)\,dy\,.
    \end{equation*}
    For any $x\in K\cap L_{\eps,\theta}^{c}$
    \begin{equation}
    \label{e:example1}
    \begin{split}
        \int_{K^c\cap L_{\eps,\theta}^{c}} g(y;x,\sigma)dy&\lesssim \sum_{|\eps k\pi|\leq m} \int_{[-\eps^{1+\theta}, \eps^{1+\theta}] + \eps \pi k} \frac{1}{\sqrt{2\pi \sigma^2}}dy\\
        &\lesssim \eps^{1+\theta - \alpha}\sum_{|\eps k\pi|\leq m} 1 \lesssim \eps^{\theta-\alpha}\,.
    \end{split}
    \end{equation}
    Additionally,
    \begin{equation}
    \label{e:example2}
    \begin{split}
        \mu(K\cap L_{\eps,\theta}^c)\lesssim \mu_0(K\cap L_{\eps,\theta}^c)&\lesssim\sum_{|\eps k \pi|>m}\int_{[-\eps^{1+\theta}, \eps^{1+\theta}] + \eps \pi k} \frac{1}{\sqrt{2\pi}}e^{-x^2/2}dx\\
        &\lesssim \sum_{|\eps k \pi|>m} \eps^{1+\theta} e^{- (\eps k \pi)^2/2}\\
        &\lesssim \eps^{1+\theta}\int_{\frac{m}{\eps \pi}}^\infty e^{- (\eps k \pi)^2/2}dk\lesssim   \eps^{\theta}\,.
        \end{split}
    \end{equation}
    Combining \eqref{e:example1} and \eqref{e:example2}, \eqref{e:smallset} is thus bounded as
    \begin{equation*}
        \P_\mu(X_{k+1}\in K^c\cap L_{\eps, \theta}^c, X_k \in K\cap L_{\eps,\theta}^c)\lesssim \eps^{2\theta - \alpha}
    \end{equation*}
    For $\theta >1/2$ and $\gamma$ and $\alpha$ sufficiently small the above expression is $\ll \eps$.

    \item As the result requires $\alpha<1/2$,  $\sigma\gg \sqrt{\epsilon}$ in this regime.

\end{enumerate}

\begin{figure}
    \centering
    \includegraphics[height=4cm]{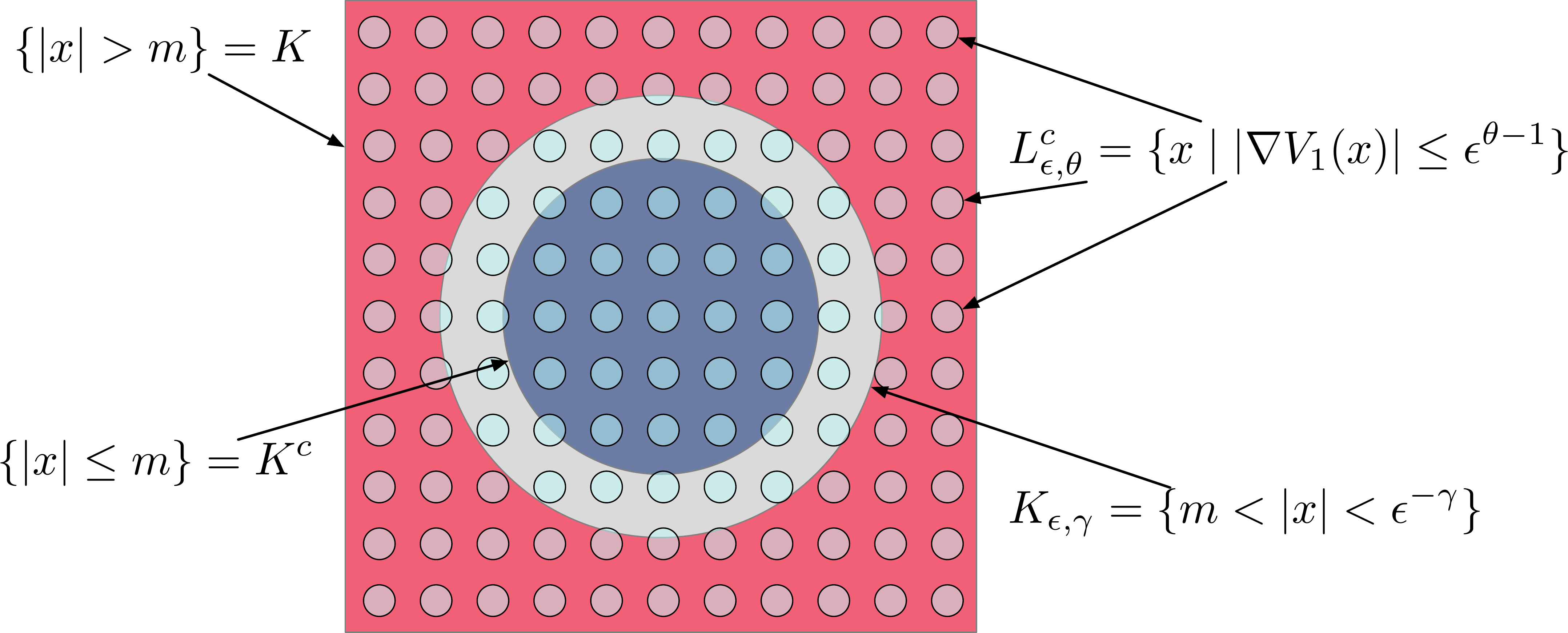}
    \caption{A diagram of the sets discussed in
    Proposition~\ref{p:malascaling}.   For the sake of clarity, we have plotted $L_{\eps,\theta}^c$, and not $L_{\eps,\theta}$, the set where the gradient of $V_1$ is ``small''.}
    \label{fig:my_label}
\end{figure}

\medskip
\begin{proof}{{\bf Proof of Proposition~\ref{p:malascaling}}.}

Since $\mu(K)$ is, by assumptions, uniformly bounded away from $0$ and $1/2$ as $\eps\to 0$, it is sufficient to study $\P_\mu(X_{k+1}\in K^c, X_k \in K)$.  For all sufficiently small $\eps$ we can write
\begin{equation*}
\begin{split}
    \P_\mu(X_{k+1}\in K^c, X_k \in K) &=  \P_\mu(X_{k+1}\in K^c, X_k \in K_{\eps,\gamma}) \\
    &+  \P_\mu(X_{k+1}\in K^c, |X_k|> \eps^{-\gamma}) \\
    &\leq  \P_\mu(X_{k+1}\in K^c, X_k \in K_{\eps,\gamma}) + C_0e^{-\kappa \eps^{-\gamma p }}\,,
\end{split}
\end{equation*}
where we have used the assumed concentration inequality at the end.

Next, we divide the set depending on whether $X_k$ and $X_{k+1}$ lie in the high or low gradient sets
\begin{equation}
\label{e:Psplit1}
    \begin{split}
         \P_\mu(X_{k+1}\in K^c, X_k \in K_{\eps,\gamma})  &= \P_\mu(X_{k+1}\in K^c, X_k \in K_{\eps,\gamma}\cap L_{\eps,\theta})\\
            &+\P_\mu(X_{k+1}\in K^c\cap L_{\eps,\theta}\,, X_k \in K_{\eps,\gamma}\cap L_{\eps,\theta}^c)\\
         &+\P_\mu(X_{k+1}\in K^c\cap L_{\eps,\theta}^c, X_k \in K_{\eps,\gamma}\cap L_{\eps,\theta}^c)\,.
    \end{split}
\end{equation}
We consider the first of the three terms on the right hand side of \eqref{e:Psplit1}:
\begin{equation*}
\begin{split}
 \P_\mu(X_{k+1}\in K^c,  X_k &\in K_{\eps,\gamma}\cap L_{\eps,\theta}) \\
 &= \P_\mu(X_{k+1}\in K^c\mid  X_k \in K_{\eps,\gamma}\cap L_{\eps,\theta})\P(X_k \in K_{\eps,\gamma}\cap L_{\eps,\theta})\\
 &\leq \P_\mu(X_{k+1}\in K^c\mid  X_k \in K_{\eps,\gamma}\cap L_{\eps,\theta})\\
    &\leq   \P_\mu(X_{k+1}^{\rm p}\in K^c\mid X_k \in K_{\eps,\gamma}\cap L_{\eps,\theta})\,.
\end{split}
\end{equation*}
Conditioned on $X_k \in  K_{\eps,\gamma}\cap L_{\eps,\theta}$,
\begin{equation*}
\begin{split}
    |X_{k+1}^{\rm p}| &= |X_k - \tfrac{\sigma^2}{2}\nabla V_0(X_k) - \tfrac{\sigma^2}{2}\nabla V_1(X_k) + \sigma \xi_{k+1}|\\
    &\geq \tfrac{\sigma^2}{2}|\nabla V_1(X_k)| - |X_k| - \tfrac{\sigma^2}{2}|\nabla V_0(X_k)|  - \sigma|\xi_{k+1}|\\
    &\geq \tfrac{1}{2}\eps^{2\alpha +\theta - 1} - \eps^{-\gamma} - \tfrac{D_0}{2}\eps^{2\alpha} - \tfrac{D_1}{2}\eps^{2\alpha - \delta\gamma} - \eps^{\alpha}|\xi_{k+1}| \,,
\end{split}
\end{equation*}
therefore
\begin{equation*}
\begin{split}
    &\P_\mu(m\geq |X_{k+1}^{\rm p}|\mid X_k \in K_{\eps,\gamma}\cap L_{\eps,\theta})\\
    &\leq \P(m \geq \tfrac{1}{2}\eps^{2\alpha +\theta - 1} - \eps^{-\gamma} - \tfrac{D_0}{2}\eps^{2\alpha} - \tfrac{D_1}{2}\eps^{2\alpha - \delta\gamma} - \eps^{\alpha}|\xi_{k+1}| )\\
    &= \P(|\xi_{k+1}|\geq \tfrac{1}{2}\eps^{\alpha +\theta - 1} - \eps^{-\gamma-\alpha} - \tfrac{D_0}{2}\eps^{\alpha} - \tfrac{D_1}{2}\eps^{\alpha - \delta\gamma} - \eps^{-\alpha}m)\\
    &= \P(|\xi_{k+1}|\geq\eps^{-(1-\alpha-\theta)}(\tfrac{1}{2}-\eps^{1-\theta - 2\alpha-\delta\gamma}-\tfrac{D_0}{2}\eps^{1-\theta}-\tfrac{D_1}{2}\eps^{1-\theta-\delta\gamma} - \eps^{1-\theta - 2\alpha}m))\,.
\end{split}
\end{equation*}
By our assumptions on the exponents, for all $\eps$ small enough,
\begin{equation*}
    \P_\mu(m\geq |X_{k+1}^{\rm p}|\mid X_k \in K_{\eps,\gamma}\cap L_{\eps,\theta})\leq \P(|\xi_{k+1}|\geq \tfrac{1}{4}\eps^{-(1-\alpha-\theta)})\lesssim e^{-\frac{1}{32}\eps^{-2(1-\alpha-\theta)}}\,,
\end{equation*}
where we have a  tail inequality on $N(0,I_{n})$.

For the second term in \eqref{e:Psplit1} we use reversibility to get
\begin{equation*}
    \P_\mu(X_{k+1}\in K^c\cap L_{\eps,\theta}, X_k \in K_{\eps,\gamma}\cap L_{\eps,\theta}^c) = \P_\mu( X_{k+1} \in K_{\eps,\gamma}\cap L_{\eps,\theta}^c,X_{k}\in K^c\cap L_{\eps,\theta})\,.
\end{equation*}
Using the same approach as above we have
\begin{equation*}
    \begin{split}
    &\P_\mu( X_{k+1} \in K_{\eps,\gamma}\cap L_{\eps,\theta}^c\mid X_{k}\in K^c\cap L_{\eps,\theta})\\
    &\leq \P_\mu( \eps^{-\gamma}\geq |X_{k+1}^{\rm p}|\mid X_{k}\in K^c\cap L_{\eps,\theta})\\
    &\leq \P( \eps^{-\gamma}\geq \tfrac{1}{2}\eps^{2\alpha +\theta - 1} - \eps^{-\alpha}m - \tfrac{D_0}{2}\eps^{2\alpha} - \tfrac{D_1}{2}\eps^{2\alpha}m^{\delta} - \eps^{\alpha}|\xi_{k+1}|)\\
    &= \P( |\xi_{k+1}|\geq \tfrac{1}{2}\eps^{\alpha +\theta - 1} - \eps^{-\alpha}m - \tfrac{D_0}{2}\eps^{\alpha} - \tfrac{D_1}{2}\eps^{\alpha}m^{\delta} - \eps^{-\alpha-\gamma})\\
    & = \P( |\xi_{k+1}|\geq \eps^{-(1-\alpha-\theta)}(\tfrac{1}{2} -  \eps^{1-2\alpha-\theta}m - \tfrac{D_0}{2}\eps^{1-\theta} - \tfrac{D_1}{2}\eps^{1-\theta}m^{\delta} - \eps^{1-\theta-2\alpha-\gamma}))\,.
    \end{split}
\end{equation*}
Again, by our assumptions on the exponents and a tail inequality on the Gaussian
\begin{equation*}
\P_\mu( X_{k+1} \in K_{\eps,\gamma}\cap L_{\eps,\theta}^c\mid X_{k}\in K^c\cap L_{\eps,\theta})\leq\P(|\xi_{k+1}|\geq \tfrac{1}{4}\eps^{-(1-\alpha-\theta)})\lesssim e^{-\frac{1}{32}\eps^{-2(1-\alpha-\theta)}} \,.
\end{equation*}
Combining these estimates we obtain \eqref{e:Interval1}.

\end{proof}

The next proposition provides an estimate of the last term in \eqref{e:Interval1} for a particular class of potentials.
\begin{thm}
\label{t:malascaling2}
Assume that
\begin{enumerate}
    \item There exists $p\geq 1$ and constants $V_{\min}$ and  $C_p>0$
\begin{equation}
\label{e:V0lowerbound}
    V_0(x)\geq V_{\min}+ C_p \sum_{i=1}^n |x_i|^p,
\end{equation}
    \item The growth bound of Proposition \ref{p:malascaling} holds for $\nabla V_0$.
    \item The potential $V_1$ takes the form
    \begin{subequations}
    \begin{gather}
    \label{e:V1assumptions}
        V_1(x) = \sum_{i=1}^n A_i \cos\paren{\frac{x}{\ell_i \eps}}\\
        0<\underline{A} \leq |A_i|\leq \overline{A}<\infty\,, \quad 0<\underline{\ell}\leq \ell_i\leq \overline{\ell}<\infty\,.
    \end{gather}
    \end{subequations}

\item $\alpha$, $\gamma$, $\delta$, and $\theta$ satisfy \eqref{e:powers}.
\end{enumerate}

Then there exists a constant $\kappa$ such that
\begin{equation}
\label{e:Inteval2}
    \Gap(\mathcal{T}_{\rm MALA})\lesssim e^{-\frac{1}{2}\kappa \eps^{-\gamma p}} + e^{-\frac{1}{32}\eps^{-2(1-\alpha-\theta})} + \eps^{n(2\theta-\alpha)}\,.
\end{equation}

\end{thm}
Before proving this result we make a few remarks:
\begin{enumerate}[(i)]
    \item If $V_0$ is harmonic it will satisfy these assumptions with $p=2$ and $\delta = 1$.

    \item As long as $\alpha<2\theta$, the performance will deteriorate as $\eps\to 0$.  By \eqref{e:powers}, this necessitates $\alpha< 2/5$.

    \item In the numerical experiments discussed below,  $\MSDO_1$ was  found to be maximized when $\sigma\sim \sqrt{\eps}$, while in higher dimensions, $\MSD$  was maximized when $\sigma\sim \eps$.   Thus, there is a ``gap'' between values of $\alpha$ for which Theorem \ref{t:malascaling2} predicts poor performance, and values for which we numerically observe peak performance.  We interpret this as a gap in the analysis because, even at the optimal choice of $\sigma$, our results show $\MSD\to 0$ as $\eps\to 0$, implying the method is not locally (or globally) robust in the sense of \eqref{e:MSDrobustness2}.

    \item As $n$ increases, the gap vanishes more and more rapidly as $\eps\to 0$.
\end{enumerate}
\medskip

\begin{proof}{{\bf Proof of Theorem \ref{t:malascaling2}}.}

A concentration inequality will obviously hold for $\mu$, as such an inequality holds for $\mu_0$,
with $\P_{\mu_0}(|X|>t)\lesssim e^{-\frac{1}{2}\kappa t^p}$; this is a consequence of \eqref{e:V0lowerbound}. The constant $\kappa$ is related to $C_p$, $p$, and the dimension.

Since the concentration inequality holds, and $V_0$ is trapping while $V_1$ is bounded, Corollary~\ref{cor:eps} ensures that a set $K = \{|x|>m\}$ exists and satisfies \eqref{e:muKbound}.  Thus, the assumptions of
Proposition~\ref{p:malascaling} hold.

Since
\begin{equation*}
\P_\mu(X_{k+1}\in K^c\cap L_{\eps, \theta}^c, X_k \in K\cap L_{\eps,\theta}^c)\leq \int_{K\cap L_{\eps,\theta}^c}\mu(dx) \int_{K^c\cap L_{\eps, \theta}^c}g(y;x,\sigma)dy\,,
\end{equation*}
it will be sufficient to estimate $\mu(K\cap L_{\eps,\theta}^c)$ and the integral $\int_{K^c\cap L_{\eps, \theta}^c}g(y;x,\sigma)dy$ for $x\in K\cap L_{\eps,\theta}^c$.

We first bound the set $L_{\eps,\theta}^c$.  For all
$\eps$ small enough
\begin{equation*}
    L_{\eps,\theta}^c \subset \prod_{i=1}^n \set{\abs{\frac{A_i}{\ell_i}\sin\paren{\frac{x_i}{\ell_i \eps}}}\leq  \eps^\theta}\subset \prod_{i=1}^n \bigcup_{k_i \in \Z} \underbrace{[- \frac{\ell_i^2}{|A_i|}\eps^{1+\theta}, \frac{\ell_i^2}{|A_i|}\eps^{1+\theta}] + \eps \ell_i \pi k_i}_{\equiv B_{i,k_i}}\,.
\end{equation*}
Next, we bound the set $K$ by
\begin{equation*}
    \begin{split}
        K &\subset \bigcup_{j=1}^n \{|x_j|> m/\sqrt{n}\}
    \end{split}
\end{equation*}
Consequently, for all $\eps$ small enough,
\begin{equation}
\label{e:LcK}
    L_{\eps,\theta}^c\cap K \subset \bigcup_{j=1}^n \paren{ \bigcup_{|k_j|>\frac{m}{\eps\ell_i\pi \sqrt{n}}} B_{j,k_{j}}\times\prod_{i\neq j} \bigcup_{k_i \in \Z} B_{i,k_{i}}}\,.
\end{equation}
Similarly
\begin{equation*}
    K^c \subset \prod_{j=1}^n \{|x_j|\leq m \}\,,
\end{equation*}
and
\begin{equation}
   \label{e:LcKc}
        L_{\eps,\theta}^c\cap K^c \subset \prod_{i=1}^n \bigcup_{|k_i|<\frac{m}{\eps \ell_i\pi}} B_{i,k_i}\,.
\end{equation}

Now, using \eqref{e:V0lowerbound} and \eqref{e:LcK},
\begin{equation*}
\begin{split}
    \mu(L^c_{\epsilon,\theta}\cap K)&\leq \sum_{j=1}^n {\mu\paren{\bigcup_{|k_j|>\frac{m}{\eps\ell_i\pi \sqrt{n}}} B_{j,k_{j}}\times\prod_{i\neq j} \bigcup_{k_i \in \Z} B_{i,k_{i}}}}\\
    &\lesssim \sum_{j=1}^n\paren{\sum_{|k_j|>\frac{m}{\eps\ell_i\pi \sqrt{n}}}\int_{B_{j,k_{j}}} e^{-C_p |x_j|^p}dx_j }\prod_{i\neq j}\paren{\sum_{k_i}\int_{B_{i,k_{i}}}
    e^{-C_p |x_i|^p}dx_i}\,.
\end{split}
\end{equation*}
The first term in the above sum is as in the example \eqref{e:example2}
\begin{equation*}
\begin{split}
    \sum_{|k_j|>\frac{m}{\eps\ell_j\pi \sqrt{n}}}\int_{B_{j,k_{j}}} e^{-C_p |x_j|^2}dx_j&\lesssim \eps^{1+\theta} \sum_{|k_j|>\frac{m}{\eps\ell_j\pi \sqrt{n}}} e^{- C_p (\eps \ell_i\pi k_j)^p }\\
    &\lesssim\eps^{1+\theta}\int_{\frac{m}{\eps\ell_j\pi \sqrt{n}}}^\infty e^{- C_p (\eps \ell_j\pi k_j)^p } dk_j\lesssim \eps^{\theta}\,.
\end{split}
\end{equation*}
The other terms are treated the same way
\begin{equation*}
\sum_{k_i}\int_{B_{i,k_{i}}} e^{-C_p |x_i|^p}dx_i\lesssim \eps^{1+\theta} \int_0^\infty e^{- C_p (\eps \ell_i\pi k_i)^p } dk_i\lesssim \eps^\theta\,.
\end{equation*}
Consequently
\begin{equation}
\label{e:muLcK}
     \mu(L^c_{\epsilon,\theta}\cap K)\lesssim \eps^{n\theta}\,.
\end{equation}

For $x \in L_{\eps,\theta}^c\cap K$, using \eqref{e:LcKc},
\begin{equation}
\label{e:gLcKc}
\begin{split}
    \int_{L_{\eps,\theta}^c\cap K^c}g(y;x,\sigma)dy &\lesssim \prod_{i=1}^n \sigma^{-1}\mathrm{Leb.}\paren{\bigcup_{|k_i|<\frac{m}{\eps \ell_i\pi\sqrt{n}}} [-2 \frac{\ell_i}{A_i}\eps^{1+\theta}, 2\frac{\ell_i}{A_i}\eps^{1+\theta}] + \eps \ell_i \pi k_i}\\
    &\lesssim \prod_{i=1}^n \paren{\eps^{1+\theta-\alpha}\sum_{|k_i|<\frac{m}{\eps \ell_i\pi\sqrt{n}}}1 }\lesssim \eps^{n(\theta-\alpha)}\,.
    % &\lesssim
    \end{split}
\end{equation}
Combining this with \eqref{e:muLcK} yields the result.

\end{proof}
}

\section{Alternative algorithms for rough landscapes.}\label{s:algorithms}

Having concluded that there are impediments to MALA for sampling on rough landscapes, but desiring a method that may perform better than RWM, we propose two methods here. Both these methods require a smoothed version of the potential which may not be accessible.  Indeed, as the potential of interest is unlikely to admit a simple decomposition like \eqref{e:V1}, we also propose a method for approximating a smoothed landscape.

\subsection{Modified MALA.}

The first method, which we call Modified MALA, involves generating proposals from an auxiliary potential, $U$. The algorithm generates the proposals
\begin{equation}
\label{e:ModMALAprop}
    X^{\mathrm{p}}_{k+1} = X_k  - \tfrac{\sigma^2}{2}\nabla U(X_k) + \sigma \xi_{k+1}\,,
\end{equation}
and $R(x,y)$ is modified to be
\begin{equation}
\label{e:RModMALA}
    R(x,y) = V(x) - V(y) -\tfrac{1}{2\sigma^2}|{x - y +\nabla U(y)\tfrac{\sigma^2}{2}}|^2+ \tfrac{1}{2\sigma^2}|{y - x +\nabla U(x)\tfrac{\sigma^2}{2}}|^2\,.
\end{equation}
If the scale separation in \eqref{e:V1} holds, we might take  $U= V_0$ to get \eqref{e:ModMALAV0prop}.  Since the proposal is $\eps$-independent, Corollary~\ref{cor:Teps} applies to this method.

\subsection{Independence sampler.}
The second method we propose is to generate auxiliary samples, $Y_k$, from another distribution $\propto e^{-U(x)}dx$, and  perform independence sampling against $\mu(dx)\propto e^{-V(x)}dx$.  More explicitly,
\begin{subequations}
\label{e:RejectionSampler}
\begin{gather}
    \Delta(x) = V(x) - U(x)\,,\;\; \quad R(x,y)   = \Delta(x) - \Delta(y)\\
    X_{k+1} =\begin{cases}
    Y_{k+1}, & \text{with probability }F(R(X_k,Y_{k+1})), \\
    X_k, & \text{otherwise.}
    \end{cases}
\end{gather}
\end{subequations}
This approach is agnostic as to how samples from the auxiliary distributions are generated -- any strategy which produces (approximately) independent samples will be satisfactory.   When the assumptions about \eqref{e:V1} hold, then we would take $U=V_0$ giving  \eqref{e:MCMCV0prop}.  Corollary~\ref{cor:Teps} again applies to this method.   Given the choice of $U$ and the mechanism for sampling from its distribution, this algorithm has no free parameters -- there is no $\sigma$ that can be tuned.

In our numerical examples below, where the potentials take the form \eqref{e:Vepssum}, we numerically approximate an inverse function sampler on each degree of freedom, generating i.i.d. samples $\propto e^{-v_0(x_i)}$ for $i=1,\ldots, n$.  This approach, of course, will not be practical for general landscapes, and we return to this issue in the discussion.

\subsection{Finding smoothed landscapes.}\label{s:locent}

A question that remains is what to do when \eqref{e:V1} fails to hold.  One option is to use physical intuition about the problem to identify a potential $U(x)$ that has suitable properties.  More systematically, we can use the local entropy approach formulated in \cite{chaudhari2016entropysgd,chaudhari2018deep}, or, equivalently the Moreau-Yosida approximation to estimate a smoothed version of $V(x)$.

Given $\gamma>0$, we define $V_\gamma$ as
\begin{equation}
\label{e:Vgamma}
    V_\gamma(x)  = -\beta^{-1}\log\set{ \int \paren{\frac{\beta}{2\pi \gamma}}^{\frac{n}{2}} e^{-\frac{\beta|x-y|^2}{2\gamma}} e^{-\beta V(y)}dy}\,,
\end{equation}
which corresponds to a Gaussian filter of the Boltzmann distribution.  The associated gradient is then
\begin{equation}
\label{e:gradVgamma}
\nabla V_\gamma(x) = \gamma^{-1}\int_{\mathbb{R}^n}(x-y)\rho_\infty (y\mid x)dy = \gamma^{-1}(x - \mathbb{E}^{\rho_\infty(\cdot\mid x)}[Y])\,.
\end{equation}
The density, $\rho_\infty(y\mid x)$, can be defined by first
introducing
\begin{equation}
\label{e:Ugamma}
U_\gamma(y\mid x) = V(y) + \tfrac{1}{2\gamma}|x-y|^2\,,
\end{equation}
so that
\begin{equation*}
    \rho_\infty(y\mid x) \propto e^{- \beta U_\gamma(y\mid x)}\,.
\end{equation*}
The potential $V_\gamma$  could be estimated with the simple Monte Carlo scheme
\begin{equation}\label{e:vgamma_mc}
    V_\gamma(x) \approx -\beta^{-1}\log\paren{ \frac{1}{\NS_{\mathrm{s}}}\sum_{j=1}^{\NS_{\mathrm{s}}} e^{-\beta V(x + Y^{(j)})}}\,,\;\;\quad Y^{(j)} \sim N(0, \gamma\beta^{-1} I)\,.
\end{equation}
Likewise, \eqref{e:gradVgamma} could be estimated by integrating the auxiliary diffusion
\begin{equation}\label{e:local_sde}
dY_t = - \nabla U_\gamma(Y_t\mid x)dt+
         \sqrt{2\beta^{-1}}dW_t^y\,.
\end{equation}
If $V$ is assumed to grow at infinity, then for sufficiently small $\gamma$, $U_\gamma(y\mid x)$ will be convex about $y=x$, and  \eqref{e:local_sde} will converge to equilibrium exponentially fast. The gradient
$\nabla V_\gamma$  can then be estimated by Monte Carlo
\begin{equation}\label{e:gradvgamma_mc}
    \nabla V_\gamma(x) \approx \frac{1}{\NS_{\mathrm{s}}}\sum_{j=1}^{\NS_{\mathrm{s}}} \gamma^{-1}(x - Y^{(j)})\,,
\end{equation}
where each $Y^{(j)}\sim \rho_\infty(\cdot\mid x)$.  Of course, $Y^{(j)}$ must now be appropriately sampled.  Running short i.i.d. samples of \eqref{e:local_sde} using, for instance, MALA, the algorithm introduces three additional numerical parameters: $\NS_{\mathrm{s}}$, the number of samples; $\delta t$, the fine scale time step; and $\NS_{\delta t}$, the number of fine scale  time steps.  We must also specify initial conditions.

With estimates of $V_\gamma$ and $\nabla V_\gamma$, we then use them as the smoothed landscapes in our coupled independence sampler scheme or  Modified MALA scheme.

\section{Numerical experiments.}
\label{s:numerics}

In this section we demonstrate, numerically, how roughness can impede MALA, and how RWM, along with the methods discussed in Section~\ref{s:algorithms}, can resist the roughness.  In these computational examples, we use the Metropolis rule, \eqref{e:metrop}.
By Remark~\ref{rem:variance}, it is sufficient to study $\MSD$ without including the variance factor in \eqref{e:MSDrobustness} or \eqref{e:MSDrobustness2}.

\subsection{Rough harmonic potential.}
\label{s:harmonic}

As a first example, we consider a potential of the type \eqref{e:Vepssum}, with
\begin{equation}\label{e:roughharm}
    v_\eps(x_i) = \tfrac{1}{2}x_i^2 + \tfrac{1}{8}\cos(x_i/\eps)\,.
\end{equation}
The potential is depicted in Figure~\ref{fig:roughharm} at various values of $\eps$.

\begin{figure}
    \centering
    \includegraphics[width=6cm]{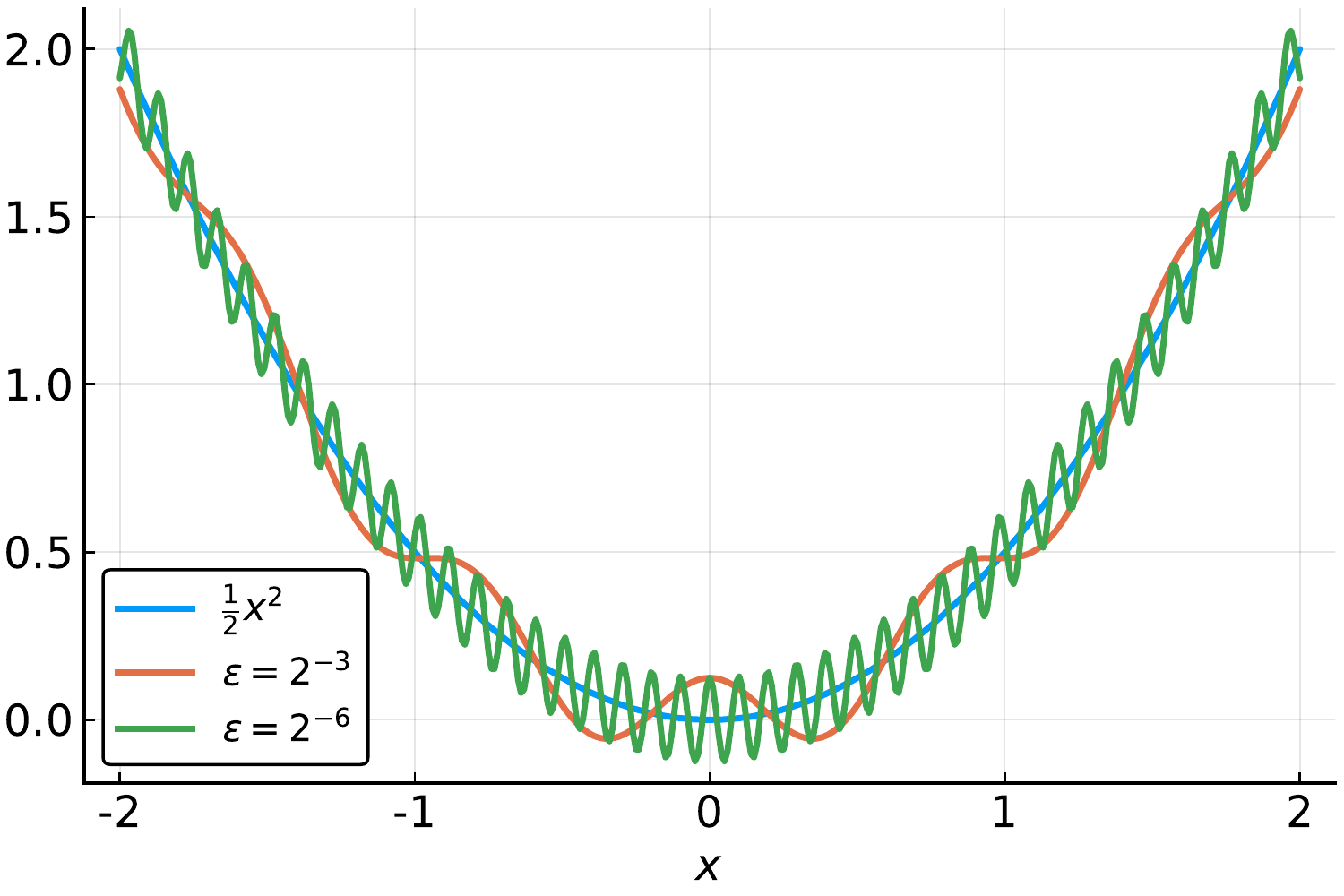}
    \caption{Harmonic potential \eqref{e:roughharm} with an additive rough term (color online).}
    \label{fig:roughharm}
\end{figure}

We explore this potential  by varying both $\eps$ and the dimension $n$.  {A priori}, we do not know the optimal value of $\sigma$ for each $\eps$ and $n$ pair, for each algorithm.  Thus, we examine a range of $\sigma$ values, running long trajectories for each, and then empirically estimate the one with the maximum $\MSD$.  See Figure~\ref{fig:roughharmMSD} for examples of such a computations.  In this way we are able to compare performance across methods.  In these experiments inverse temperature is set to  $\beta=5$ and $10^8$ iterations are performed in each run.  The starting point for these runs is $x_0 = (0,0,\ldots,0)^{T}\in \R^n$.   For the independence sampler, we combine numerical quadrature and interpolation to approximate the inverse cumulative distribution function of the distribution $e^{-\beta v_0(x_i)}$.  This allows us to perform inverse function sampling.  Consequently, there are no free parameters to tune for this strategy.

\begin{figure}
    \subfigure[]{\includegraphics[width=6.25cm]{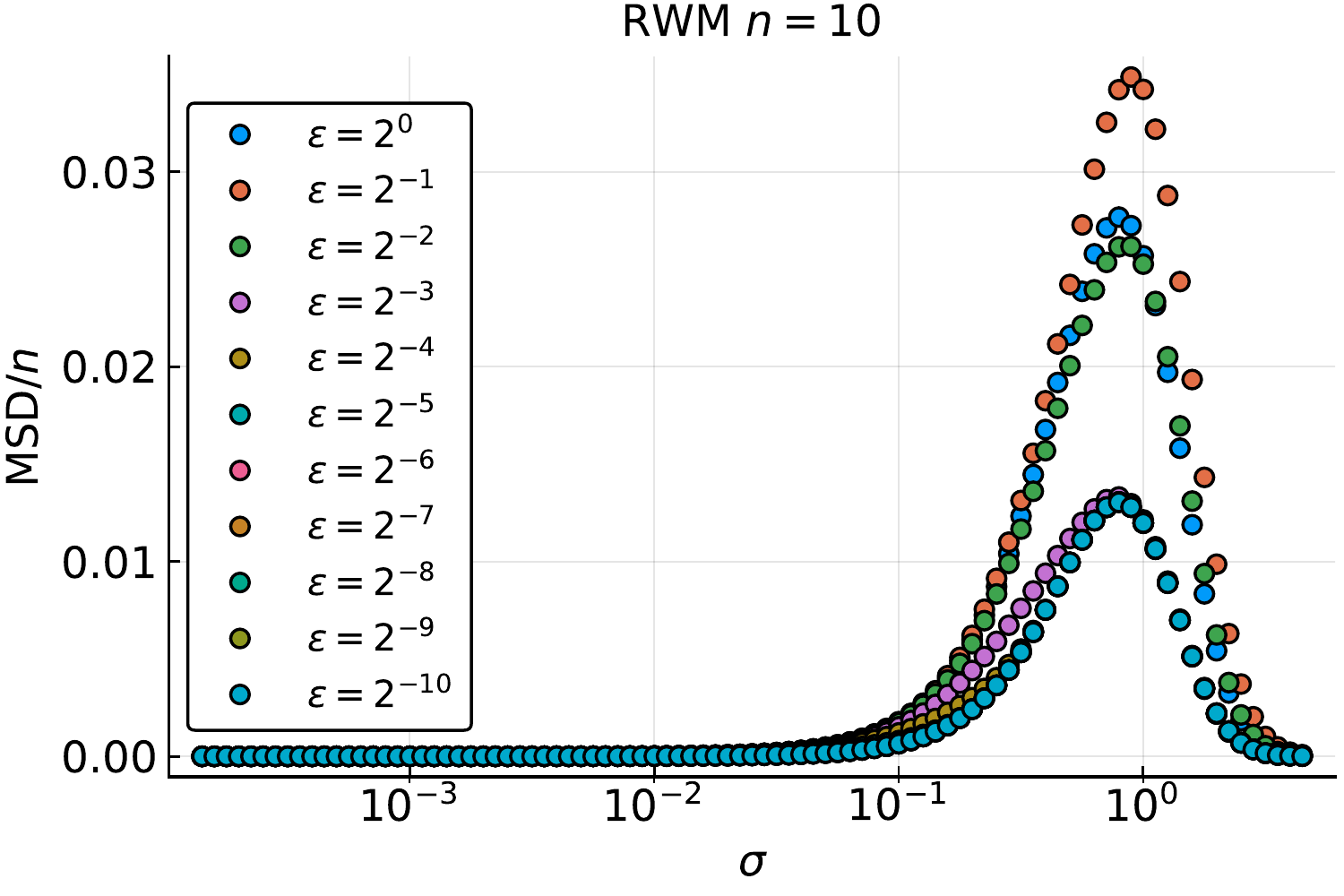}}
    \subfigure[]{\includegraphics[width=6.25cm]{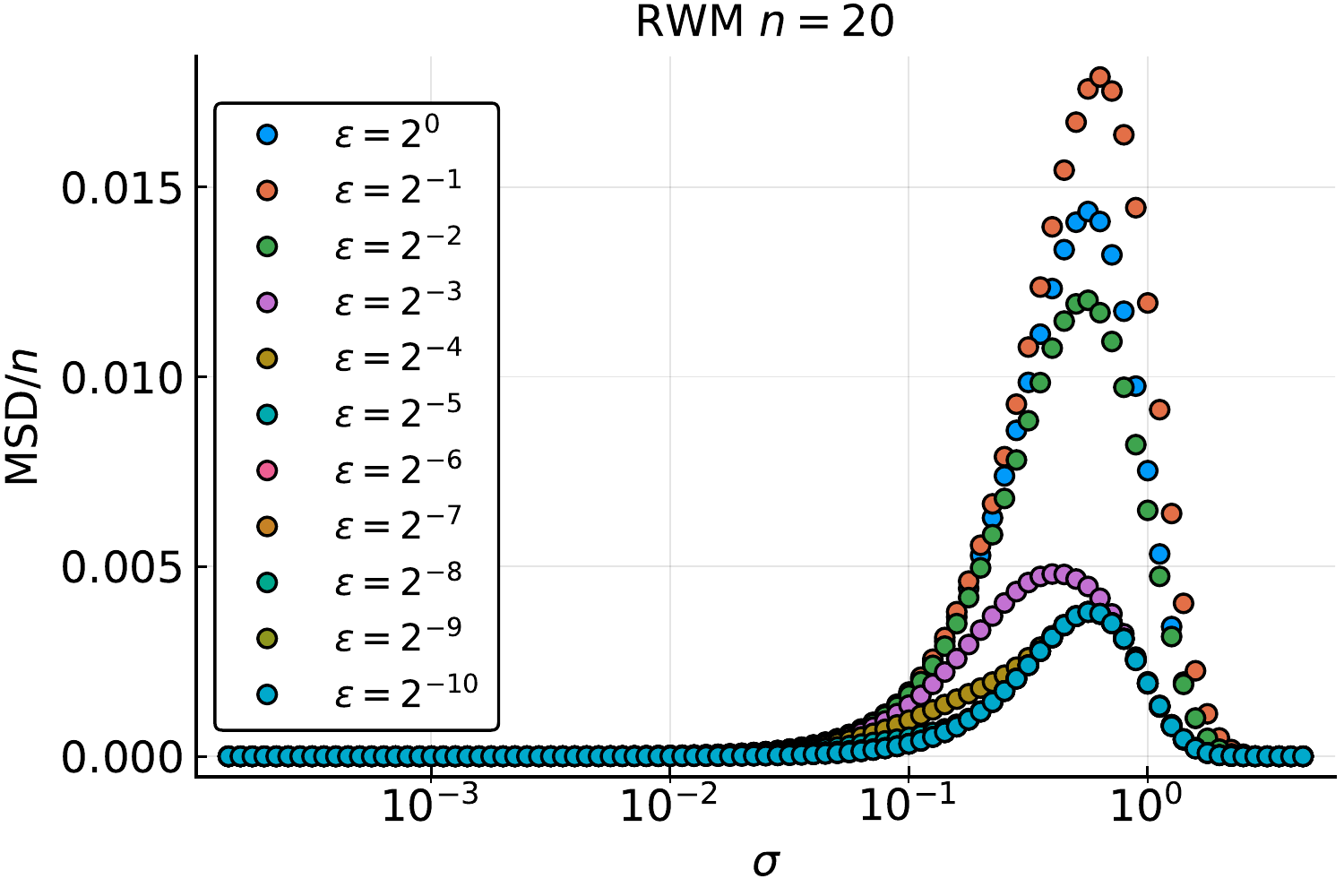}}

    \subfigure[]{\includegraphics[width=6.25cm]{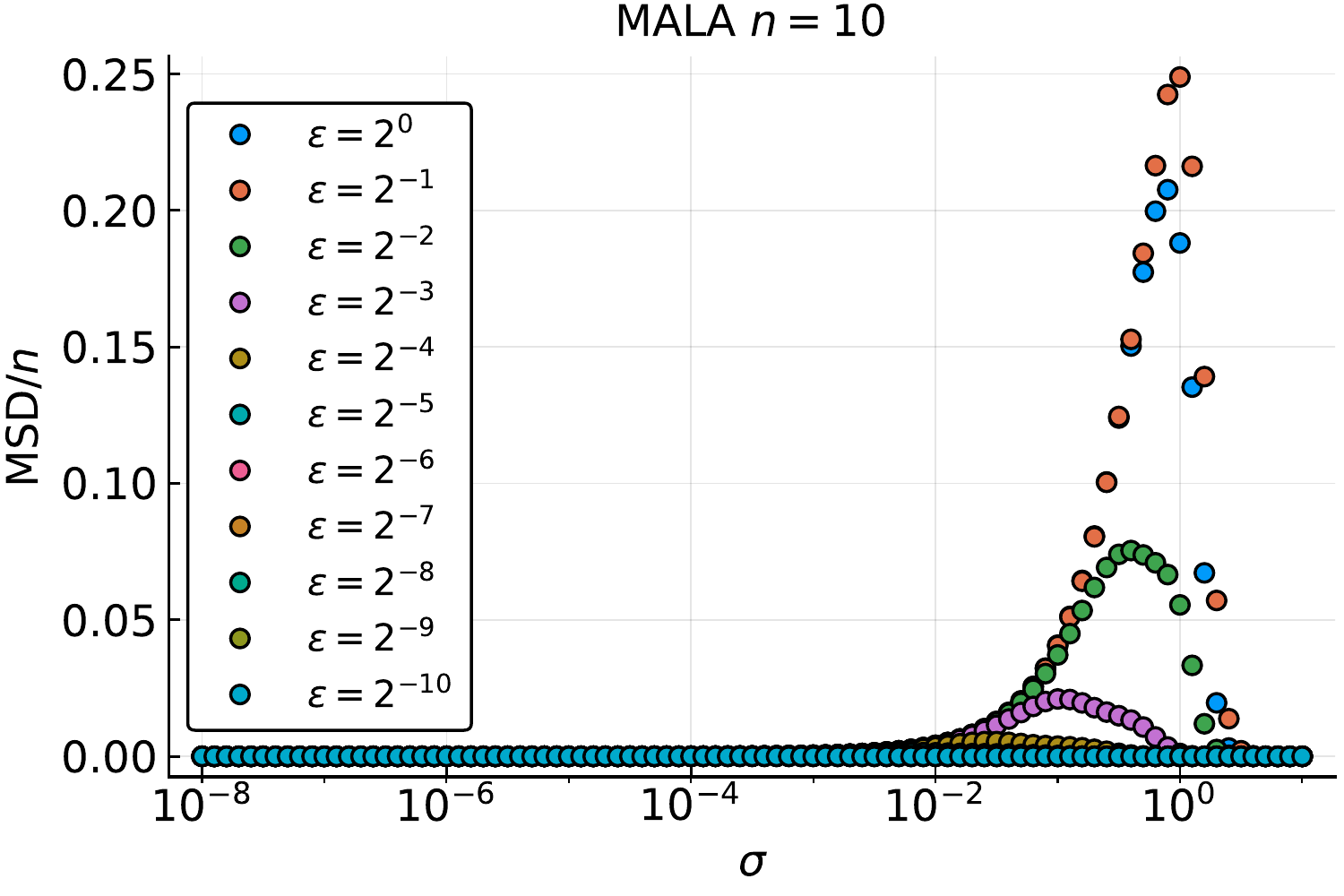}}
    \subfigure[]{\includegraphics[width=6.25cm]{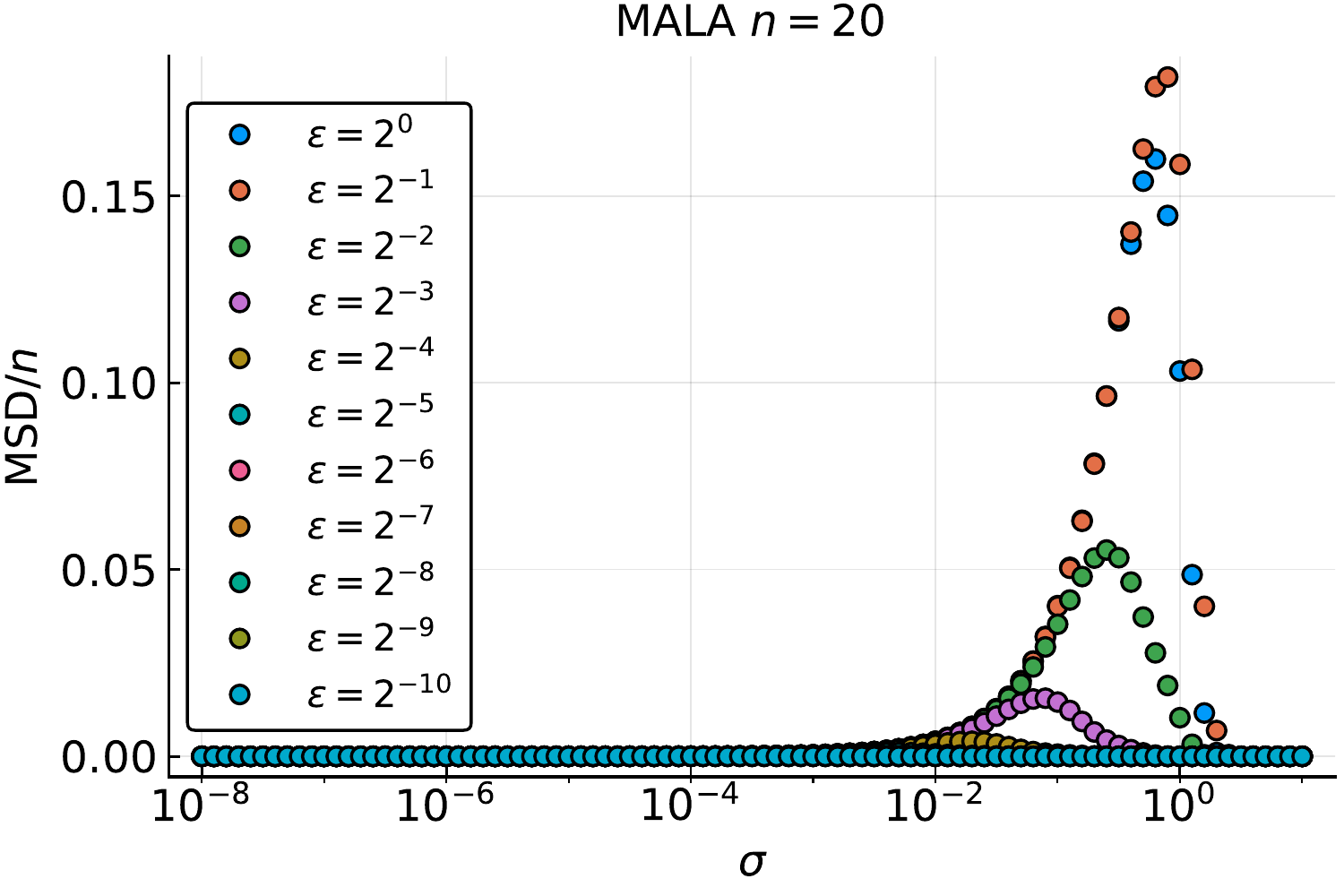}}

    \subfigure[]{\includegraphics[width=6.25cm]{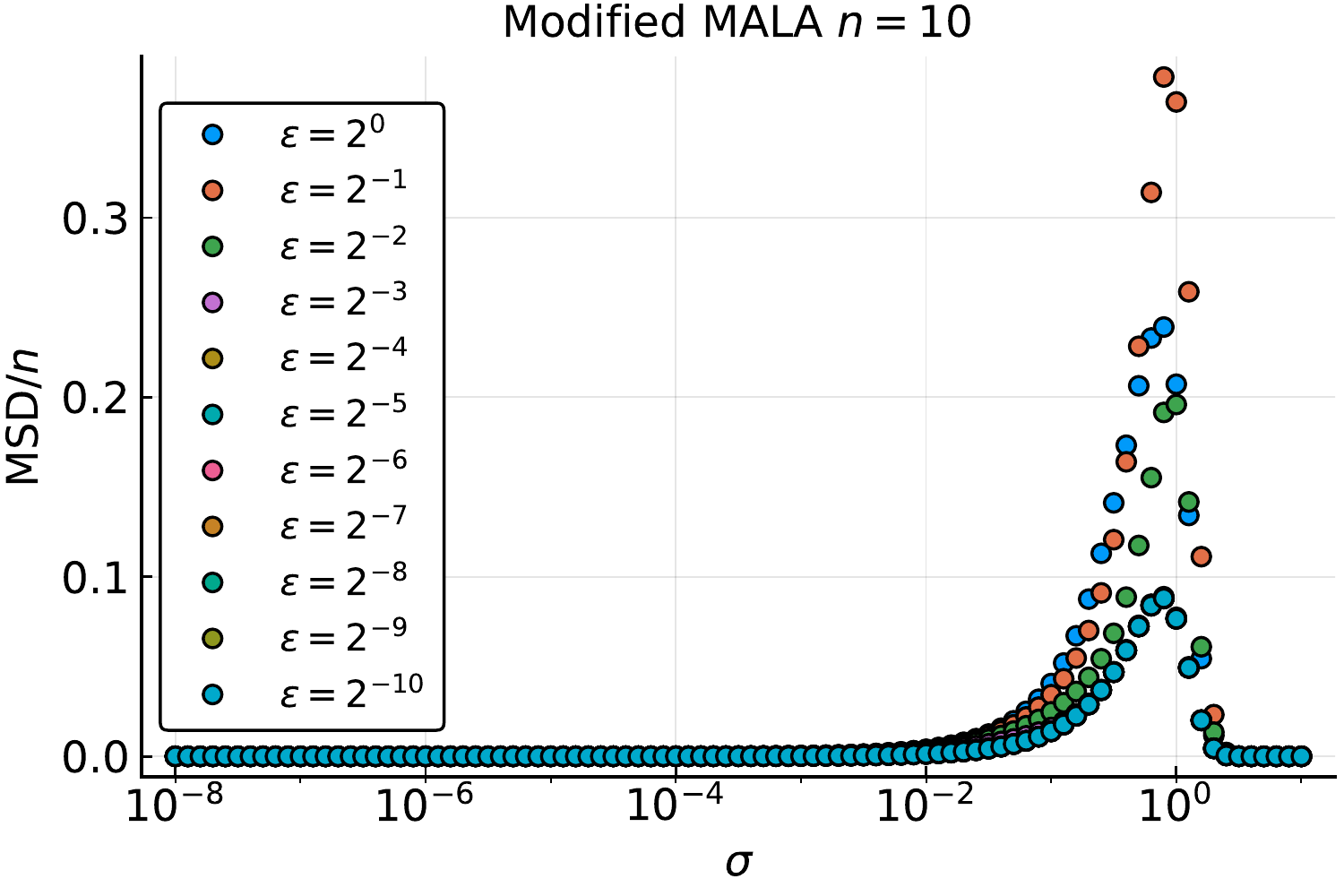}}
    \subfigure[]{\includegraphics[width=6.25cm]{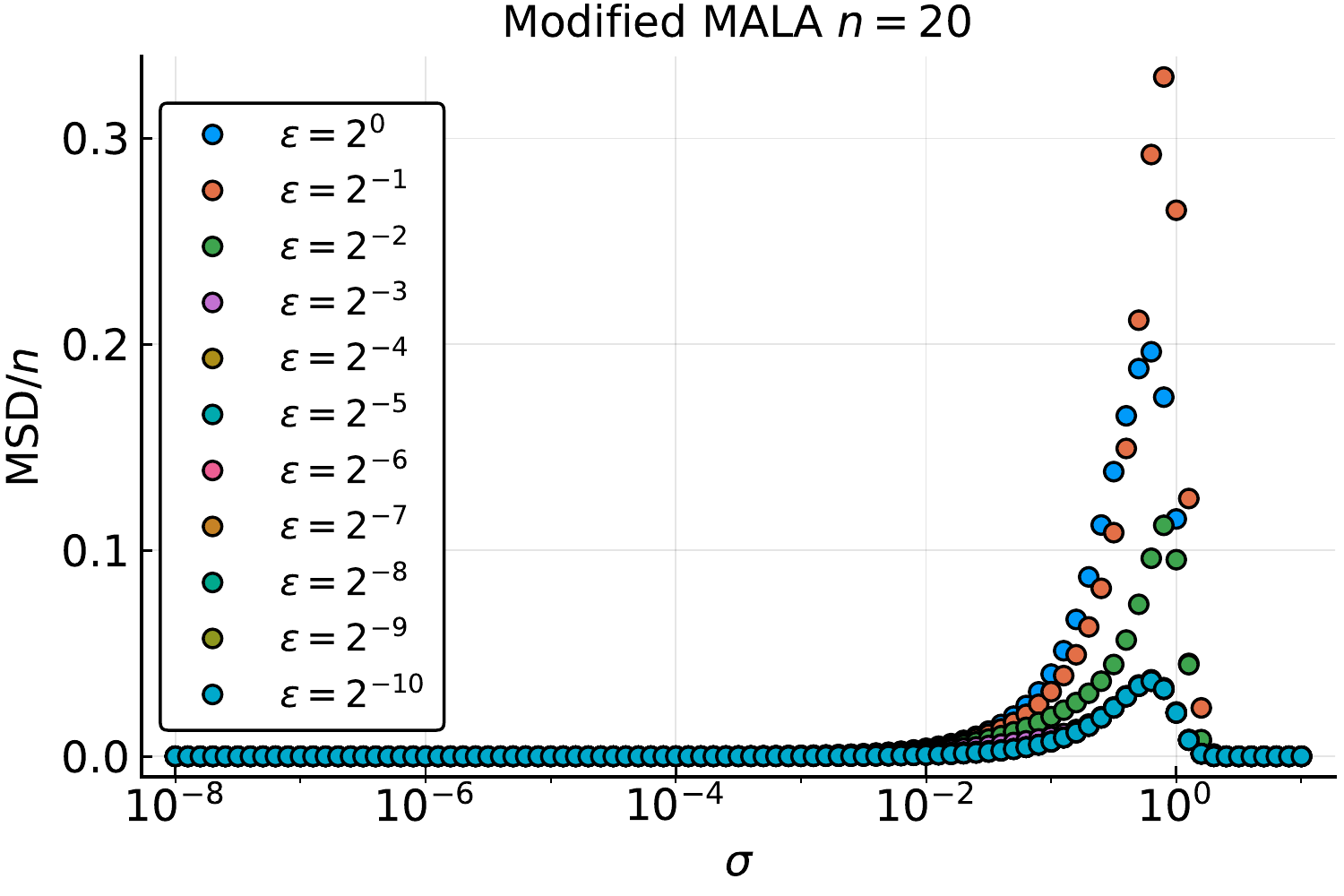}}

    \caption{Example computations of the MSD for \eqref{e:roughharm} for a variety of $\sigma$ and $\eps$ values.  The empirical maximum, for each $(\sigma,\eps)$ pair is interpreted to be the optimal choice.  Note that the peak of $\MSD$ for MALA falls off rapidly with $\eps$ (color online).}
    \label{fig:roughharmMSD}
\end{figure}

{Having estimated the optimal $\sigma$, we can then compare performance between the algorithms, as a function of $\eps$ and $n$.  This is shown in Figure~\ref{fig:roughharmopt}.  This indicates that  $\sigma\propto \eps$ for MALA, provided $n$ is large enough.  In this scaling, Theorem \ref{t:malascaling2} does not apply.  However, the numerical experiments reveal that mixing, as measured by MSD, is impeded because of the size of $\sigma$.  We thus conclude from the numerical experiment that MALA fails to satisfy \eqref{e:MSDrobustness2}, implying that it will also fail to be globally robust, even with optimal $\sigma$.  Also note that the mean acceptance rates across a range of $\eps$ and $n$ deviate from the idealized $n\to \infty$ RWM and MALA values.}

Since RWM, Modified MALA, and the Independence sampler are all globally robust to $\eps\to 0$ in the sense of condition \eqref{e:Globalrobustness}, we examined amplification in performance at different $\eps$ and $n$, as measured by
\begin{equation}\label{e:amp}
    \text{Performance Amplification} = \frac{\text{Optimal MSD for Alternative Method}}{\text{Optimal MSD for RWM}}
\end{equation}
This is shown in Table~\ref{tab:harmamp}.  The alternative methods always beat RWM, and, there is a greater improvement in higher dimension, though the performance improvement saturates as $\eps \to 0$.  The independence sampler method typically outperforms the Modified MALA method.

\begin{figure}

    \subfigure[]{\includegraphics[width=6.25cm]{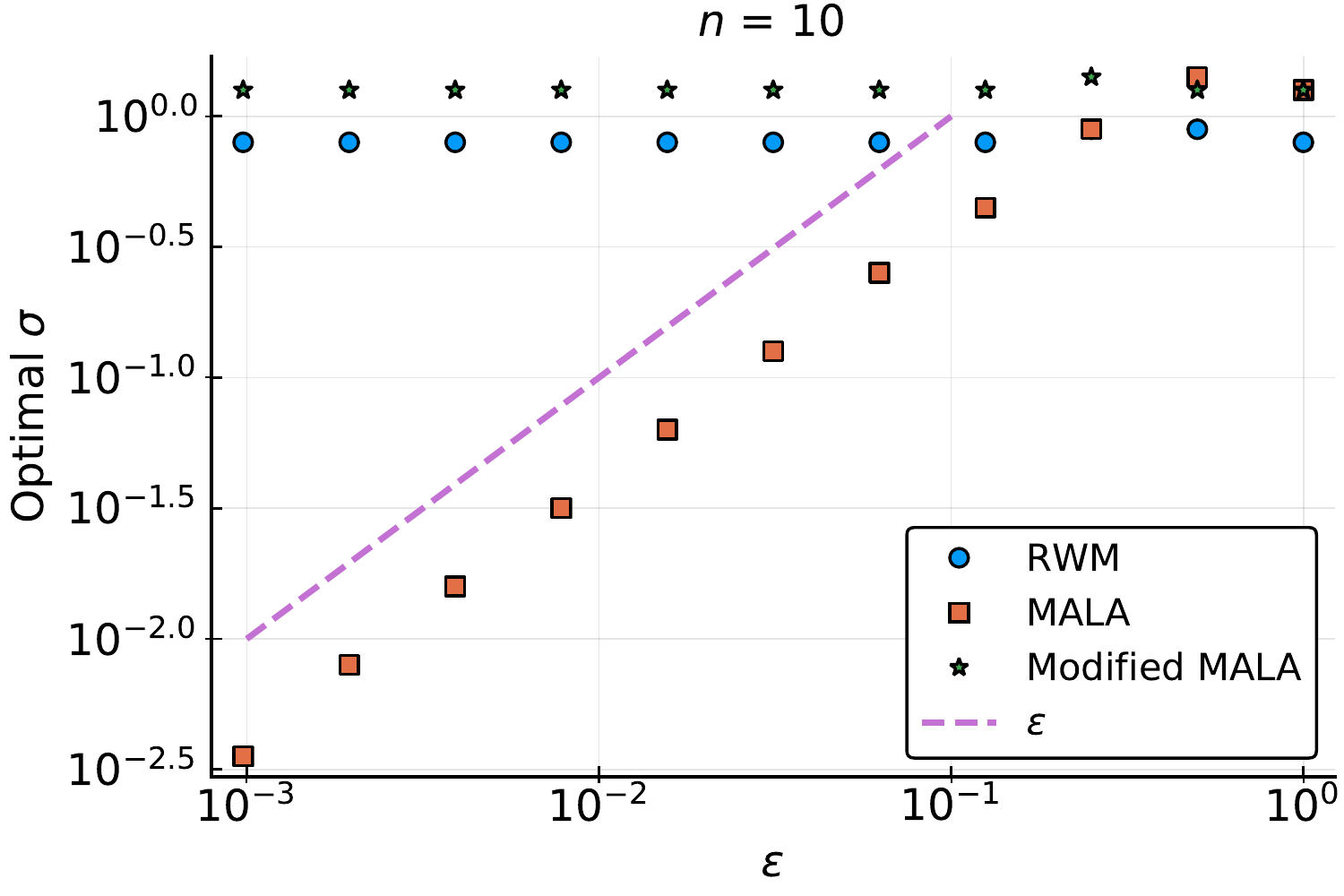}}
    \subfigure[]{\includegraphics[width=6.25cm]{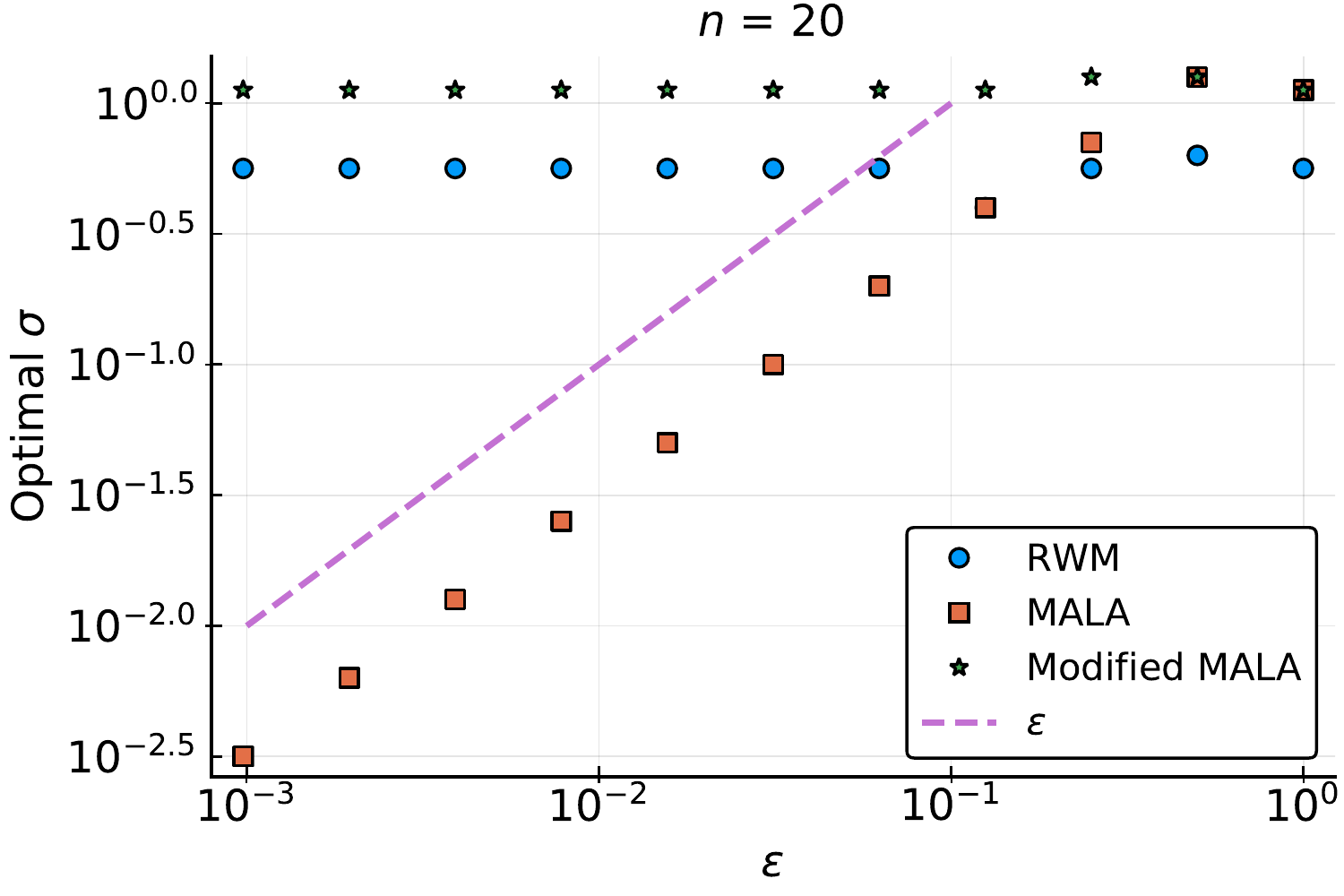}}

    \subfigure[]{\includegraphics[width=6.25cm]{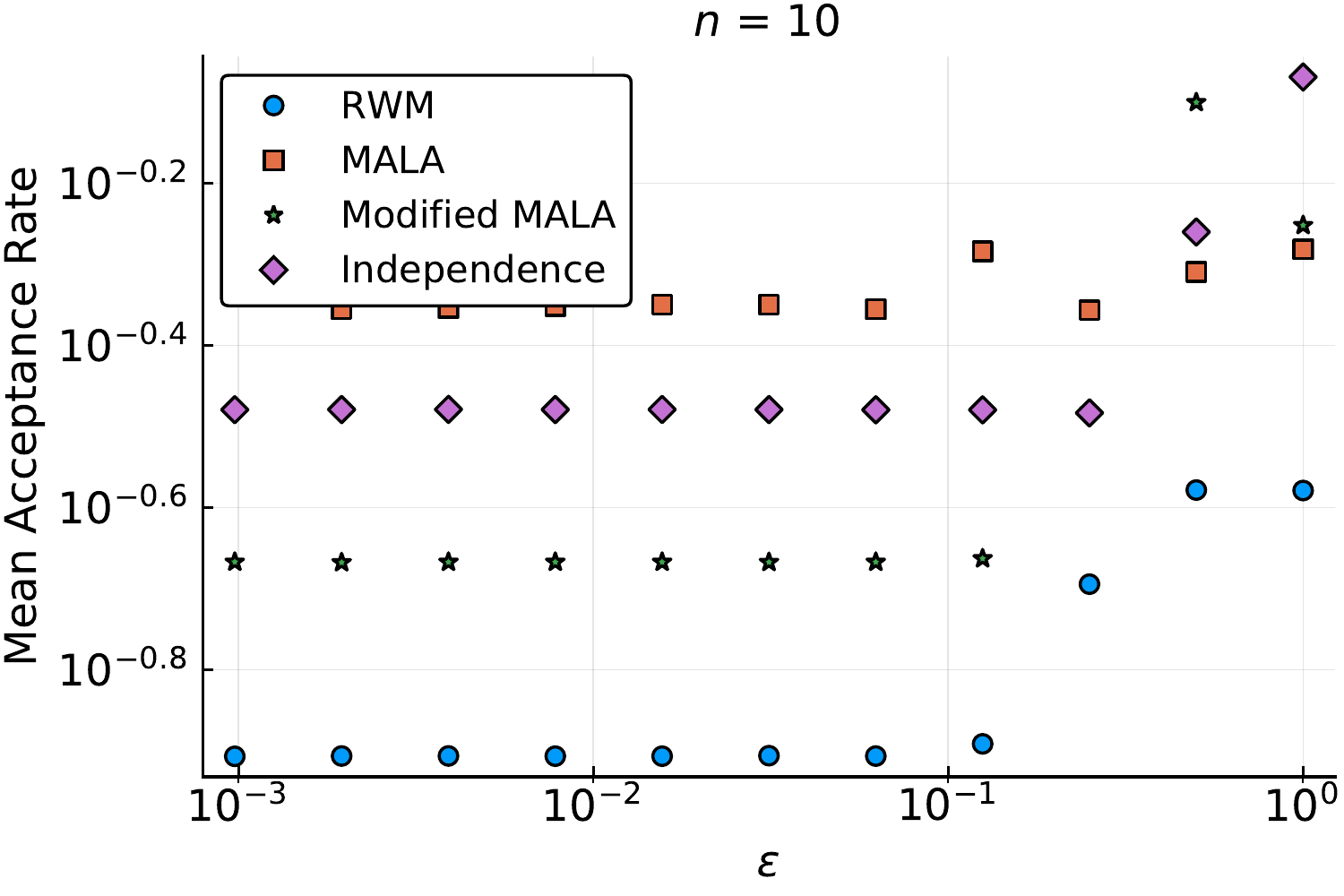}}
    \subfigure[]{\includegraphics[width=6.25cm]{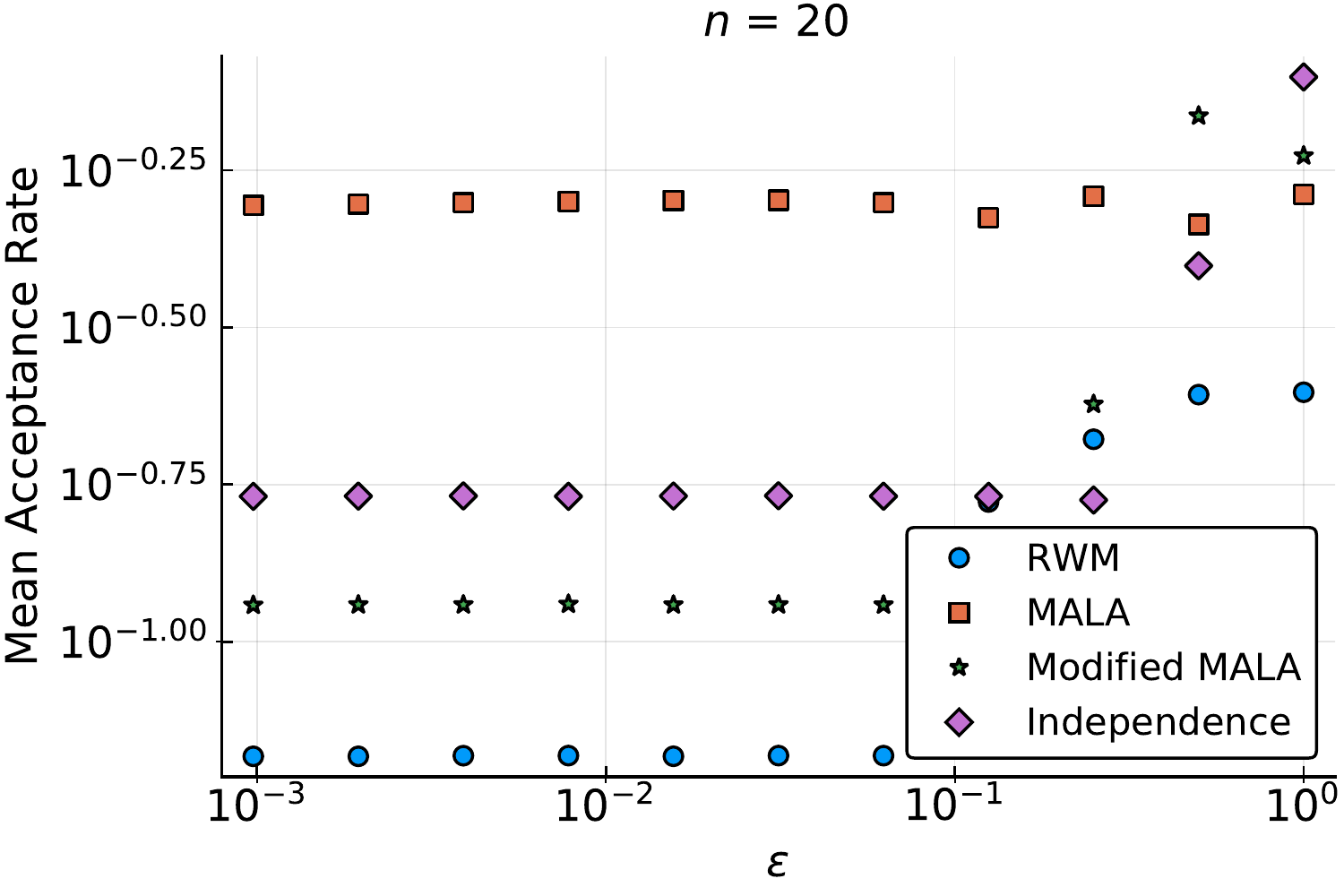}}

    \subfigure[]{\includegraphics[width=6.25cm]{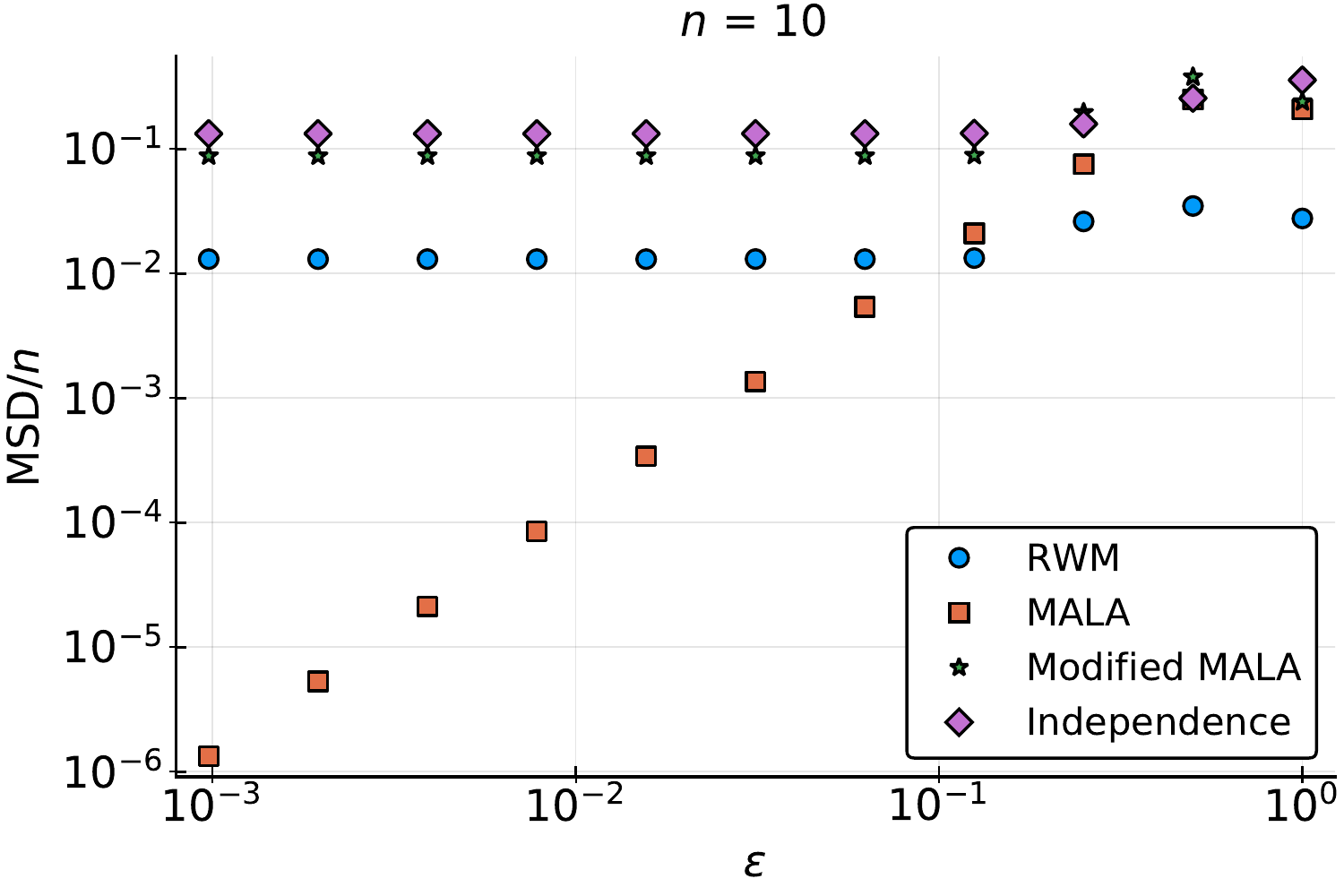}}
    \subfigure[]{\includegraphics[width=6.25cm]{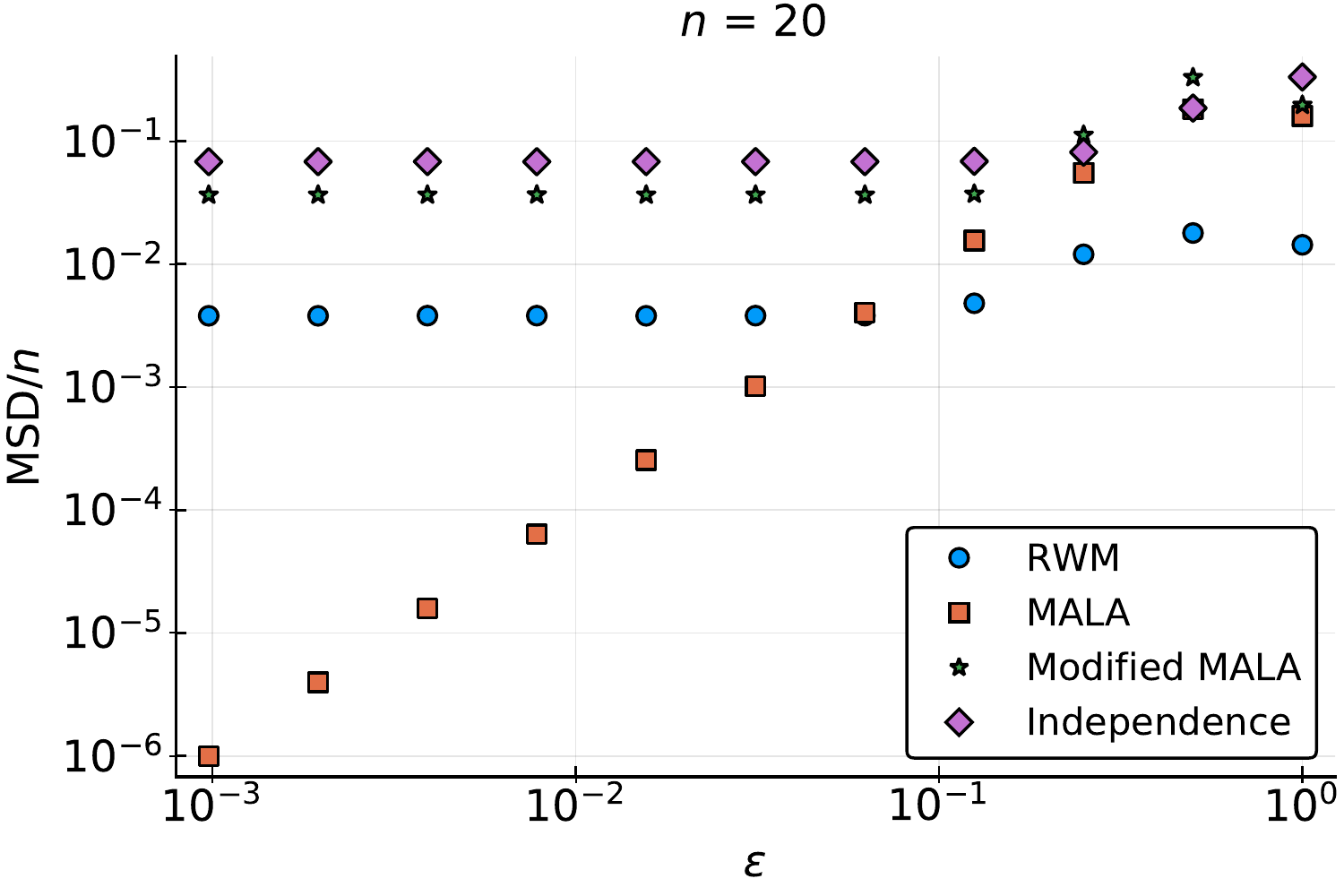}}

    \caption{Comparison of performance for \eqref{e:roughharm} in different dimensions for a range of $\eps$.  Observe that the optimal $\sigma\propto \eps$ , as predicted.  Also note that while MALA maintains a comparatively higher acceptance rate, because its step size is shrinking, the mixing rate, as measured by the MSD vanishes with $\eps$ (color online).}
    \label{fig:roughharmopt}
\end{figure}

\begin{table}
    \centering
    \caption{Ratios of the MSD at optimal $\sigma$ for Modified MALA and  Independence sampler to the RWM for the rough harmonic landscape, \eqref{e:roughharm}.  As $\eps\to 0$, the amplification in performance saturates, but increases with dimension.  The Independence sampler method appears to give better performance in moderate to high dimensions.}
    \label{tab:harmamp}
\centering
    \subtable[Modified MALA Sampler]{\begin{tabular}{l rrrrr}
\hline\hline
$\epsilon$ & $n= 1$ & $n= 10$ & $n= 20$ & $n= 50$ & $n= 100$ \\
\hline
$2^{-5}$ & 2.43 & 6.73 & 9.59 & 15.71 & 4.38 \\
$2^{-6}$ & 2.43 & 6.74 & 9.61 & 15.88 & 15.16 \\
$2^{-7}$ & 2.43 & 6.74 & 9.62 & 15.57 & 20.93 \\
$2^{-8}$ & 2.43 & 6.73 & 9.60 & 15.40 & 25.07 \\
$2^{-9}$ & 2.43 & 6.73 & 9.62 & 15.72 & 20.36 \\
$2^{-10}$ & 2.43 & 6.74 & 9.61 & 15.48 & 22.51 \\
\hline
\hline    \end{tabular}}

    \subtable[Independence Sampler]{\begin{tabular}{l rrrrr}
\hline\hline
$\epsilon$ & $n= 1$ & $n= 10$ & $n= 20$ & $n= 50$ & $n= 100$ \\
\hline
$2^{-5}$ & 2.20 & 10.15 & 17.91 & 40.36 & 12.59 \\
$2^{-6}$ & 2.20 & 10.17 & 17.94 & 39.56 & 57.39 \\
$2^{-7}$ & 2.20 & 10.17 & 17.88 & 40.07 & 70.85 \\
$2^{-8}$ & 2.20 & 10.17 & 17.91 & 39.26 & 67.93 \\
$2^{-9}$ & 2.20 & 10.17 & 17.94 & 39.98 & 67.61 \\
$2^{-10}$ & 2.20 & 10.17 & 17.92 & 39.12 & 73.55 \\
\hline
\end{tabular}}

\end{table}

\subsection{Rough double well potential.}\label{s:double}

We repeat the experiment from Section~\ref{s:harmonic} for a more challenging problem of sampling from the distribution given by the potential
\begin{equation}
    \label{e:roughdouble}
    v_\eps(x_i) = (x_i^2-1)^2 + \tfrac{1}{8}\cos(x_i/\eps)\,,
\end{equation}
with parameters otherwise the same as in the harmonic case. The starting point for these runs is $x_0 = (-1,-1,\ldots,-1)^{T}\in \R^n$. Results here are shown in
Figure~\ref{fig:roughdoubleopt} and Table~\ref{tab:dbleamp}.  These are similar to the experiments for the harmonic potential, but there are some serious differences and some cases where there is little or no performance gain.

First, in comparison to the harmonic problem, the performance amplification of the independence sampler is significantly larger in the double well problem in dimensions $n=10$ and $n=20$.  We attribute this to the ability of the independence sampler to switch between the left and right super basins, a behavior that is entirely absent from the harmonic problem and one which RWM and the Modified MALA method are incapable of reproducing.

The second difference is that both the Modified MALA sampler and the independence sampler have significant performance issues in dimensions $n=50$ and $n=100$  values of $\eps$ including $2^{-7}$ and $2^{-9}$.   Examining the actual trajectories,  in the case of the Modified MALA,
the optimal $\sigma$ is so small that the contribution of $\tfrac{\sigma^2}{2} \nabla V_0$ is negligible.  Consequently, it is trending towards RWM resulting in no performance amplification.  This is why the ratio is unity.  For the independence sampler, the trajectory is stagnant.  This is partially a function of the choice of initial conditions for the trajectories.  If, we instead start the independence sampler at $x_0 = (0,0,\ldots, 0)^T \in \R^n$, instead of Table~\ref{tab:dbleamp} (b), we obtain Table~\ref{tab:dbleamp2}.  This variation requires a more detailed discussion which we provide in  Section~\ref{s:disc}.  Despite these problems in the higher dimension, there is clearly improvement upon RWM in more modest dimension over a broad range of $\eps$.

\begin{figure}

    \subfigure[]{\includegraphics[width=6.25cm]{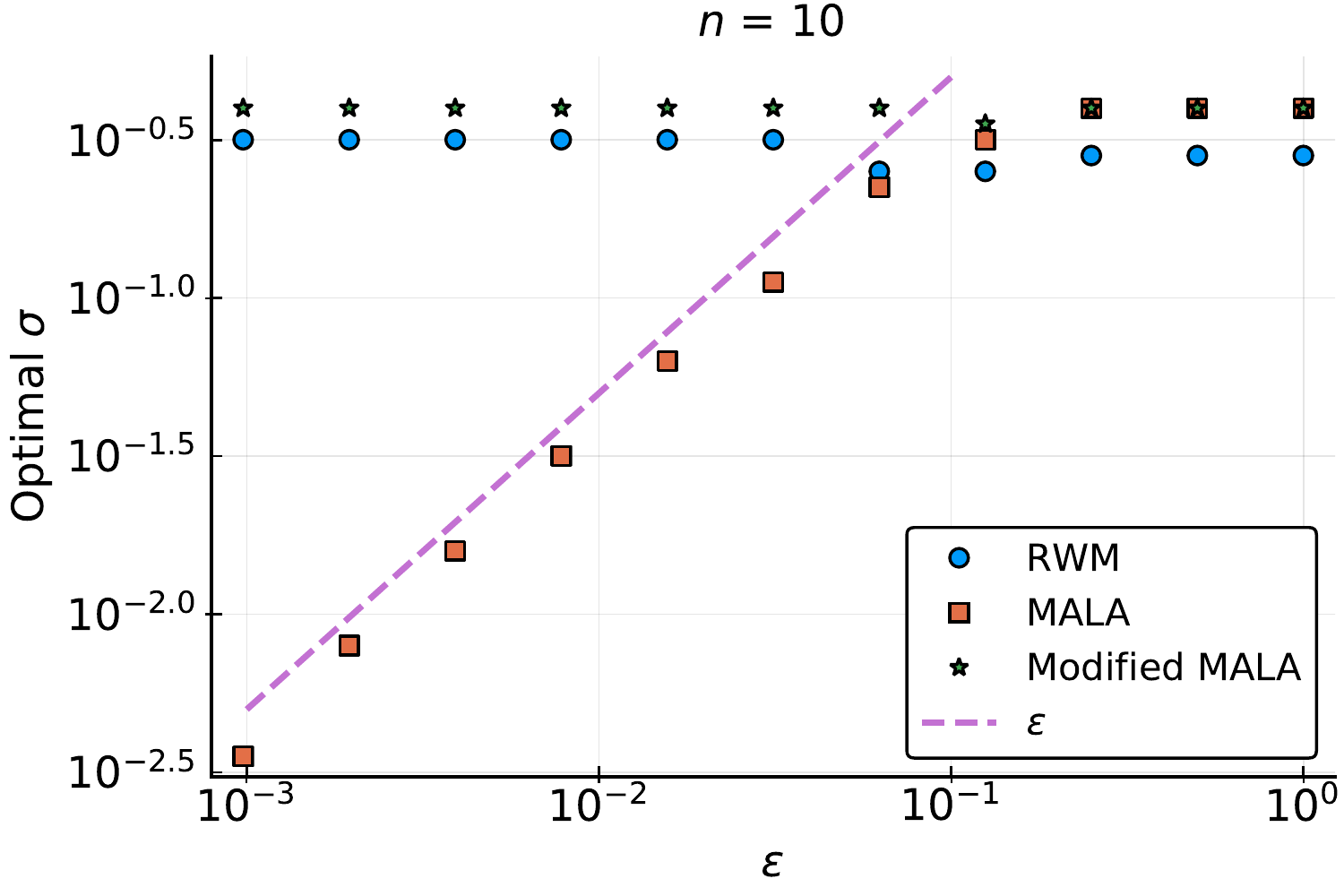}}
    \subfigure[]{\includegraphics[width=6.25cm]{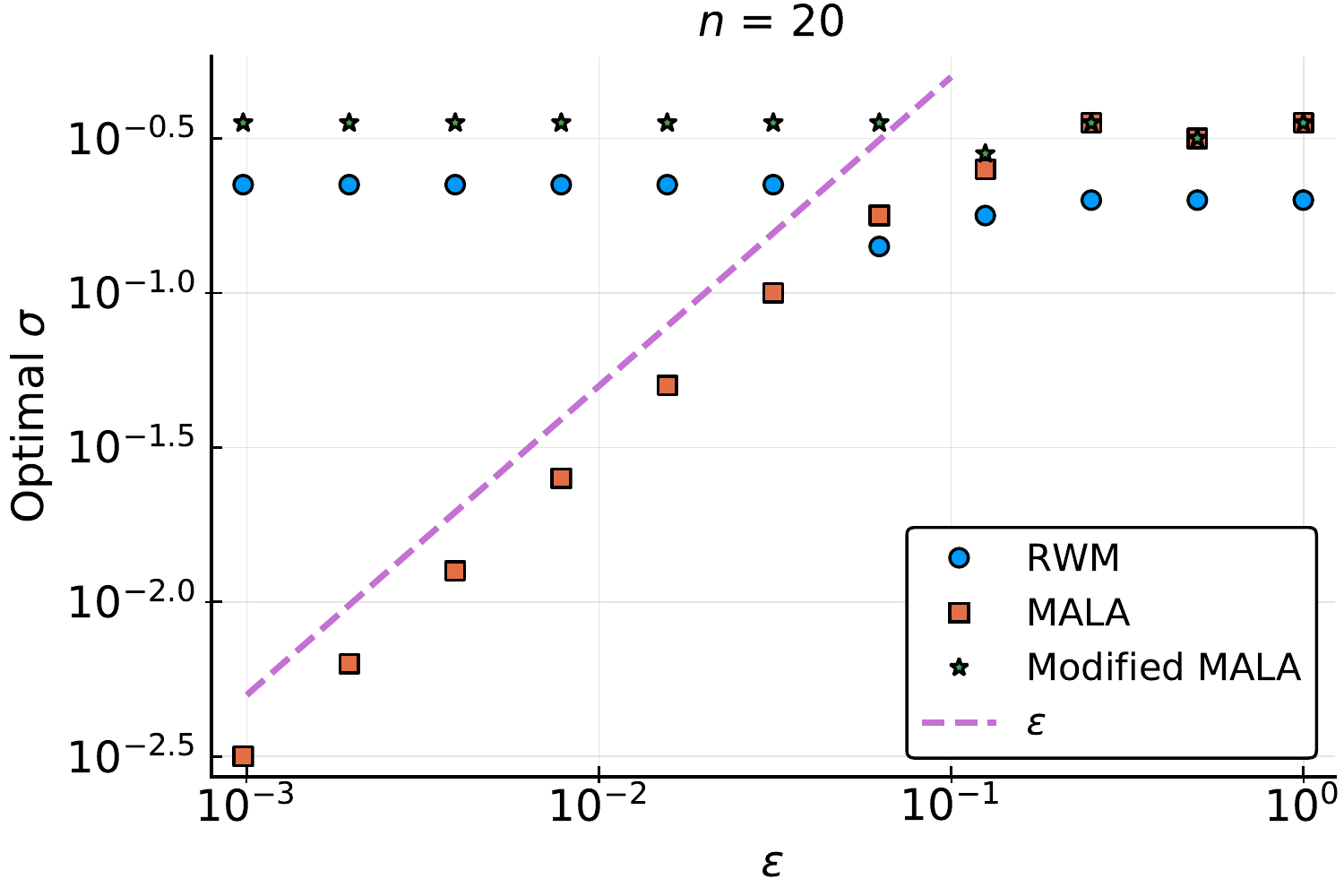}}

    \subfigure[]{\includegraphics[width=6.25cm]{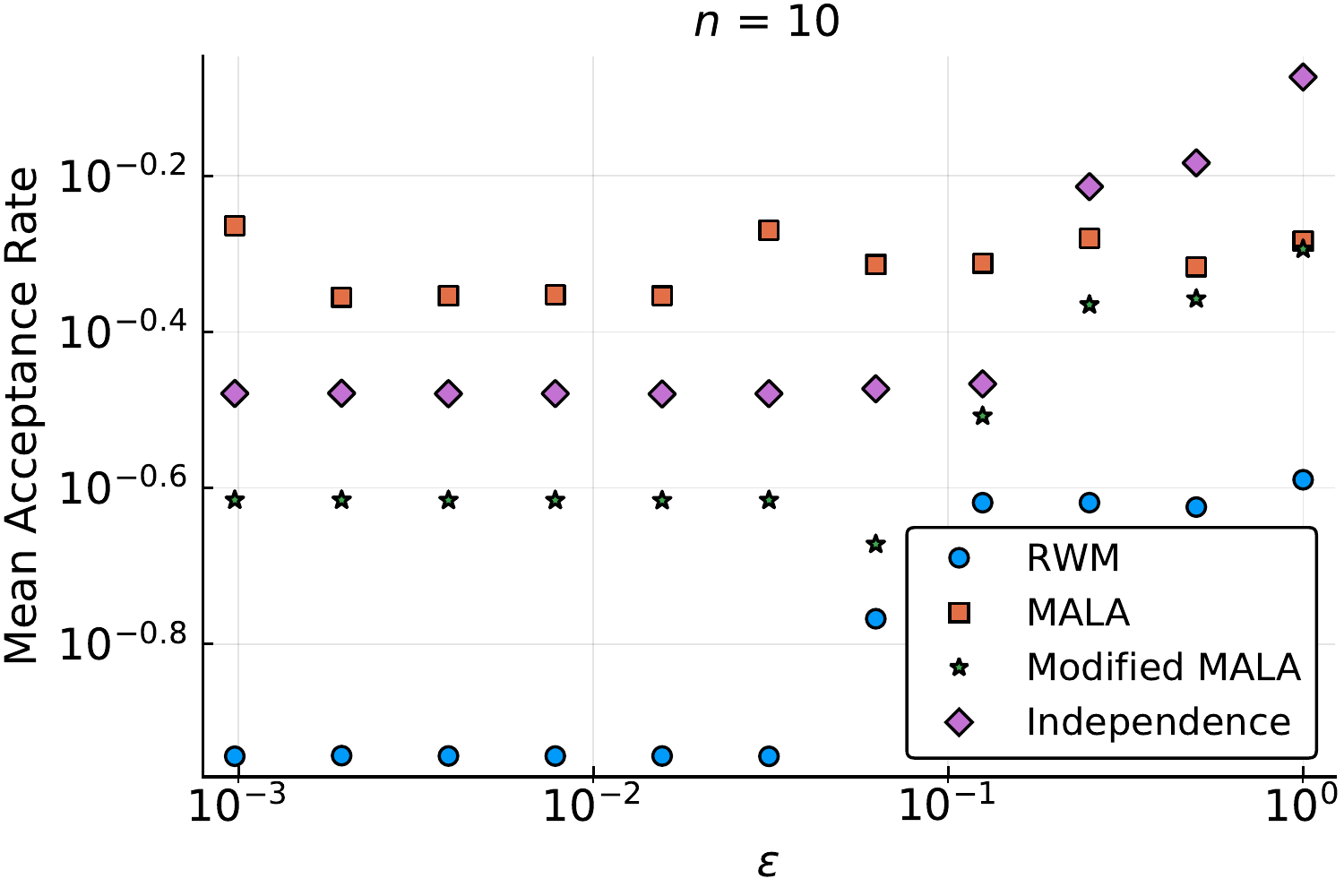}}
    \subfigure[]{\includegraphics[width=6.25cm]{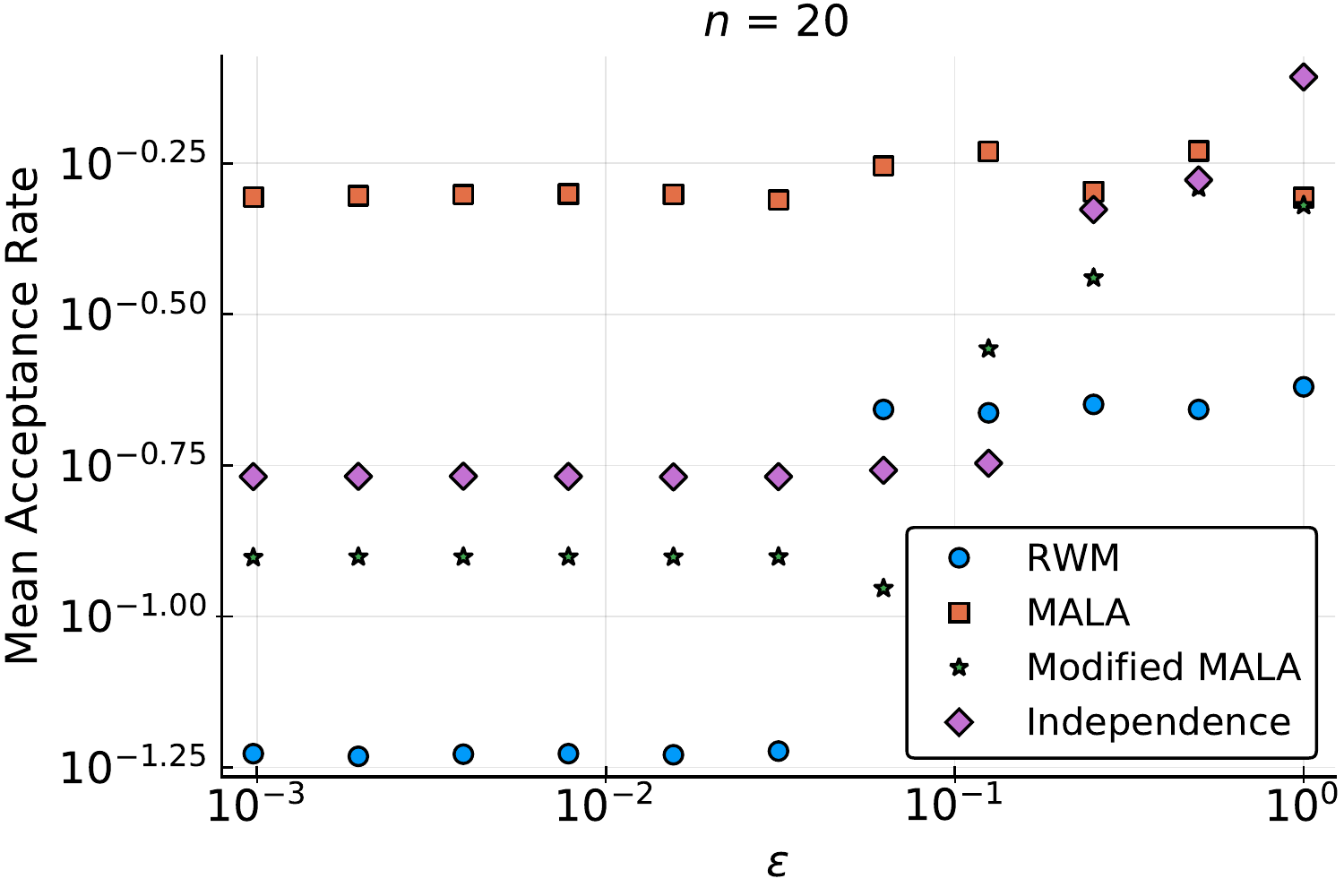}}

    \subfigure[]{\includegraphics[width=6.25cm]{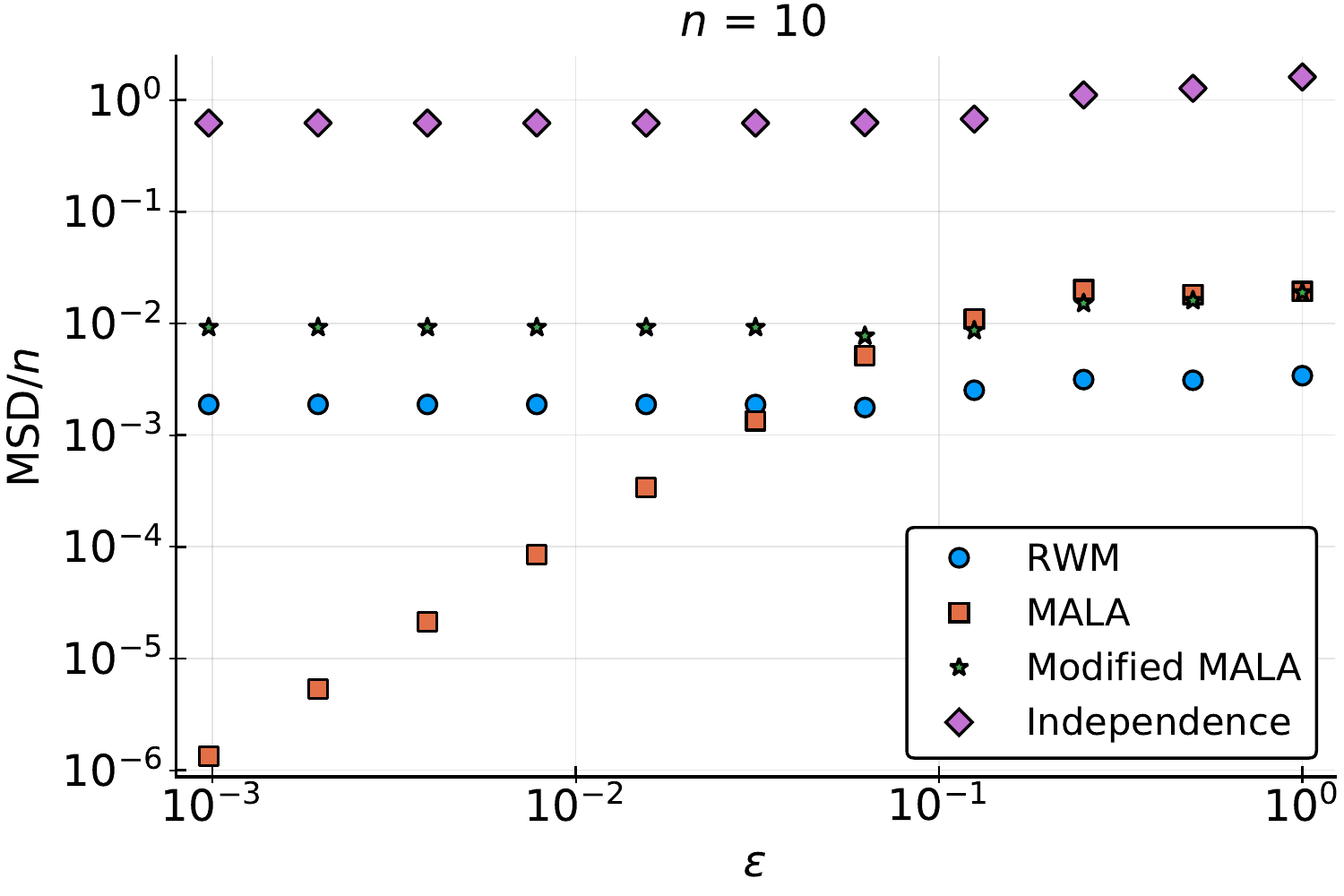}}
    \subfigure[]{\includegraphics[width=6.25cm]{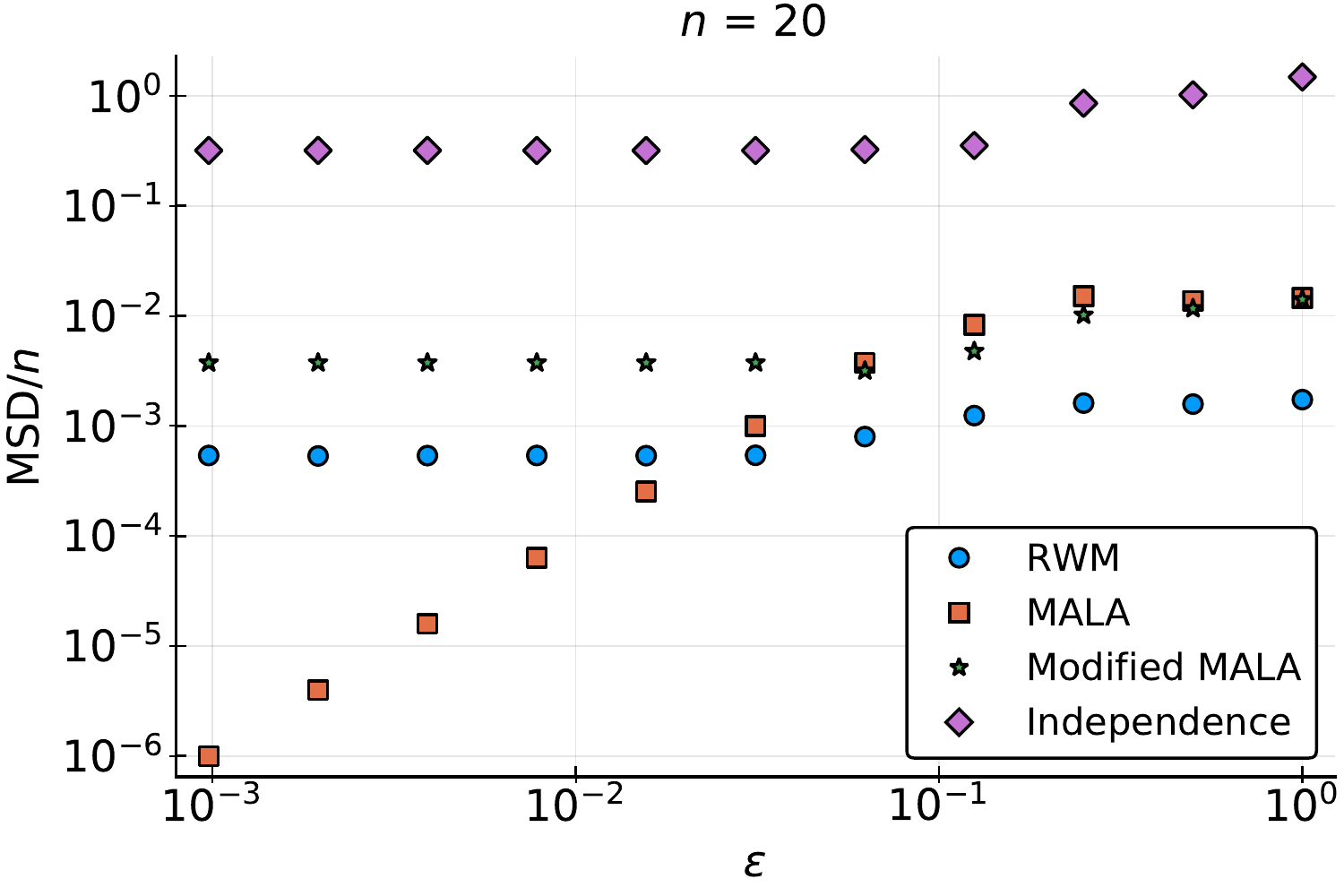}}

    \caption{Comparison of performance for \eqref{e:roughdouble} in different dimensions for a range of $\eps$.  Observe that the optimal $\sigma\propto \eps$, as predicted.  Also note that while MALA maintains a comparatively higher acceptance rate, because its step size is shrinking, the mixing rate, as measured by the MSD vanishes with $\eps$ (color online).}
    \label{fig:roughdoubleopt}
\end{figure}

\begin{table}
    \centering
    \caption{Ratios of the MSD at optimal $\sigma$ for Modified MALA and  Independence sampler to the RWM for the rough double well landscape, \eqref{e:roughdouble}.  As $\eps\to 0$, the amplification in performance saturates, but increases with dimension.  The Independence sampler method appears to give better performance in moderate to high dimensions.  In contrast to the harmonic potential, both the Modified MALA sampler and the Independence sampler show signs of decreased performance in high dimensions.}
    \label{tab:dbleamp}
\centering
    \subtable[Modified MALA Sampler]{\begin{tabular}{l rrrrr}
\hline\hline
$\epsilon$ & $n= 1$ & $n= 10$ & $n= 20$ & $n= 50$ & $n= 100$ \\
\hline
$2^{-5}$ & 0.31 & 4.90 & 6.93 & 3.84 & 1.22 \\
$2^{-6}$ & 0.30 & 4.89 & 6.98 & 11.45 & 1.80 \\
$2^{-7}$ & 0.29 & 4.89 & 6.96 & 5.22 & 1.01 \\
$2^{-8}$ & 0.29 & 4.89 & 6.97 & 11.69 & 15.40 \\
$2^{-9}$ & 0.29 & 4.89 & 7.03 & 1.00 & 1.00 \\
$2^{-10}$ & 0.29 & 4.90 & 6.94 & 11.64 & 17.30 \\
\hline
\end{tabular}}

    \subtable[Independence Sampler]{\begin{tabular}{l rrrrr}
\hline\hline
$\epsilon$ & $n= 1$ & $n= 10$ & $n= 20$ & $n= 50$ & $n= 100$ \\
\hline
$2^{-5}$ & 6.30 & 330.89 & 589.73 & 441.69 & 62.95 \\
$2^{-6}$ & 6.31 & 330.06 & 594.21 & 1317.64 & 262.12 \\
$2^{-7}$ & 6.31 & 330.42 & 592.44 & 434.52 & 0.00 \\
$2^{-8}$ & 6.32 & 330.32 & 594.40 & 1343.84 & 2150.46 \\
$2^{-9}$ & 6.32 & 330.32 & 598.89 & 0.00 & 0.00 \\
$2^{-10}$ & 6.31 & 330.81 & 592.33 & 1335.96 & 2287.34 \\
\hline

\end{tabular}}

\end{table}

\begin{table}
    \centering
    \caption{Ratios of the MSD at optimal $\sigma$ for Modified MALA and  Independence Sampler to the RWM for the rough double well landscape, \eqref{e:roughdouble} with $x_0 = (0,0,\ldots, 0)^T\in \R^n$.  Compare with Table~\ref{tab:dbleamp}.}
    \label{tab:dbleamp2}
\centering
    \subtable[Modified MALA Sampler]{\begin{tabular}{l rrrrr}
\hline\hline
$\epsilon$ & $n= 1$ & $n= 10$ & $n= 20$ & $n= 50$ & $n= 100$ \\
\hline
$2^{-5}$ & 0.31 & 4.91 & 6.95 & 3.88 & 1.22 \\
$2^{-6}$ & 0.30 & 4.89 & 6.99 & 11.49 & 1.73 \\
$2^{-7}$ & 0.29 & 4.89 & 6.98 & 11.54 & 7.48 \\
$2^{-8}$ & 0.29 & 4.89 & 6.98 & 11.43 & 16.84 \\
$2^{-9}$ & 0.29 & 4.88 & 6.99 & 11.67 & 17.11 \\
$2^{-10}$ & 0.29 & 4.88 & 6.96 & 11.60 & 15.67 \\
\hline
\end{tabular}}

    \subtable[Independence Sampler]{\begin{tabular}{l rrrrr}
\hline\hline
$\epsilon$ & $n= 1$ & $n= 10$ & $n= 20$ & $n= 50$ & $n= 100$ \\
\hline
$2^{-5}$ & 6.31 & 331.09 & 591.19 & 439.82 & 72.42 \\
$2^{-6}$ & 6.31 & 330.56 & 594.66 & 1339.65 & 191.19 \\
$2^{-7}$ & 6.32 & 330.24 & 594.54 & 1337.39 & 976.85 \\
$2^{-8}$ & 6.32 & 330.13 & 594.19 & 1311.59 & 2562.82 \\
$2^{-9}$ & 6.32 & 329.43 & 595.04 & 1355.55 & 2566.64 \\
$2^{-10}$ & 6.31 & 329.40 & 593.38 & 1325.89 & 2199.09 \\
\hline

\end{tabular}}

\end{table}

In Figure~\ref{f:paths2}, we repeat the experiment from Figure~\ref{f:paths1} with $n=1$, but for the Modified MALA and the Independence sampler schemes.  Computing at $\sigma=1$, we see better mixing than in Figure~\ref{f:paths1}, for the same rough energy landscape.

\begin{figure}

\subfigure[Modified MALA]{\includegraphics[width=6.25cm]{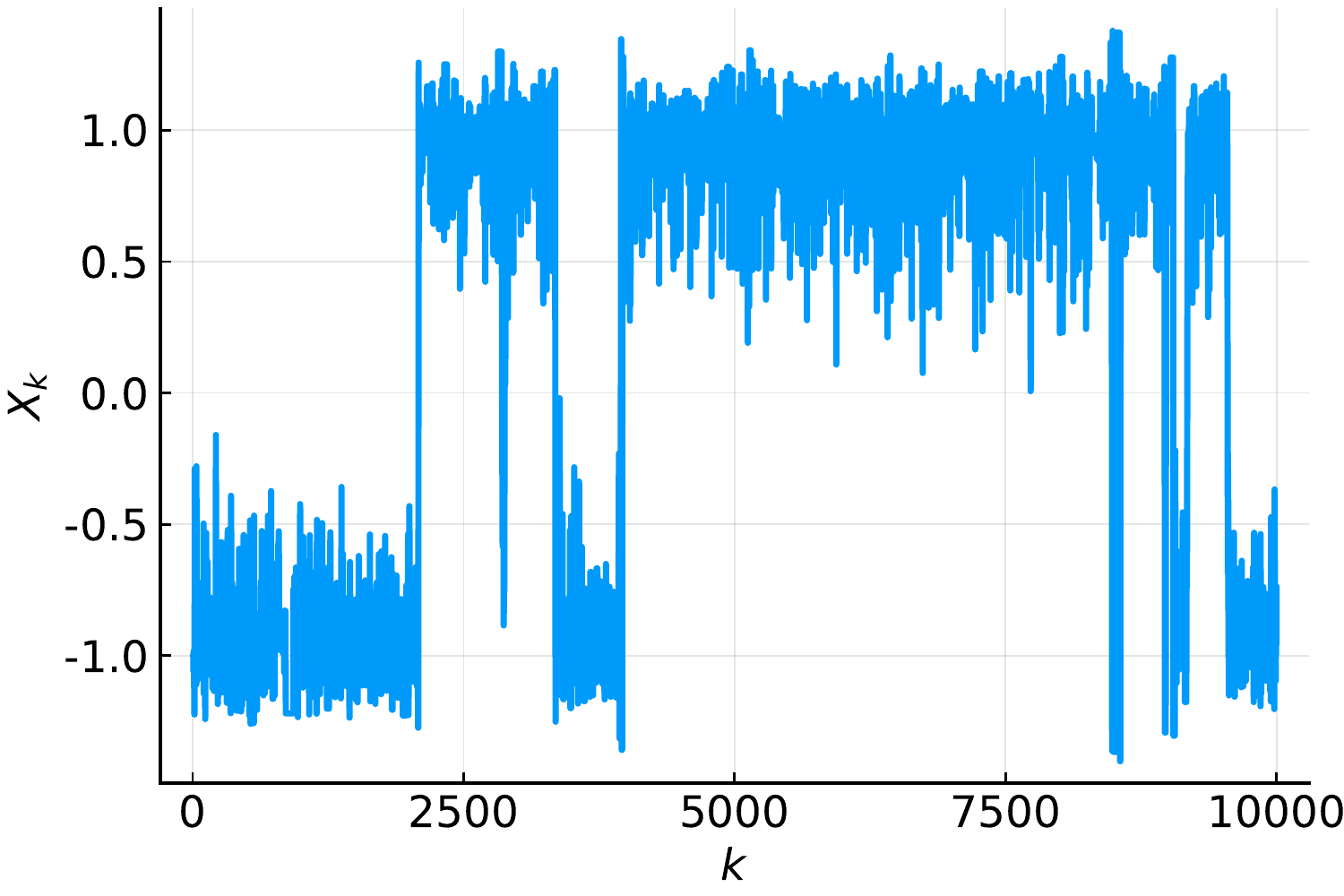}}
\subfigure[Independence]{\includegraphics[width=6.25cm]{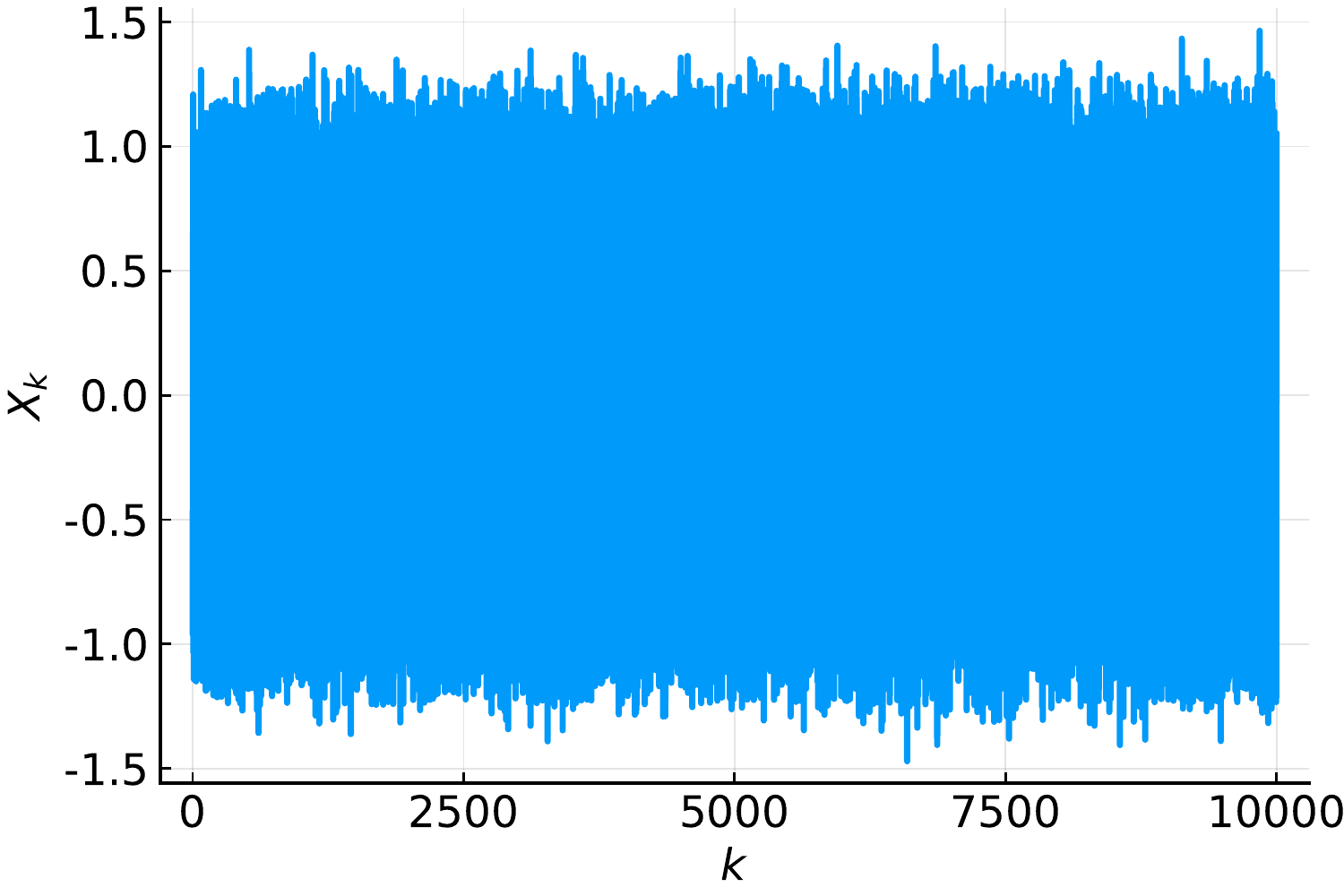}}

\caption{Sample paths corresponding to the energy landscapes in Figure~\ref{f:rough1}(b).  These were generated using MALA with $\sigma=1$.  Compare with standard MALA sampling in Figure~\ref{f:paths1}(b)
(color online).}
\label{f:paths2}
\end{figure}

\subsection{Results in dimension one.}

We briefly consider the behavior in dimension one for the harmonic potential and the double well.  For both energy landscapes, as shown in Figure~\ref{fig:opt_d1}, the optimal $\sigma\propto \sqrt{\eps}$, not $\sigma\propto \eps$.  Thus, while the performance of MALA will also degrade in $n=1$, a different scaling appears here.  This is consistent with the  harmonic problem analyzed in Section~\ref{s:oned}.

\begin{figure}

    \subfigure[Harmonic Potential]{\includegraphics[width=6.25cm]{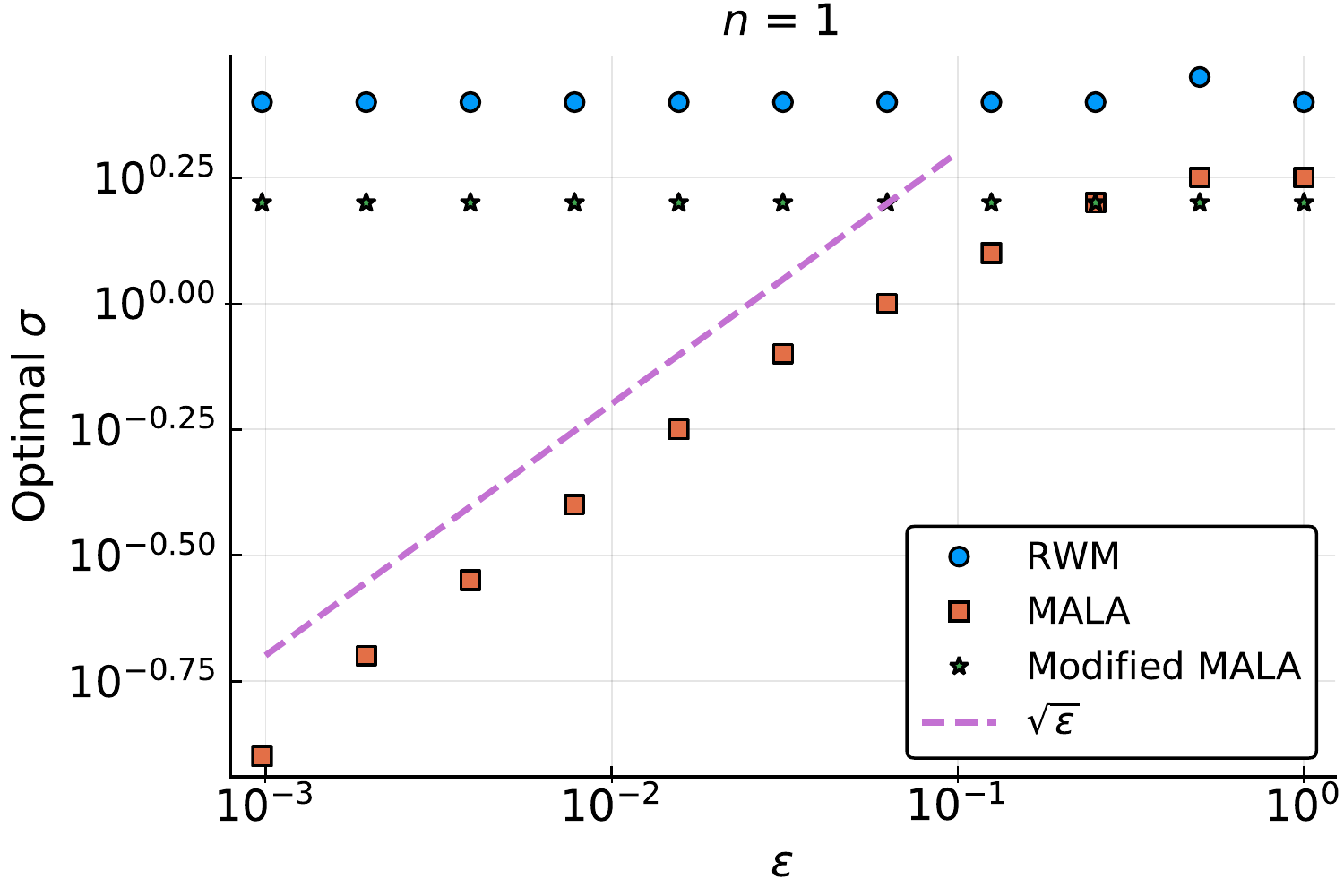}}
    \subfigure[Double Well Potential]{\includegraphics[width=6.25cm]{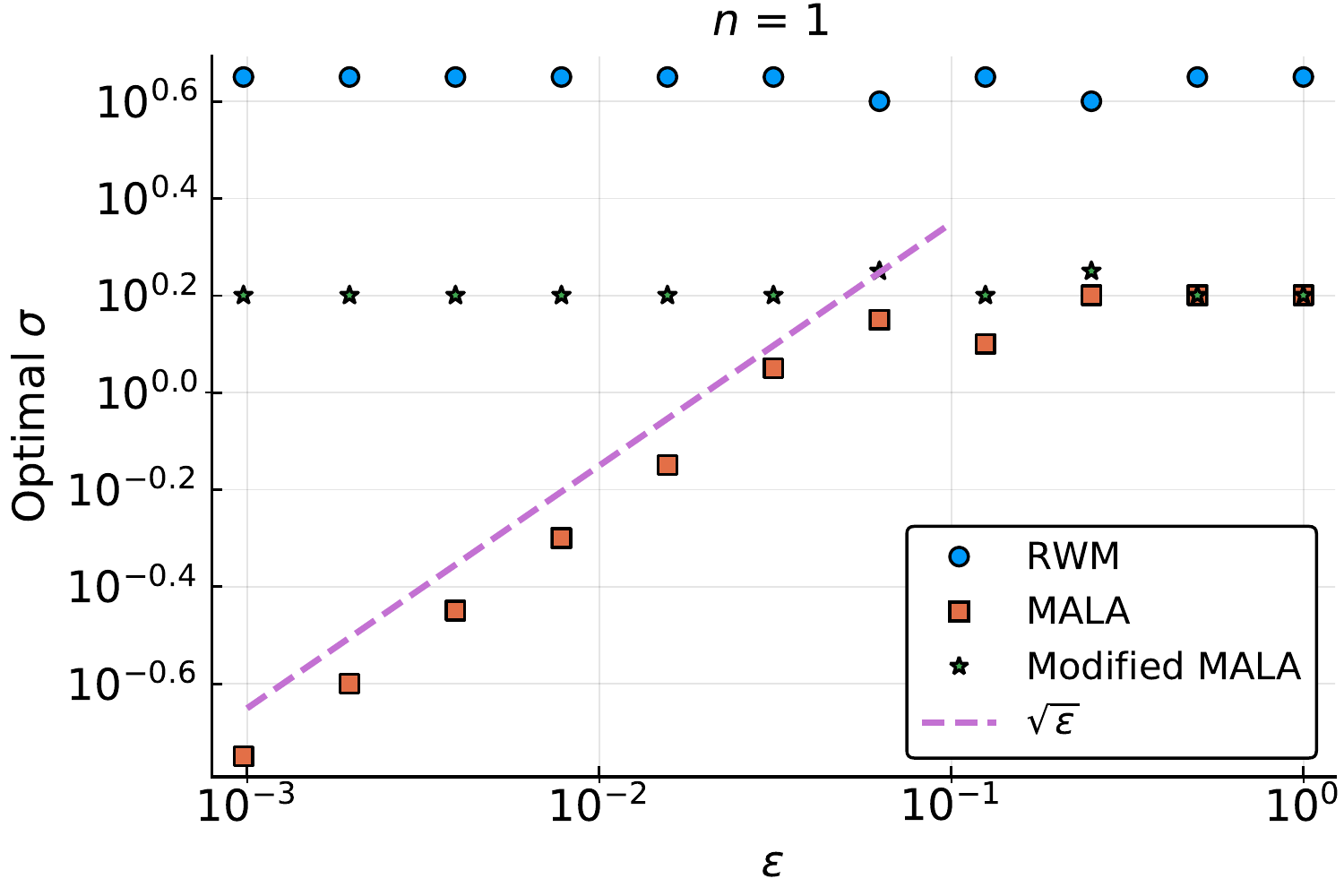}}
    \caption{Optimal $\sigma$ values for the rough harmonic and double well potentials in $n=1$.  In contrast to results in higher dimensions
    (see Figures~\ref{fig:roughharmopt} and \ref{fig:roughdoubleopt}), for MALA, $\sigma\propto \sqrt{\eps}$ for $n=1$ (color online).}
    \label{fig:opt_d1}
\end{figure}

\subsection{Local entropy approximations.}
\label{s:locent_numerics}

We briefly consider the possibility of using local entropy, introduced in Section~\ref{s:locent} with \eqref{e:Vgamma} and \eqref{e:gradVgamma}.  This approach may be of use in problems where no straightforward scale separation, of the type found in \eqref{e:V1}, is present in the energy landscape.  As motivation we consider the potential in dimension one
\begin{equation}
\label{e:randomdble}
    V(x) = \underbrace{(x^2-1)^2}_{V_0(x)} + \sum_{j=1}^M c_j \cos(k_j x)\,,
\end{equation}
where the $c_j$ and $k_j$ are random, from a particular distribution.  Indeed, taking $M=10$ and $c_j \sim U(-.1, .1)$, $\log(k_j)\sim U(10^1,10^3)$, we obtain the landscape shown in Figure~\ref{fig:randomdble}, along with the leading contribution, $V_0$, and $V_\gamma$ obtained through a numerical quadrature at $\beta =5 $ and $\gamma = 0.05$.  Clearly, the local entropy approximation eliminates the fine scale roughness found in the original potential.

In general, Monte Carlo approximations of $V_\gamma$ and $\nabla V_\gamma$ will be needed, as a quadrature will be impractical in high dimensions.  An example of this is shown in Figure~\ref{fig:randomdble_mc}.
In Figure~\ref{fig:randomdble_mc}(a), we compare $V_0$, $V_\gamma$ and a Monte Carlo estimate of $V_\gamma$ computed using \eqref{e:vgamma_mc} with $\NS_{\mathrm{s}}=10^2$.

In Figure~\ref{fig:randomdble_mc}(b), we compare $\nabla V_0$, $\nabla V_\gamma$, and a Monte Carlo estimate of $V_\gamma$, computed using \eqref{e:gradvgamma_mc}.  In this latter figure we take  $\NS_{\mathrm{s}} = 10^4$, and each sample is obtained by taking $\NS_{\rm \delta t} =4$ time steps with $\delta t=1$ in a variant of MALA that exactly linearly integrates the Ornstein-Uhlenbeck component of \eqref{e:local_sde}.  Obviously, there are many options for how  Monte Carlo estimates of $V_\gamma$ and $\nabla V_\gamma$ can be obtained.  We will return to this in the discussion.

\begin{figure}

    \subfigure[]{\includegraphics[width=6.25cm]{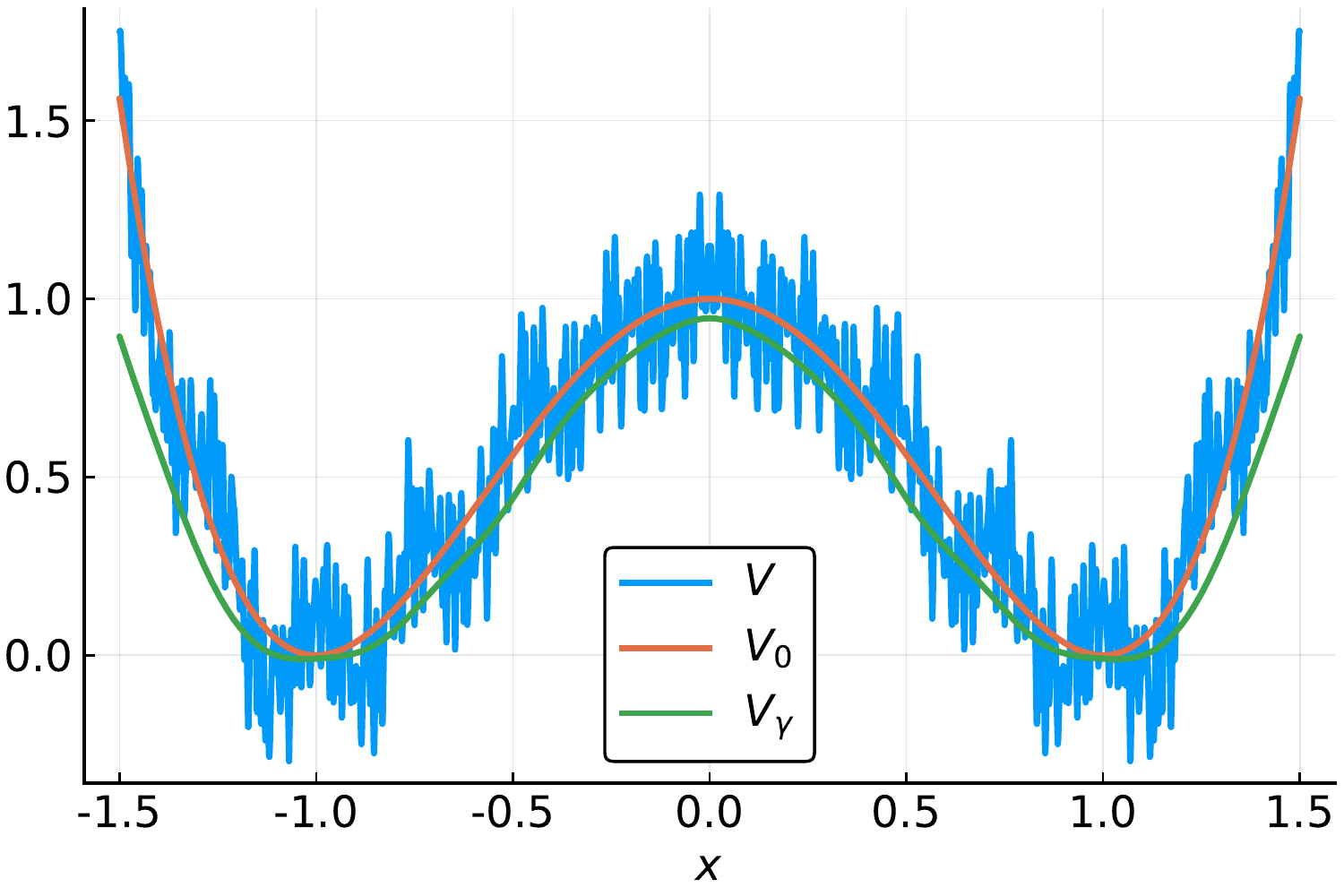}}
    \subfigure[]{\includegraphics[width=6.25cm]{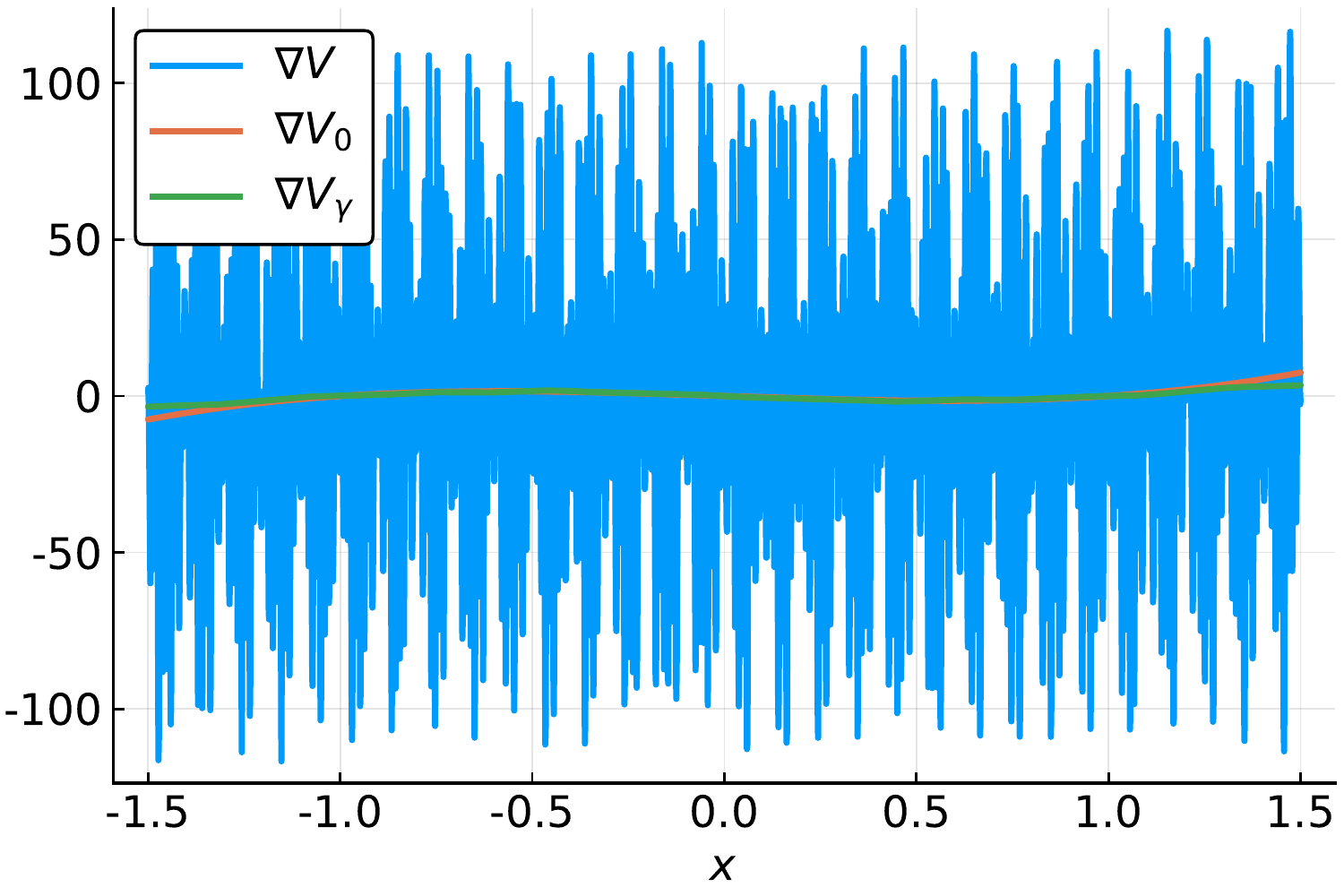}}
    \caption{A particular realization of the energy landscape given by \eqref{e:randomdble}, along with its leading order, long range component $V_0(x) = (x^2-1)^2$, and the local entropy approximation, $V_\gamma$, computed using \eqref{e:Vgamma} with $\beta=5$ and $\gamma = 0.05$ (color online).}
    \label{fig:randomdble}
\end{figure}

\begin{figure}
    \subfigure[]{\includegraphics[width=6.25cm]{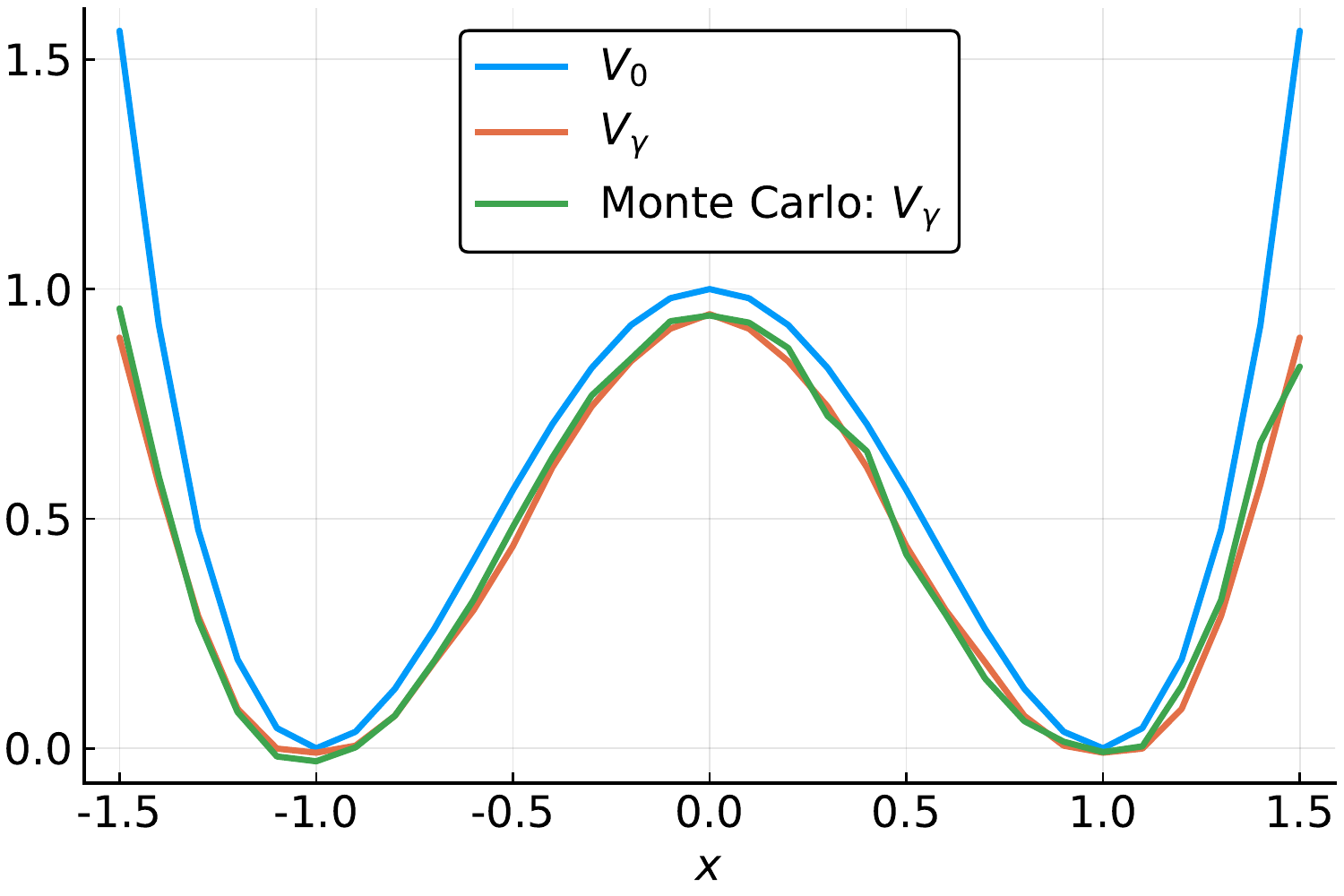}}
    \subfigure[]{\includegraphics[width=6.25cm]{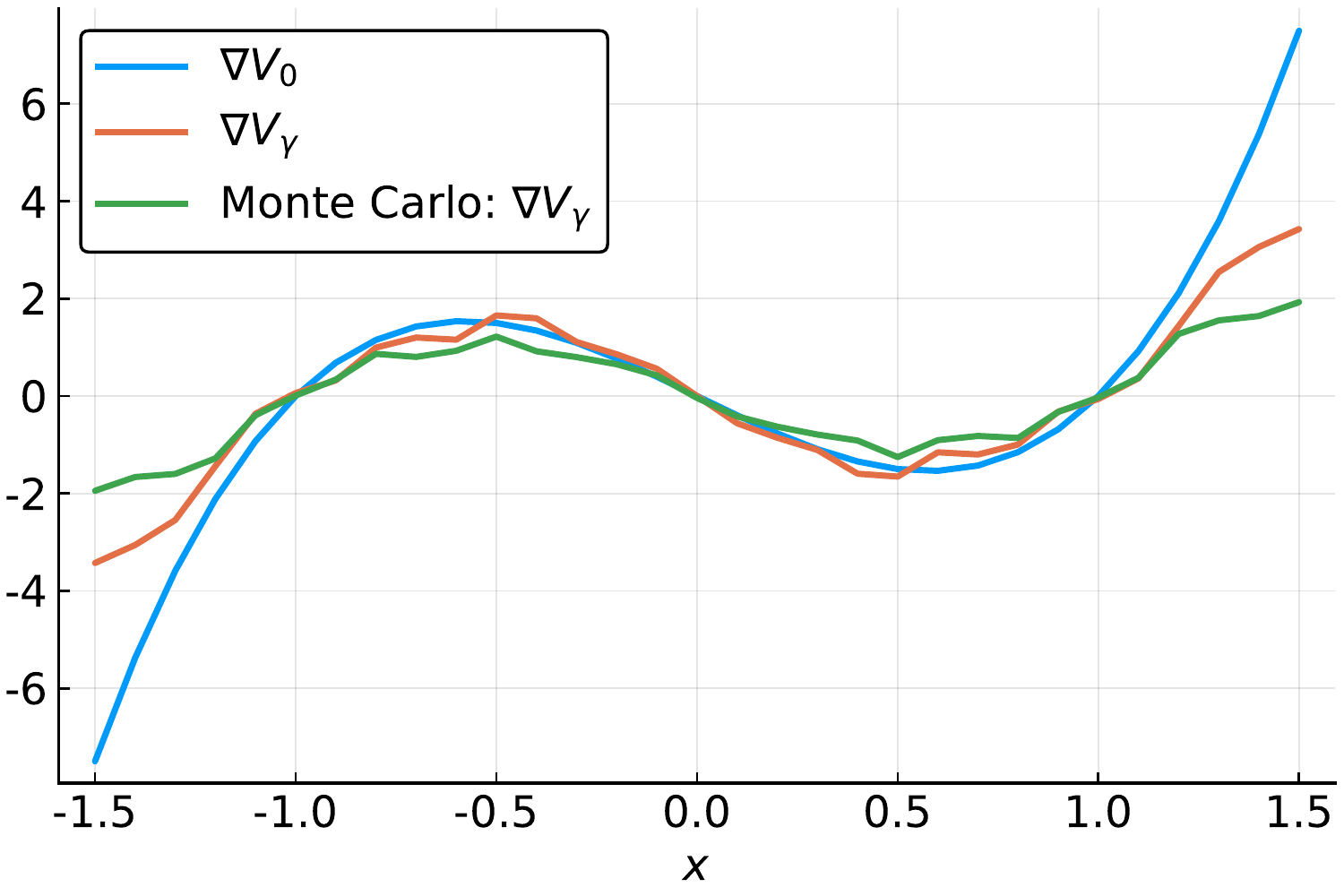}}
    \caption{Monte Carlo estimates of $V_\gamma$ and $\nabla V_\gamma$ for the landscapes in
    Figure~\ref{fig:randomdble} (color online).}
    \label{fig:randomdble_mc}
\end{figure}

\section{Discussion.}\label{s:disc}

We have examined a class of rough energy landscapes where the performance of MALA can be driven to zero at a fixed dimension.  {When $\sigma$ is inadequately scaled with $\eps$, Proposition~\ref{p:malascaling} and
Theorem~\ref{t:malascaling2} reveal that MALA will fail to be globally robust, \eqref{e:Globalrobustness}.  Even if $\sigma$ is optimally scaled, according to the empirical estimates, the numerical simulations indicate that it also suffers.  }

{There are several outstanding questions on MALA that merit investigation.  First, it would be desirable to develop a rigorous understanding of why the optimal $\sigma$ scaling is $\sigma\propto \sqrt{\eps}$ in dimension one, while it is $\sigma\propto \eps$ for $n$ sufficiently large.  Next, there is an analysis of why the spectral gap collapses at the optimal scaling as $\eps\to 0$; recall that our result in Theorem~\ref{t:malascaling2} does not apply to either scaling. Another,  related, challenge is how to quantify the size of the small gradient set, \eqref{e:smallset}, outside of the separable case, as $\eps\to 0$, or, alternatively, how to avoid analyzing the small set.  Separately, there is the question of how $\sigma$ should be scaled for MALA before it has reached stationarity.}

We have also demonstrated that RWM along with the Modified MALA and the Independence sampler are insensitive to the roughness of the landscape and are globally robust.  While modified MALA and the Independence sampler  require a smoothed energy landscape for proposals, RWM is a viable option without additional information.  Indeed, Corollary~\ref{cor:Teps} tells us that for rough energy landscapes like \eqref{e:V1}, with  roughness bounded uniformly in $\eps$, if the proposal of the MCMC scheme is $\eps$-independent, then the performance will be insensitive to $\eps$.  Corollary~\ref{cor:Teps2} indicates that other methods that have a weak $\eps$-dependence, such as the tamed MALA, will also exhibit robustness to the roughness.

As the results in Section~\ref{s:numerics} show, there still appear to be challenges in high dimensions.  The degradation of the Independence sampler shown in Table \ref{tab:dbleamp} is a consequence of an unusually difficult starting point,  $x_0 = (-1,-1,\ldots,-1)^{T}\in \R^n$, as the following calculation shows.

 For $\eps=2^{-7}$, $\cos(-1/\eps) = -0.69$ and at $\eps=2^{-9}$, $\cos(-1/\eps) = -1.00$.  Consequently, every single component of the initial guess is close to a global minimum of the energy landscape and almost every proposal will be at a position in state space in which every coordinate has higher energy.  Starting at any $X_0 = x_0$,
\begin{equation}
\begin{split}
    \E[|X_{1} - X_0|^2]/n &= \E[|X_{1} - X_0|^2]/n = \E[|X_{1}^{\rm p} - X_0|^2 1 \wedge e^{R(X_0,X_1^{\rm p})}]/n \\
    &\leq \sqrt{\E[|X_{1}^{\rm p} - X_0|^4/n^2} \sqrt{\E[ 1 \wedge e^{R(X_0,X_1^{\rm p})}]}\,.
    \end{split}
\end{equation}
Under our assumptions on $U = V_0$, the fourth moment term is bounded by a constant independent of both $n$ and $\eps$.  In the above expressions the expectation is over the proposal.  The exponent in the acceptance probability is
\begin{equation}
    R(X_0,X_1^{\rm p}) = \underbrace{R_n(x,y)}_{x = X_0, y = X_1^{\rm p}} = \sum_{i=1}^n r(x_i,y_i) = \sum_{i=1}^n r_i\,.
\end{equation}
With $x_1 = x_2=\ldots x_n = -1$, the $r_i$'s are thus i.i.d.  Furthermore, since
\begin{equation}
    r(x,y) =  v_1(x/\eps) - v_1(y/\eps) \Rightarrow |r(x,y)|\leq  \osc v_1<\infty\,.
\end{equation}
Suppose $\mu_r \equiv \E[r_i]<0$; $\E[R_n] = n \mu_r$.  Then using Hoeffding's inequality, \cite{vershynin2018high},
\begin{equation}
\label{e:accept_bound}
\begin{split}
    \E[{1 \wedge e^{R_n(x,y)}}] &= \E[1 \wedge e^{R_n(x,y)} 1_{R_n-n\mu_r< n|\mu_r|/2}] +\E[1 \wedge e^{R_n(x,y)} 1_{R_n-n\mu_r\geq n|\mu_r|/2}]\\
    &\leq e^{-n|\mu_r|/2} + \P(R_n-n\mu_r\geq n|\mu_r|/2)\\
    &<  e^{-n|\mu_r|/2}  + \exp \paren{-\frac{n}{8}\frac{|\mu_r|^2}{(\osc v_1)^2} }\,.
\end{split}
\end{equation}
Consequently the mean acceptance and the sampling performance will both be driven to zero exponentially fast as $n\to \infty$ once $\mu_r < 0$.  Checking this numerically with $x_1=-1$, $\beta=5$, $\eps =2^{-9}$, and $v_0 = \beta \cdot (x^2-1)^2$, $v_1 = \beta \cdot \tfrac{1}{8}\cos(x)$,
\begin{equation}
\begin{split}
    \mu_r &=   v_1(x_1/\eps) - \E[v_1(y_1/\eps)] =  v_1(x_1/\eps) -  \int v_1(y_1/\eps) z_0^{-1}e^{- v_0(y_1)}dy_1\\
    &=-0.623\,.
\end{split}
\end{equation}
In fact, due to the oscillatory nature of the integral, $\mu_r\approx v_1(x_1/\eps)$.  With such a value of $\mu_r$, it is straightforward to see, from \eqref{e:accept_bound}, that the performance will rapidly be driven to zero for large enough $n$.  In contrast, in Table~\ref{tab:dbleamp2}, we see better results when $x_1=0$.  Repeating the above computation, $\mu_r = 0.625>0$, making it resistant to the previous pathology.

This phenomenon does not plague RWM because, in contrast to the independence sampler, RWM has a parameter, $\sigma$, that can be tuned to maintain an $\bigo(1)$ mean acceptance probability as $n\to \infty$.  Indeed, the scaling $\sigma^2 =
\bigo(n^{-1})$ for RWM, discussed in Section~\ref{s:highd}, ensures this.  We conjecture that a Metropolis within Gibbs sampler, by which only a subset of the $n$ coordinates are altered in each step of the sampler, will alleviate this problem in the Independence sampler.

Generating the samples from the smooth distribution for the Independence sampler is also a challenge in the general case.  We believe this can be easily  accomplished using MALA or HMC.  These other samplers should be well behaved on the smooth energy landscape, allowing for the straightforward construction of approximately independent proposals for the rough landscape.

Smoothed energy landscapes might be available through a known decomposition like \eqref{e:V1}.  We also conjecture that the optimal choice of approximate landscapes for potentials like \eqref{e:V1} corresponds to the homogenized energy landscapes discussed in \cite{duncan2016noiseinduced,arous2003multiscale,owhadi2003anomalous}.  Unfortunately, computing such smoothed energies requires solving an elliptic PDE in a space of the same dimension as the considered state space; a Monte Carlo estimator of the solution may partially overcome this difficulty.  Alternatively, physical knowledge of the system may motivate some choice for a surrogate smoothed landscape.

When these options are not available, the local entropy approximation is another possibility.  The challenge to using local entropy, which we do not further develop here, is that unless the problem is in a very low dimension, auxiliary sampling algorithms must be formulated and tuned  to first estimate $V_{\gamma}$ and $\nabla V_{\gamma}$.  This task would involve determining a sample size, a sampling strategy, and some form of parallelization in order to outperform  simpler alternatives like RWM.

Finally, the weakness of MALA in the presence of roughness can be seen as the MCMC manifestation of stiffness.  We conjecture that it is a generic problem in gradient based MCMC methods, including HMC.  Indeed, the magnitude of $\nabla V$ in HMC will constrain the time step of, for instance, the Verlet method used in the Hamiltonian flow subroutine.  Thus, the number of force calls per HMC step will tend to increase with roughness degrading the overall performance. This has been partially addressed in \cite{Livingstone1908}, where the authors considered potentials of the form \eqref{e:Vmismatch} and rigorously established that HMC suffers from scaling issues.

\clearpage
\bibliographystyle{elsarticle-num}
\bibliography{refs}

\begin{thebibliography}{10}
\expandafter\ifx\csname url\endcsname\relax
  \def\url#1{\texttt{#1}}\fi
\expandafter\ifx\csname urlprefix\endcsname\relax\def\urlprefix{URL }\fi
\expandafter\ifx\csname href\endcsname\relax
  \def\href#1#2{#2} \def\path#1{#1}\fi

\bibitem{pollak2008}
E.~Pollak, A.~Auerbach, P.~Talkner, {Observations on Rate Theory for Rugged
  Energy Landscapes}, Biophysical Journal 95 (2008) 4258--4265.

\bibitem{Hu:2018977}
M.~Hu, J.-D. Bao, {Diffusion crossing over a barrier in a random rough
  metastable potential}, Physical Review E 97 (2018) 062143.

\bibitem{duncan2016noiseinduced}
A.~Duncan, S.~Kalliadasis, G.~Pavliotis, M.~Pradas, Noise-induced transitions
  in rugged energy landscapes, Physical Review E 94~(3) (2016) 032107.

\bibitem{arous2003multiscale}
G.~B. Arous, H.~Owhadi, Multiscale homogenization with bounded ratios and
  anomalous slow diffusion, Communications on Pure and Applied Mathematics
  56~(1) (2003) 80--113.

\bibitem{owhadi2003anomalous}
H.~Owhadi, Anomalous slow diffusion from perpetual homogenization, The Annals
  of Probability 31~(4) (2003) 1935--1969.

\bibitem{Rosenthal2003}
J.~S. Rosenthal, Asymptotic variance and convergence rates of nearly-periodic
  {M}arkov chain {M}onte {C}arlo algorithms, Journal of the American
  Statistical Association 98~(461) (2003) 169--177.
\newblock \href {http://dx.doi.org/10.1198/016214503388619193}
  {\path{doi:10.1198/016214503388619193}}.

\bibitem{roberts1997weak}
G.~O. Roberts, A.~Gelman, W.~R. Gilks, Weak convergence and optimal scaling of
  random walk {M}etropolis algorithms, The Annals of Applied Probability 7~(1)
  (1997) 110--120.

\bibitem{roberts1998optimal}
G.~O. Roberts, J.~S. Rosenthal, Optimal scaling of discrete approximations to
  langevin diffusions, Journal of the Royal Statistical Society: Series B
  (Statistical Methodology) 60~(1) (1998) 255--268.

\bibitem{beskos2009optimal}
A.~Beskos, G.~Roberts, A.~Stuart, Optimal scalings for local
  {M}etropolis--{H}astings chains on nonproduct targets in high dimensions, The
  Annals of Applied Probability 19~(3) (2009) 863--898.

\bibitem{beskos2015asymptotic}
A.~Beskos, G.~Roberts, A.~Thiery, N.~Pillai, Asymptotic analysis of the
  random-walk {M}etropolis algorithm on ridged densities, Annals of
  Probability.

\bibitem{kuntz2014}
J.~Kuntz, M.~Ottobre, A.~M. Stuart, Diffusion limit for the random walk
  {M}etropolis algorithm out of stationarity (2014).
\newblock \href {http://arxiv.org/abs/arXiv:1405.4896}
  {\path{arXiv:arXiv:1405.4896}}.

\bibitem{kuntz2018non}
J.~Kuntz, M.~Ottobre, A.~M. Stuart, Non-stationary phase of the {MALA}
  algorithm, Stochastics and Partial Differential Equations: Analysis and
  Computations 6~(3) (2018) 446--499.

\bibitem{kuntz2018}
J.~Kuntz, M.~Ottobre, A.~M. Stuart, {Non-stationary phase of the MALA
  algorithm}, Stochastics and Partial Differential Equations: Analysis and
  Computations (2018) 1--54.

\bibitem{ottobre2016function}
M.~Ottobre, N.~S. Pillai, F.~J. Pinski, A.~M. Stuart, et~al., A function space
  {HMC} algorithm with second order {L}angevin diffusion limit, Bernoulli
  22~(1) (2016) 60--106.

\bibitem{bourabee2018}
N.~Bou-Rabee, J.~Sanz-Serna, {Geometric integrators and the {H}amiltonian
  {M}onte {C}arlo method}, Acta Numerica 27 (2018) 113--206.

\bibitem{jourdain2014}
B.~Jourdain, T.~Leli{\`e}vre, B.~Miasojedow, {Optimal scaling for the transient
  phase of {M}etropolis--{H}astings algorithms: The longtime behavior},
  Bernoulli 20 (2014) 1930--1978.

\bibitem{jourdain2015}
B.~Jourdain, T.~Leli{\`e}vre, B.~Miasojedow, {Optimal scaling for the transient
  phase of the random walk {M}etropolis algorithm: {T}he mean-field limit}, The
  Annals of Applied Probability 25 (2015) 2263--2300.

\bibitem{beskos2013optimal}
A.~Beskos, N.~Pillai, G.~Roberts, J.-M. Sanz-Serna, A.~Stuart, Optimal tuning
  of the hybrid monte carlo algorithm, Bernoulli 19~(5A) (2013) 1501--1534.

\bibitem{Potter:201533a}
C.~C. Potter, R.~H. Swendsen, 0.234: {T}he myth of a universal acceptance ratio
  for {M}onte {C}arlo simulations, Physics Procedia 68 (2015) 120--124.

\bibitem{yang2020optimal}
J.~Yang, G.~O. Roberts, J.~S. Rosenthal, Optimal scaling of random-walk
  metropolis algorithms on general target distributions, Stochastic Processes
  and their Applications.

\bibitem{Livingstone1908}
S.~Livingstone, G.~Zanella, {On the robustness of gradient-based MCMC
  algorithms} (2019).
\newblock \href {http://arxiv.org/abs/1908.11812} {\path{arXiv:1908.11812}}.

\bibitem{durmus2018efficient}
A.~Durmus, E.~Moulines, M.~Pereyra, {Efficient Bayesian computation by proximal
  Markov chain Monte Carlo: when Langevin meets Moreau}, SIAM Journal on
  Imaging Sciences 11~(1) (2018) 473--506.

\bibitem{chaudhari2018deep}
P.~Chaudhari, A.~Oberman, S.~Osher, S.~Soatto, G.~Carlier, Deep relaxation:
  partial differential equations for optimizing deep neural networks, Research
  in the Mathematical Sciences 5~(3) (2018) 30.

\bibitem{chaudhari2016entropysgd}
P.~Chaudhari, A.~Choromanska, S.~Soatto, Y.~LeCun, C.~Baldassi, C.~Borgs,
  J.~Chayes, L.~Sagun, R.~Zecchina, Entropy-{SGD}: {B}iasing gradient descent
  into wide valleys (2016).
\newblock \href {http://arxiv.org/abs/1611.01838} {\path{arXiv:1611.01838}}.

\bibitem{bourabee2010pathwise}
N.~Bou‐Rabee, E.~Vanden‐Eijnden, Pathwise accuracy and ergodicity of
  metropolized integrators for {SDE}s, Communications on Pure and Applied
  Mathematics 63~(5) (2010) 655--696.

\bibitem{roberts1996geometric}
G.~Roberts, R.~Tweedie, Geometric convergence and central limit theorems for
  multidimensional {H}astings and {M}etropolis algorithms, Biometrika 83~(1)
  (1996) 95--110.

\bibitem{hutzenthaler2012strong}
M.~Hutzenthaler, A.~Jentzen, P.~E. Kloeden, Strong convergence of an explicit
  numerical method for {SDE}s with non-globally {L}ipschitz continuous
  coefficients, The Annals of Applied Probability 22~(4) (2012) 1611--1641.
\newblock \href {http://dx.doi.org/10.1214/11-AAP803}
  {\path{doi:10.1214/11-AAP803}}.

\bibitem{Zanella:201733a}
G.~Zanella, M.~Bédard, W.~S. Kendall, {A Dirichlet form approach to MCMC
  optimal scaling}, Stochastic Processes and their Applications 127~(12) (2017)
  4053--4082.
\newblock \href {http://dx.doi.org/10.1016/j.spa.2017.03.021}
  {\path{doi:10.1016/j.spa.2017.03.021}}.

\bibitem{vershynin2018high}
R.~Vershynin, High-dimensional probability: {A}n introduction with applications
  in data science, Vol.~47, Cambridge University Press, 2018.

\end{thebibliography}

\clearpage
\appendix
\section{Details of the mean square displacement computation}
\label{s:derviative}

In this section, we give details of the derivation of \eqref{e:Msigopt2}.  Differentiating \eqref{e:Msig} with respect to $\sigma$
\begin{gather}
    M' = \E_\mu\bracket{\int |x-y|^2 \paren{F'(R) \partial_\sigma R g(y;x,\sigma) + F(R) \partial_\sigma g(y;x,\sigma)} dy}\\
    F'(r) = F(r)( 1- F(r))\\
    \partial_\sigma R = \tfrac{\sigma}{4}\paren{|\nabla V(x)|^2 - |\nabla V(y)|^2},\\
    \begin{split}
    \partial_\sigma g &= - \tfrac{n}{\sigma} g\\
    &\quad+\tfrac{1}{\sigma^3} \paren{|y-x + \tfrac{\sigma^2}{2}\nabla V(x)|^2 -2(y-x + \tfrac{\sigma^2}{2}\nabla V(x))\cdot \tfrac{\sigma^2}{2}\nabla V(x) }g\,.
    \end{split}
\end{gather}
Consequently,
\begin{equation}\label{e:dMdsig}
\begin{split}
        M' &= \frac{\sigma}{4}\E\bracket{|x-y|^2F(1-F)\paren{|\nabla V(x)|^2 - |\nabla V(y)|^2}}-\frac{n}{\sigma}M \\
        & \quad + \frac{1}{\sigma^3}\E[|x-y|^2F |y-x + \tfrac{\sigma^2}{2}\nabla V(x)|^2]\\
        &\quad - \frac{2}{\sigma^3}\E[|x-y|^2F(y-x + \tfrac{\sigma^2}{2}\nabla V(x))\cdot \tfrac{\sigma^2}{2}\nabla V(x) ]\,.
\end{split}
\end{equation}
Assuming that the optimal $\sigma$ occurs at a finite value, the first order condition $M'=0$ will hold.  The expression \eqref{e:dMdsig}, at the optimal value, can then be expressed as
\begin{equation}
\label{e:Msigopt1}
    M = \frac{1}{\sigma^2 n}\E\bracket{|x-y|^2F \paren{|x-y|^2 - |\tfrac{\sigma^2}{2}\nabla V(x)|^2F -|\tfrac{\sigma^2}{2}\nabla V(y)|^2 (1-F)  }}\,.
\end{equation}
This calculation makes use of the identity
\begin{equation}
\label{e:MALAident}
\begin{split}
    &|y-x + \tfrac{\sigma^2}{2}\nabla V(x)|^2-2 (y-x + \tfrac{\sigma^2}{2}\nabla V(x))\cdot \tfrac{\sigma^2}{2}\nabla V(x)\\
    & = |y-x|^2 - |\tfrac{\sigma^2}{2}\nabla V(x)|^2\,.
    \end{split}
\end{equation}
Since $M = \E[|x-y|^2 F]$, we can re-write \eqref{e:Msigopt1} to get \eqref{e:Msigopt2}.

Next, note that for the Barker proposal,  $F(r) = (1+e^{-r})^{-1} = 1 - F(-r)$ so
\begin{equation*}
    F(y,x) = F(R(y,x)) = 1 - F(-R(y,x)) = 1 - F(R(x,y)) =  1- F(x,y)\,.
\end{equation*}
Additionally, recall that the Metropolis method satisfies detailed balance.  Therefore,
\begin{equation*}
\begin{split}
   & \E[|x-y|^2 F(x,y)(1-F(x,y))|\nabla V(y)|^2]\\
   &= \int |x-y|^2 F(x,y)(1-F(x,y))|\nabla V(y)|^2 g(y;x,\sigma)dy\mu(dx)\\
   & = \int |x-y|^2 F(y,x)(1-F(x,y))|\nabla V(y)|^2 g(x;y,\sigma)dx\mu(dy)\\
   & = \int |x-y|^2 F(y,x)^2|\nabla V(y)|^2 g(x;y,\sigma)dx\mu(dy)\,,
    \end{split}
\end{equation*}
and we conclude
\begin{equation*}
    \E\bracket{|x-y|^2 F\paren{ F|\nabla V(x)|^2 +(1-F)|\nabla V(y)|^2}} = 2 \E\bracket{|x-y|^2 F^2|\nabla V(x)|^2}\,.
\end{equation*}
Using this result in \eqref{e:Msigopt1} gives us \eqref{e:Msigopt2}.

\section{Details of computations in dimension one}
\label{s:dimensionone}

In this section, we provide a derivation of \eqref{e:MSDone}.  We denote $g_x(x)$  the Gaussian density for $N(0, \eps)$ and $g_y(y|x)$ the Gaussian density for $N((1-\sigma^2\eps^{-1}/2)x, \sigma^2)$.  Then
\begin{equation}
    \label{e:Aone}
    \begin{split}
    A_1 &= \E[1 \wedge e^{R(x,y)}]\\
     &= \int_{R(x,y)>0}  g_x(x) g_y(y|x) dxdy  + \int_{R(x,y)\leq 0}   e^{R(x,y)} g_x(x) g_y(y|x) dxdy\,,
    \end{split}
\end{equation}
and we observe that
\begin{equation}
\label{e:xysymm}
    e^{R(x,y)} g_x(x) g_y(y|x) = g_x(y) g_y(x|x)\,.
\end{equation}
Furthermore, we note that the set $R(x,y)>0$ corresponds to $|x|>|y|$ and $R(x,y)<0$ corresponds to $|x|< |y|$. We can thus use the symmetry $(x,y)\mapsto (y,x)$ to reduce
\eqref{e:Aone} to
\begin{equation}
\label{e:A1split}
\begin{split}
    A_1 &= 2 \int_{R(x,y)>0} g_x(x)g_y(y|x)dxdy\\
    & = \int_{x=-\infty}^\infty \int_{y=-|x|}^{|x|} 2  g_x(x)g_y(y|x)dxdy\\
    &= \underbrace{\int_{x=-\infty}^0 \int_{y=x}^{0}(\ldots) dydx}_{I} + \underbrace{\int_{x=-\infty}^0 \int_{y=0}^{-x}(\ldots) dydx}_{II} \\
    &\quad +  \underbrace{\int_{x=0}^\infty \int_{y=-x}^{0}(\ldots) dydx}_{III}+  \underbrace{\int_{x=0}^\infty \int_{y=0}^{x}(\ldots) dydx}_{IV}\,.
\end{split}
\end{equation}
Since the integrand, $2(y-x)^2 g_x(x) g_y(y|x)$ is invariant to $(x,y)\mapsto (-x,-y)$, $I=IV$ and  $II=III$.  Thus we have
\begin{equation}
A_1 = 4 \int_{x=0}^\infty \int_{y=-x}^x g_x(x)g_y(y|x)dxdy\,.
\end{equation}
Using Mathematica and making the change of variables, $\xi = x/\sqrt{\eps}$ and $\delta = \eps^{-1}\sigma^2/2$,
\begin{equation}
\label{e:A1result}
A_1 = \frac{2}{\pi}\arctan\paren{(2\delta^{-1})^{3/2}}\,.
\end{equation}

Analogously,
\begin{equation}
\label{e:MSDone1}
\begin{split}
    \mathrm{MSD}_1 = \E[(y-x)^2 1 \wedge e^{R(x,y)}] &= \int_{R(x,y)>0} (y-x)^2 g_x(x) g_y(y|x) dxdy \\
    &\quad + \int_{R(x,y)\leq 0} (y-x)^2 e^{R(x,y)} g_x(x) g_y(y|x) dx dy\,.
     \end{split}
\end{equation}
As in the case of the computation of $A_1$, we use the symmetry
\begin{equation}
    (x-y)^2e^{R(x,y)} g_x(x) g_y(y|x) = (x-y)^2g_x(y) g_y(x|x)\,,
\end{equation}
to reduce  \eqref{e:MSDone1} to
\begin{equation}
\label{e:MSDone2}
\begin{split}
    \mathrm{MSD}_1 &=2\int_{R(x,y)>0} (y-x)^2 g_x(x) g_y(y|x) dxdy \\
    &=\int_{x=-\infty}^\infty \int_{y=-|x|}^{|x|} 2(y-x)^2 g_x(x) g_y(y|x) dxdy\,.
\end{split}
\end{equation}
This is the split into four integrals, as in \eqref{e:A1split}.  Since the integrand, $2(y-x)^2 g_x(x) g_y(y|x)$ is invariant to $(x,y)\mapsto (-x,-y)$, $I=IV$ and  $II=III$.  Thus
\begin{equation}
    \mathrm{MSD}_1 = 4 \int_{x=0}^\infty\int_{y=-x}^x  g_x(x) g_y(y|x) dxdy\,.
\end{equation}
Using Mathematica with $\xi = x/\sqrt{\eps}$ and $\delta = \eps^{-1}\sigma^2/2$, $\MSD = \eps m(\delta)$ we have
\begin{equation}\label{e:mfucntion}
    m(\delta) = \frac{2\delta}{\pi(4 + \delta(-2+\delta))}\set{(8+\delta^3)\arctan\paren{\sqrt{\frac{8}{\delta^3}}} - 2 \sqrt{2}\delta^{3/2}}\,.
\end{equation}
This function is plotted in Figure~\ref{fig:mfunction}.

\begin{figure}
    \centering
    \includegraphics[width=8cm]{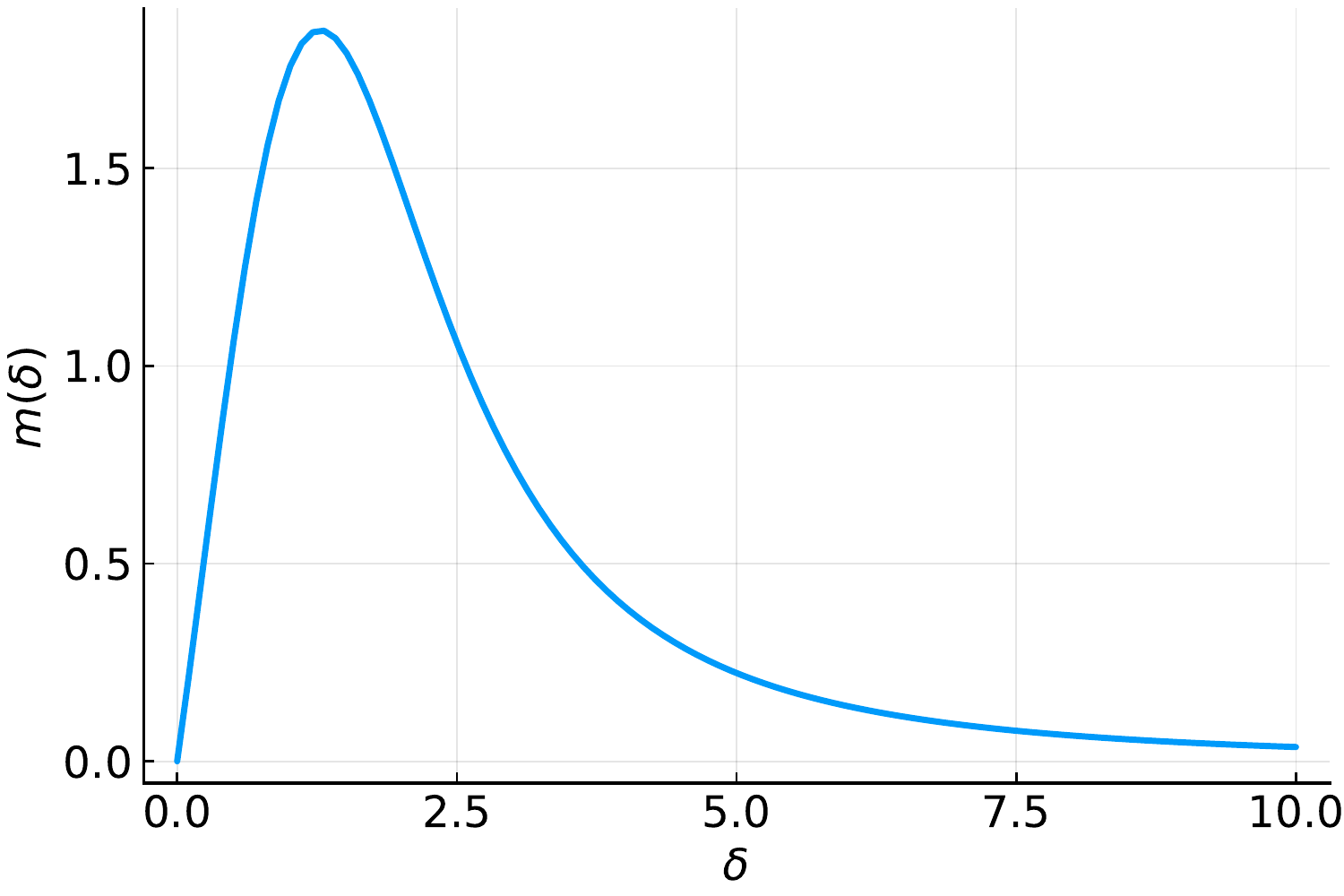}
    \caption{Function \eqref{e:mfucntion}.  Note that it has a single maximum at $\delta_\star = 1.2779727440041808$.  At this value the acceptance rate is $0.70$.}
    \label{fig:mfunction}
\end{figure}

\end{document}